\documentclass[11pt]{amsart}
\usepackage{amsmath, amssymb}
\usepackage{amsfonts}
\usepackage{mathrsfs}
\usepackage[arrow,matrix,curve,cmtip,ps]{xy}

\usepackage{amsthm}
\usepackage{mathtools} 
\usepackage{array}
\usepackage{mathdots}

\usepackage[top= 2.25cm, bottom=2cm, left = 2cm, right= 2cm]{geometry}
\usepackage{hyperref}
\usepackage{setspace}

\allowdisplaybreaks

\makeatletter
\@namedef{subjclassname@2020}{%
  \textup{2020} Mathematics Subject Classification}
\makeatother

\numberwithin{equation}{section}

\newtheorem{theorem}{Theorem}[section]
\newtheorem{lemma}[theorem]{Lemma}
\newtheorem{proposition}[theorem]{Proposition}
\newtheorem{corollary}[theorem]{Corollary}

\newtheorem{conjecture}[theorem]{Conjecture}

\newtheorem*{theorem*}{Theorem}
\theoremstyle{remark}
\newtheorem{remark}[theorem]{Remark}

\numberwithin{equation}{section}

\newcommand{\lif}[1]{\widetilde{#1}}  

\renewcommand{\P}[1]{\Phi(#1)}

\newcommand{\Q}[1]{\Psi(#1)}

\newcommand{\Pbd}[1]{\Phi_{bdd}(#1)}

\newcommand{\Pdt}[1]{\Phi_{2}(#1)}

\newcommand{\cP}[1]{\bar{\Phi}(#1)}

\newcommand{\cQ}[1]{\bar{\Psi}(#1)}

\newcommand{\cPbd}[1]{\bar{\Phi}_{bdd}(#1)}

\newcommand{\cPdt}[1]{\bar{\Phi}_{2}(#1)}

\newcommand{\cQdt}[1]{\bar{\Psi}_{2}(#1)}

\newcommand{\cPsm}[1]{\bar{\Phi}_{sim}(#1)}

\newcommand{\p}{\phi}

\newcommand{\q}{\psi}

\newcommand{\lp}{\tilde{\phi}}

\renewcommand{\lq}{\tilde{\psi}}

\newcommand{\Pkt}[1]{\Pi_{#1}}

\newcommand{\lPkt}[1]{\tilde{\Pi}_{#1}}

\newcommand{\cPkt}[1]{\bar{\Pi}_{#1}}

\newcommand{\clPkt}[1]{\tilde{\bar{\Pi}}_{#1}}

\renewcommand{\r}{\pi}

\newcommand{\lr}{\tilde{\pi}}

\newcommand{\sH}{\bar{\mathcal{H}}}

\newcommand{\cS}[1]{\bar{S}_{#1}}

\renewcommand{\S}[1]{\mathcal{S}_{#1}}

\newcommand{\+}{\oplus}

\renewcommand{\c}{\lambda}

\newcommand{\lG}{\widetilde{G}}

\newcommand{\lM}{\widetilde{M}}

\newcommand{\lP}{\widetilde{P}}

\newcommand{\x}{\omega}

\renewcommand{\L}[1]{{}^L#1}

\newcommand{\D}[1]{\widehat{#1}}

\newcommand{\Gal}[1]{\Gamma_{#1}}

\renewcommand{\a}{\alpha}

\renewcommand{\Im}{\text{Im}\,}

\newcommand{\lZ}{Z_{\widetilde{G}}}

\newcommand{\Z}{Z_{G}}

\newcommand{\Ind}{\text{Ind}}

\newcommand{\Two}{\mathbb{Z}/2\mathbb{Z}}

\newcommand{\C}{\mathbb{C}}

\newcommand{\A}{\mathbb{A}}

\newcommand{\lf}{\tilde{f}}

\newcommand{\Hom}{\text{Hom}}

\newcommand{\e}{\varepsilon}

\newcommand{\Rep}{\text{Rep}}

\newcommand{\Jac}{\text{Jac}}

\newcommand{\ul}{\underline{l}}

\newcommand{\ueta}{\underline{\eta}}

\begin{document}
\title{Arthur packets for quasisplit $GSp(2n)$ and $GO(2n)$ over a $p$-adic field}

\author{Bin Xu}
\thanks{Supported by NSFC No. 20191300979 and Tsinghua University Initiative Scientific Research Program No. 2019Z07L02016}

\address{Yau Mathematical Sciences Center and Department of Mathematics \\  Tsinghua University, Beijing, China}
\email{binxu@tsinghua.edu.cn}


\begin{abstract}
We construct the Arthur packets for symplectic and even orthogonal similitude groups over a $p$-adic field and show that they are stable and satisfy the twisted endoscopic character relations.
\end{abstract}

\maketitle

\tableofcontents

\section{Introduction}
\label{sec: introduction}

Let $F$ be a number field and $G$ a quasisplit symplectic or special even orthogonal group over $F$. Let $\mathbb{A}_{F}$ be the ad\`ele ring of $F$. We fix an automorphism $\theta_{0}$ of $G$ preserving an $F$-splitting. It induces a dual automorphism $\D{\theta}_0$ on the dual group $\D{G}$: when $G$ is symplectic, $\theta_{0}$ is trivial and when $G$ is special even orthogonal, we require $\theta_{0}$ to be the unique nontrivial outer automorphism induced from the conjugation of the full orthogonal group. We say two irreducible admissible representations of $G(\A_{F})$ are $\theta_0$-conjugate if they are $\theta_0$-conjugate at every place. In \cite{Arthur:2013} Arthur proved that the discrete automorphic spectrum of $G(\A_{F})$ can be decomposed as follows
\begin{align}
\label{eq: discrete spectral decomposition classical group}
L^{2}_{disc}(G(F) \backslash G(\mathbb{A}_{F})) = \bigoplus_{[\psi] \in \cQdt{G}} \bigoplus_{[\pi] \in \cPkt{\q}} m_{\psi}(\r) \, \pi
\end{align}
modulo $\theta_0$-conjugation. Here $\cQdt{G}$ is the set of $\D{\theta}_0$-conjugacy classes of the discrete Arthur parameters of $G$, which can be understood in terms of automorphic representations of general linear groups. For $\q \in \cQdt{G}$, Arthur associated a multi-set $\cPkt{\q}$ of isomorphism classes of irreducible admissible representations of $G(\A_{F})$ modulo $\theta_0$-conjugation and $m_{\q}(\r)$ is the multiplicity of $\r$ contributed by $[\q]$ in the discrete spectrum modulo $\theta_0$-conjugation. More precisely, 
\[
\cPkt{\q} = \otimes'_{v} \cPkt{\q_{v}}.
\]
Here $\cPkt{\q_{v}}$ is a finite multi-set of $\theta_0$-conjugacy classes of isomorphism classes of irreducible admissible representations of $G(F_{v})$. It is equipped with a map
\[
\cPkt{\q_{v}} \rightarrow \D{\S{\q_v}}, \quad [\r_v] \mapsto \langle \cdot, \r_v \rangle
\]
where $\S{\q_{v}} = \pi_0(Z_{\D{G}}(\q_{v})/Z(\D{G}_v)^{\Gamma_v})$ and $\D{\S{\q_v}}$ is the set of irreducible characters of $\S{\q_v}$. If $\r_{v}$ is unramified, then $\langle \cdot, \r_v \rangle$ is trivial. For $[\r] \in \cPkt{\q}$, let
\[
\langle s, \r \rangle := \prod_{v} \langle s_v, \r_{v} \rangle
\]
through a natural homomorphism
\[
\S{\q} \rightarrow  \S{\q_v}, \quad s \mapsto s_{v}.
\]
Then
\[
m_{\q}(\r) = m_{\q} |\S{\q}|^{-1} \sum_{s \in \S{\q}} \langle s, \r \rangle \e_{\q}(s),
\]
where $\e_{\q}$ is some linear character of $\S{\q}$ and $m_{\q} = 1$ or $2$.

Let $\lG$ be the group of symplectic or orthogonal similitudes over $F$, whose derived group is $G$. Then $\theta_0$ extends to an automorphism of $\lG$. We have a short exact sequence
\[
\xymatrix{1 \ar[r] & G \ar[r] & \lG \ar[r]^{\lambda_{\lG}} & \mathbb{G}_{m} \ar[r]  &  1,
}
\]
where $\lambda_{\lG}$ is the similitude character. On the dual side, we have
\[
\xymatrix{
1 \ar[r] &  \C^{\times} \ar[r] &  \D{\lG} \ar[r]^{{\bf p}}  &  \D{G} \ar[r] & 1.
}
\]
Let $\tilde{\zeta}$ be a character of $Z_{\lG}(\mathbb{A}_{F})/Z_{\lG}(F)$ and $\zeta$ its restriction to $Z_{G}(\mathbb{A}_{F})$. We will parametrize the $\D{\theta}_0$-conjugacy classes of the discrete Arthur parameters of $\lG$ associated with central character $\tilde{\zeta}$ by pairs $([\q], \x)$, where $[\q] \in \cQdt{G, \zeta}$ and $\x \in Y := \Hom(\lG(\A_{F}) / \lG(F) \lZ(\A_{F})G(\A_{F}), \C^{\times})$. There is a commutative diagram with exact rows
\begin{align*}
\xymatrix{1 \ar[r] &  \S{\lq} \ar[r]^{\iota} \ar[d] & \S{\q} \ar[r]^{\a \quad \quad \quad \quad \quad \quad} \ar[d] & \Hom(\lG(\A_{F})/ \lG(F)G(\A_{F}), \C^{\times}) \ar[d] \\
1 \ar[r] &  \S{\lq_v} \ar[r]^{\iota_v} & \S{\q_v} \ar[r]^{\a_v \quad \quad \quad \quad \quad} & \Hom(\lG(F_{v})/G(F_{v}), \C^{\times}).           
}               
\end{align*}
(cf. \cite[2.18]{Xu:Lpacket}). Then Arthur's conjectural description of the discrete automorphic spectrum of $\lG(\A_{F})$ (cf. \eqref{eq: discrete spectral decomposition classical group}) can be reformulated as follows.

\begin{conjecture}
\label{conj: global}
For any $[\q] \in \cQdt{G}$, there exists a multi-set $\cPkt{\lq}$ of irreducible admissible representations of $\lG(\A_{F})$ modulo $\theta_0$-conjugation, which is unique up to twist by $Y$, such that
\[
L^{2}_{disc}(\lG(F) \backslash \lG(\mathbb{A}_{F}), \tilde{\zeta}) = \bigoplus_{[\q] \in \cQdt{G, \zeta}} \bigoplus_{\omega \in Y/\a(\S{\q})} \bigoplus_{[\lr] \in \cPkt{\lq} \otimes \omega} m_{\lq}(\lr) \, \lr.
\]
Here
\(
\cPkt{\lq} = \otimes'_{v} \cPkt{\lq_v},
\)
and the correspondence
\[
\cPkt{\lq_v} \longrightarrow \cPkt{\q_{v}}, \quad [\lr_v] \mapsto [\r_v]
\]
by requiring $\r_{v} \subseteq \lr_{v}|_{G(F_v)}$ fits into a commutative diagram
\[
\xymatrix{\cPkt{\lq_v} \ar[r]  \ar[d] & \D{\S{\lq_v}}, & [\lr_v] \mapsto  \langle s,  \lr_v \rangle \\
\cPkt{\q_v} \ar[r] & \D{\S{\q_v}}. \ar[u] &
}
\]
At last,
\[
m_{\lq}(\lr) = m_{\q} |\S{\lq}|^{-1} \sum_{s \in \S{\lq}} \langle s, \lr \rangle \e_{\lq}(s)
\]
where $\e_{\lq} = \e_{\q}|_{\S{\lq}}$.

\end{conjecture}

This conjecture has been proved in the tempered case (cf. \cite[Theorem 1.1]{Xu:Lpacket}). The main goal of this paper is to construct $\cPkt{\lq_v}$ at the nonarchimedean places, as the first step to deal with the nontempered case.

Let $F$ be a $p$-adic field, $\Gamma_{F}$ the absolute Galois group of $F$ and $W_{F}$ the Weil group of $F$. Let $\L{G} := \D{G} \rtimes W_{F}$ be the Langlands dual group of $G$. An Arthur parameter of $G$ is a $\D{G}$-conjugacy class of admissible homomorphisms 
\[
\q: W_{F} \times SL(2, \C) \times SL(2, \C) \longrightarrow \L{G}
\]
such that $\q(W_{F})$ is bounded. We denote the set of $\D{\theta}_0$-conjugacy classes of Arthur parameters of $G$ (resp. $\lG$) by $\cQ{G}$ (resp. $\cQ{\lG}$). For $[\q] \in \cQ{G}$, M{\oe}glin \cite{Moeglin1:2011} showed that $\cPkt{\q}$ is multiplicity free. Let $\tilde{\zeta}$ be a character of $\lZ(F)$ extending the central character of $\cPkt{\q}$ and $\clPkt{\q, \tilde{\zeta}}$ the subset of $\theta_0$-conjugacy classes of isomorphism classes of irreducible admissible representations of $\lG(F)$ with central character $\tilde{\zeta}$, such that their restrictions to $G(F)$ have irreducible constituents in $\cPkt{\q}$. Suppose $\q = {\bold p} \circ \lq$ for $[\lq] \in \cQ{\lG}$, we can define a map $\clPkt{\q, \tilde{\zeta}} \rightarrow \D{\S{\lq}}$ uniquely determined through the following diagram 
\[
\xymatrix{\clPkt{\q, \tilde{\zeta}} \ar[r] \ar[d] & \D{\S{\lq}} \\
\cPkt{\q} \ar[r] & \D{\S{\q}} \ar[u]
}
\]
(cf. Corollary~\ref{cor: Whittaker pairing}). Let $\sH(\lG(F))$ be the space $\theta_0$-invariant smooth compactly supported functions on $\lG(F)$. Let 
\(
X := \Hom(\lG(F)/G(F) \lZ(F) , \C^{\times}).
\)
For $\q = \p$ trivial on the second $SL(2, \mathbb{C})$, we have shown in \cite[Theorem 4.6]{Xu:2018} that there exists a subset $\cPkt{\lp}$ of $\clPkt{\p, \tilde{\zeta}}$ unique up to twists by $X$ such that
\begin{enumerate}

\item
\[
\bigoplus_{[\lr] \in \cPkt{\lp}} [\lr|_{G(F)}] = \bigoplus_{[\r] \in \cPkt{\p}} [\r],
\]

\item
\[
\lf(\lp) := \sum_{[\lr] \in \cPkt{\lp}} \lf_{\lG}(\lr), \quad \lf \in \sH(\lG(F))
\]
is stable, where $\lf_{\lG}(\lr)$ is the Harish-Chandra character of $\lr$.

\item for any semisimple $s \in \cS{\lp}$ and $(G' ,\p') \rightarrow (\p, s)$ (see \eqref{eq: endoscopy} for the notation), we have
\[
\lf'(\lp') = \sum_{[\lr] \in \cPkt{\lp}} \langle s, \lr \rangle \lf_{\lG}(\lr), \quad \lf \in \sH(\lG(F)),
\]
where $\lf'$ is the Langlands-Shelstad transfer of $\lf$ and $\lf'(\lp')$ is defined as in (2) with respect to some choice of $\cPkt{\lp'}$.
\end{enumerate}
Our main result generalizes this to $\cQ{G}$.

\begin{theorem}
For $[\q] \in \cQ{G}$, there exists a subset $\cPkt{\lq}$ of $\clPkt{\q, \tilde{\zeta}}$ such that
\begin{enumerate}

\item
\[
\bigoplus_{[\lr] \in \cPkt{\lq}} [\lr|_{G(F)}] = \bigoplus_{[\r] \in \cPkt{\q}} [\r],
\]

\item
\[
\lf(\lq) := \sum_{[\lr] \in \cPkt{\lq}} \langle s_{\lq}, \lr \rangle \,  \lf_{\lG}(\lr), \quad \lf \in \sH(\lG(F))
\]
is stable, where $s_{\lq}$ is the image of $\lq(1, 1, -1)$ in $\S{\lq}$,

\item for any semisimple $s \in \cS{\lq}$ and $(G', \q') \rightarrow (\q, s)$ (see \eqref{eq: endoscopy} for the notation), we have
\[
\lf'(\lq') = \sum_{[\lr] \in \cPkt{\lq}}  \langle ss_{\lq}, \lr \rangle \,  \lf_{\lG}(\lr), \quad \lf \in \sH(\lG(F)).
\]
where $\lf'$ is the Langlands-Shelstad transfer of $\lf$ and $\lf'(\lq')$ is defined as in (2) with respect to some choice of $\cPkt{\lq'}$.
\end{enumerate}
\end{theorem}

This result is contained in Theorem~\ref{thm: Apacket}, where we have also shown the twisted character relations under which $\cPkt{\lq}$ would be uniquely determined up to twist by $X$. One of our principles for constructing $\cPkt{\lq}$, inspired by Conjecture~\ref{conj: ABV}, is that its elements should have the same infinitesimal character. We follow the notion of infinitesimal character introduced by Vogan in \cite{Vogan:1993}, which depends on the local Langlands correspondence. In Section~\ref{sec: infinitesimal character}, we review this notion and study how the lift of an infinitesimal character would determine the lifts of Langlands parameters of elements in $\cPkt{\q}$, through the geometry of Langlands parameters (cf. \cite{CFMMX}). In Section~\ref{sec: LLC}, we construct a candidate for the local Langlands correspondence of $\lG$, through which we could define the infinitesimal characters for representations of $\lG(F)$ modulo $\D{\theta}_0$-conjugation. In Section~\ref{sec: main results}, we formulate the main result (cf. Theorem~\ref{thm: Apacket}) and treat the special case when the lift of infinitesimal character determine the lift of Langlands parameters of elements in $\cPkt{\q}$. The key input of the proof is that the parabolic induction preserves the infinitesimal character (cf. Proposition~\ref{prop: induction preserves infinitesimal character}). In Section~\ref{sec: construction}, we treat the general case following M{\oe}glin's strategy \cite{Moeglin:2009} for constructing $\cPkt{\q}$. It reduces to the case of discrete $L$-packets, which has already been treated in our earlier work \cite{Xu:2018}. This is the most difficult part of the paper. In Section~\ref{sec: uniqueness}, we address the uniqueness part of the main result and complete its proof. In the last section, we prove that $\cPkt{\lq}$ contains an $L$-packet $\cPkt{\p_{\lq}}$, where $\p_{\lq}$ is the associated $L$-parameter of $\lq$. Based on this result, we extend the local Langlands correspondence constructed in Section~\ref{sec: LLC} to a correspondence between Arthur parameters and Arthur packets.

\section{Index of notation}

We list some of the notations with a brief description and references to precise definitions.

\begin{spacing}{1.5}

\begin{itemize}

\item $G$ is $Sp(2n), SO(2n)$ and  $\lG$ is $GSp(2n), GSO(2n)$ in most cases (Section~\ref{sec: infinitesimal character}, ~\ref{sec: Jac}).

\item $G^{\Sigma_0}$ is $Sp(2n), O(2n)$ and $\lG^{\Sigma_0}$ is $GSp(2n), GO(2n))$ in most cases (Section~\ref{sec: infinitesimal character}, ~\ref{sec: Jac}).

\item $\lambda_{\lG}$ is the similitude character of $\lG$ (Section~\ref{sec: infinitesimal character}).

\item $\lambda, \lambda_{\q}, \tilde{\lambda}, \lambda_{\lM}, \lambda_{\widetilde{L}}$ are infinitesimal characters (Section~\ref{sec: infinitesimal character}).

\item $\bar{\mathcal{H}}(G(F)), \bar{\mathcal{H}}(G(F), \chi), \bar{\mathcal{H}}(\lG(F)), \bar{\mathcal{H}}(\lG(F), \tilde{\chi})$ are spaces of functions (Section~\ref{sec: introduction}, ~\ref{sec: main results})

\item $\Pi(G^{\Sigma_0}(F)), \bar{\Pi}(G(F)), \Pi(\lG^{\Sigma_0}(F)), \bar{\Pi}(\lG(F))$ are sets of isomorphism classes of irreducible representations (Section~\ref{sec: infinitesimal character}).

\item $\P{G}, \cP{G}, \cP{\lG}$ are sets of Langlands parameters (Section~\ref{sec: infinitesimal character}).

\item $\p, \p_a, \p_{\r}, \p_{\q}$ are Langlands parameters (Section~\ref{sec: infinitesimal character}, ~\ref{sec: LLC}).

\item $\q, \q_d, \q_p, \q_{np}$ are Arthur parameters (Section~\ref{sec: infinitesimal character}, ~\ref{subsec: Arthur parameters}).

\item $\cS{\q}, \cS{\lq}, \S{\q}, \S{\lq}, \S{\q}^{\Sigma_0}, \S{\q}^{\theta}$ are variants of centralizer groups of parameters (Section~\ref{sec: infinitesimal character}).

\item $\D{\S{\q}}, \D{\S{\q}^{\Sigma_0}}, \D{\S{\lq}}, \D{\S{\lq}^{\Sigma_0}}$ are characters of centralizer groups, their elements are denoted by $\bar{\e}, \e, \tilde{\bar{\e}}, \tilde{\e}$ respectively (Section~\ref{sec: main results}).

\item $\r, \r^{\Sigma_0}, \lr, \lr^{\Sigma_0}$, $\lr_{?}(\lq, \bar{\e}), \lr^{\Sigma_0}_{?}(\lq, \e)$ ($? = W, M, MW$) are representations of $G, G^{\Sigma_0}, \lG, \lG^{\Sigma_0}$ respectively (Section~\ref{sec: main results}, ~\ref{subsec: elementary case}, ~\ref{subsubsec: construction}).

\item $\lr_1 \tilde{\otimes} \lr_2$ is the restriction of $\lr_1 \otimes \lr_2$ to $\{(g_1, g_2) \in \lG_1 \times \lG_2 \, | \, \lambda_{\lG_1}(g_1) = \lambda_{\lG_2}(g_2)\}$ (Section~\ref{sec: infinitesimal character}).

\item $\cPkt{\q}, \cPkt{\q}^{\Sigma_0}, \cPkt{\lq}, \Pkt{\lq}^{\Sigma_0}, \clPkt{\q,\tilde{\zeta}}, \lPkt{\q, \tilde{\zeta}}^{\Sigma_0}$ are packets of representations of $G, G^{\Sigma_0}, \lG, \lG^{\Sigma_0}$ respectively (Section~\ref{sec: main results}).

\item $X = \Hom(\lG(F)/G(F) \lZ(F) , \C^{\times})$ (Section~\ref{sec: introduction}).

\item $X(\lp), X^{\Sigma_0}(\lp), X(\lr), X(\lr^{\Sigma_0})$ are stabilizers in $\bar{H}^1(W_F, \mathbb{C}^{\times}) \cong {\rm Hom}(\lG(F)/G(F), \mathbb{C}^{\times})$ (Section~\ref{sec: infinitesimal character}, ~\ref{sec: LLC}).

\item $\alpha: \S{\q}^{\Sigma_{0}} \rightarrow {\rm Hom}(\lG(F)/G(F), \mathbb{C}^{\times})$ (Lemma~\ref{lemma: twist invariant}, \eqref{eq: combinatorial description}).

\item ${\rm Jac}^{\rho}_{x}, \overline{{\rm Jac}}^{\rho}_{x}, {\rm Jac}^{\rho}_{X^{T}_{(\rho, A, B, \zeta)}}, {\rm Jac}_{(\rho, A+T, B+T, \zeta) \mapsto (\rho, A, B, \zeta)}$ are variants of Jacquet modules (Section~\ref{sec: Jac}).

\item ${\rm Rep}(\lG^{\Sigma_0}(F))$, $\overline{{\rm Rep}}(\lG(F))$, ${\rm Rep}(\lG^{\Sigma_0}(F), \omega)$, $\overline{{\rm Rep}}(\lG(F), \theta, \omega)$ are category of representations, we add $``K"$ in front to denote the corresponding Grothendieck groups, and $R(\lG^{\Sigma_0}, \omega), \bar{R}(\lG, \theta, \omega)$ are spaces of twisted characters (Section~\ref{sec: Jac}).

\item $Jord(\q)$, $Jord_{\rho}(\q)$, $Jord(\q)_p$, $Jord_{\pm}$ are sets of Jordan blocks,  $J(\q)$ is the set of $\rho$ appearing in $Jord(\q)$ (Section~\ref{subsec: Arthur parameters}, ~\ref{subsubsec: construction}).

\item $f_{G}(\r), \lf_{\lG}(\lr),  f_{G^{\theta}}(\r), \lf_{\lG^{\theta}, W}(\lr, \omega), f_{W}(\q), \lf_{W}(\lq)$ are (twisted) Harish-Chandra characters (Section~\ref{sec: main results}).

\item $\langle \cdot, \cdot \rangle_{?}$ $(? = W, M, MW)$ are pairings between the centralizer groups (e.g. $\S{\q}$) and the packets (e.g. $\cPkt{\q}$) (Section~\ref{sec: main results}).

\item $X^{T}_{(\rho, A, B, \zeta)}$ is a matrix of half-integral entries, and $|X^{T}_{(\rho, A, B, \zeta)}|$ is the sum of its entries (Section~\ref{sec: Jac}).

\item $\widetilde{X}_{C}, \widetilde{X}_{\eta}, \widetilde{X}_{\eta, C}, \widetilde{X}_{\eta, \eta''}$ are representations of $\lG^{\Sigma_0}$ (Section~\ref{subsubsec: construction}).

\item $\tilde{\zeta}, \chi, \tilde{\chi}, \chi_{\tilde{\rho}}, \chi_{\lp_a}, \eta_{\rho}$ are central characters (Section~\ref{sec: main results}, ~\ref{subsubsec: construction}, Lemma~\ref{lemma: Jacquet module}).

\item $\zeta, \zeta_{a,b}, \zeta_{\rho}$, $\eta$, $\eta_{0}, , \eta_{T,0}$ are signs (Section~\ref{subsec: Arthur parameters}, ~\ref{subsubsec: construction}, Theorem~\ref{thm: uniqueness}).

\item $\iota, \iota_{G}, \lif{\iota}, \iota_{\lif{H}}, \iota^{\lif{H}'}_{\lif{H}}$ are (twisted) endoscopic embeddings (Section~\ref{sec: infinitesimal character}, ~\ref{subsec: Arthur parameters}, ~\ref{sec: Arthur LLC}).

\item ${\rm Tran}(\cdot)$ is the transfer of packets defined through the twisted character relation (Lemma~\ref{lemma: cuspidal support}, Proposition~\ref{prop: compatible with endoscopic transfer}).

\end{itemize}
\end{spacing}

The following rules have been adapted to the notations. Overhead $``\sim"$ refer to objects defined for $\lG$ (e.g. $\lr, \lp$). Superscript $``\Sigma_0"$ refer to objects defined for $G^{\Sigma_0}$ (e.g. $\S{\q}^{\Sigma_0}$). Subscript $``\lambda"$ means fixing an infinitesimal character (e.g. $\Pi(G(F))_{\lambda}, \P{G}_{\lambda}$). Overhead $``-"$ means modulo conjugation by $\Sigma_0$ or $\D{\Sigma}_0$ (e.g. $\bar{\Pi}(G(F)), \cP{G}$). We add $[\,]$ to mean equivalence class under conjugation by $\Sigma_0$ or $\D{\Sigma}_0$ (e.g. $[\p], [\r]$). We add subscripts $? = W, M, MW$ to indicate the choice of normalization: $``W"$ for Whittaker, $``M"$ for M{\oe}glin, $``MW"$ for M{\oe}glin-Waldspurger (e.g. $\langle \cdot, \cdot \rangle_{?}$).


\section{Infinitesimal character}
\label{sec: infinitesimal character}


Let $G$ be a quasisplit connected reductive algebraic group over a $p$-adic field $F$. Let $\Pi(G(F))$ be the set of isomorphism classes of irreducible admissible representations of $G(F)$ and $\Phi(G)$ the set of Langlands parameters of $G$, i.e., $\D{G}$-conjugacy classes of admissible homomorphisms from $W_{F} \times SL(2, \mathbb{C})$ to $\L{G}$. We assume the local Langlands correspondence for $G$, i.e., there is a surjection
\[
\Pi(G(F)) \rightarrow \Phi(G), \quad \pi \mapsto \phi_{\pi}
\]
whose fibers are called $L$-packets. Then the infinitesimal character of $\r \in \Pi(G(F))$ is defined to be that of $\p_{\r}$, namely the $\D{G}$-conjugacy class of
\[
\lambda: W_{F} \rightarrow \L{G}, \quad w \mapsto \phi_{\pi} \Big(w, \begin{pmatrix} |w|^{1/2} & 0 \\ 0 & |w|^{-1/2} \end{pmatrix} \Big).
\]
Fix an infinitesimal character $\lambda$. Let $Z_{\D{G}}(\lambda)$ be the centralizer of $\lambda$ in $\D{G}$. The set $\Phi(G)_{\lambda}$ of Langlands parameters with infinitesimal character $\lambda$ is in bijection with $Z_{\D{G}}(\lambda)$-orbits in 
\[
V_{\lambda} := \Big\{x \in \widehat{\mathfrak{g}}^{\lambda(I_{F})} | \text{Ad}(\lambda(\text{Fr}))x = q x \Big\}
\]
through the map 
\[
\p \mapsto d\phi|_{SL(2, \mathbb{C})} \Big( \begin{pmatrix} 0 & 1 \\ 0 & 0 \end{pmatrix} \Big),
\]
where $\widehat{\mathfrak{g}}$ is the Lie algebra of $\D{G}$ and $q$ is the order of the residue field of $F$ (cf. \cite[Proposition 4.2.2]{CFMMX}). The number of orbits is finite and there is a unique open orbit (cf. \cite[Proposition 5.6.1]{CFMMX}). For $\p \in \Phi(G)_{\lambda}$, let us denote the corresponding orbit by $C_{\p}$. We can also view $\lambda$ as a Langlands parameter, which is trivial on $SL(2, \C)$. It corresponds to the origin in $V_{\lambda}$. For $\p \in \P{G}$, we say it is bounded (resp. discrete) if $\p(W_F)$ is bounded (resp. $Z_{\D{G}}(\p)$ is finite). Denote the set of bounded (resp. discrete) parameters by $\Pbd{G}$ (resp. $\Pdt{G}$). We have $\Pdt{G} \subseteq \Pbd{G}$.

Let $\Q{G}$ be the set of Arthur parameters of $G$, i.e., $\D{G}$-conjugacy classes of admissible homomorphisms from $W_{F} \times SL(2, \C) \times SL(2, \mathbb{C})$ to $\L{G}$ such that $\q(W_{F})$ is bounded. Then we can view $\Pbd{G} \subseteq \Q{G}$. There is an inclusion   
\[
\Q{G} \hookrightarrow \P{G}, \quad \q \mapsto \p_{\q}(w, x) := \q \Big( w, x, \begin{pmatrix} |w|^{\frac{1}{2}} & \\ & |w|^{-\frac{1}{2}}  \end{pmatrix}\Big)
\]
(cf. \cite[Lemma 3.6.1]{CFMMX}). We define the infinitesimal character of $\q$ to be that of $\p_{\q}$, denoted by $\lambda_{\q}$. Let $\Q{G}_{\lambda}$ be the subset of $\Q{G}$ with infinitesimal character $\lambda$. For $\q \in \Q{G}$, let $\q_{d} := \q \circ \Delta \in \Pbd{G}$, where
\[
\Delta: W_{F} \times SL(2, \C) \rightarrow W_{F} \times SL(2, \C) \times SL(2, \C), \quad (w, x) \mapsto (w, x, x).
\]
We can also view $\q_{d}$ as an Arthur parameter, which is trivial on the second $SL(2, \C)$. Then $\q_{d}$ corresponds to the unique open orbit in $V_{\lambda}$ (cf. \cite[Proposition 6.1.1]{CFMMX}). Let $\Pkt{}(G(F))_{\lambda}$ be the subset of $\Pkt{}(G(F))$ with infinitesimal character $\lambda$. 
The following conjecture is suggested by \cite{ABV:1992} \cite{Vogan:1993} \cite{CFMMX}.

\begin{conjecture}
\label{conj: ABV}
For $\q \in \Q{G}$, $\Pkt{\q} \subseteq \Big\{\r \in \Pi(G(F))_{\lambda} \, : \, \overline{C}_{\p_{\r}} \supseteq C_{\p_{\q}} \Big\}$.
\end{conjecture}

Suppose $G$ is a quasisplit symplectic or special even orthogonal group over $F$ and $\lG$ is the corresponding similitude group. We have
\(
\lG \cong (\mathbb{G}_{m} \times G) / (\Two),
\) 
where $\Two$ is embedded diagonally into the centre of each factor. The similitude character $\c_{\lG}$ is square on $\mathbb{G}_{m}$ and trivial on the other factor. We fix an automorphism $\theta_{0}$ of $G$ preserving an $F$-splitting: when $G$ is symplectic, we require $\theta_{0}$ to be trivial and when $G$ is special even orthogonal, we require $\theta_{0}$ to be the unique nontrivial outer automorphism induced from the conjugation of the full orthogonal group. Clearly, $\theta_{0}^{2} = 1$, $\theta_{0}$ extends to $\lG$ by acting trivially on $\lZ$, and $\c_{\lG}$ is $\theta_{0}$-invariant. It induces a dual automorphism $\D{\theta}_0$. Let $\Sigma_{0} = \langle \theta_{0} \rangle$ and $\D{\Sigma}_0 = \langle \D{\theta}_0 \rangle$. More generally, we can consider 
\begin{align}
\label{eq: product}
G = G_{1} \times \cdots \times G_{q}
\end{align}
where $G_i$ is a quasisplit symplectic or special even orthogonal group. We define
\[
\lG = (\mathbb{G}_{m} \times G_{1} \times G_{2} \times \cdots \times G_{q}) / (\Two),
\]
where $\Two$ is embedded diagonally into the center of each factor. We also define a character $\c_{\lG}$ of $\lG$, which is square on $\mathbb{G}_{m}$ and trivial on the other factors. Then there is an exact sequence
\[
\xymatrix{1 \ar[r] & G \ar[r] & \lG \ar[r]^{\lambda_{\lG}} & \mathbb{G}_{m} \ar[r]  &  1.
}
\]
On the dual side, we have
\[
\xymatrix{
1 \ar[r] &  \C^{\times} \ar[r] &  \D{\lG} \ar[r]^{{\bf p}}  &  \D{G} \ar[r] & 1.
}
\]
We can also view $\lG$ as a subgroup of $\lG_1 \times \cdots \times \lG_q$ by 
\[
\lG \cong \Big\{ (g_i) \in \prod_{i} \lG_{i} \, | \, \lambda_{\lG_{1}}(g_{1}) = \cdots = \lambda_{\lG_{q}}(g_{q}) \Big\}.
\] 

We define a group of automorphisms of $G$ by taking the product of $\Sigma_{0}$ on each factor, and we denote this group again by $\Sigma_{0}$. We can extend $\Sigma_0$ to $\lG$ by the trivial action on $Z_{\lG}$. It induces a group $\D{\Sigma}_0$ of automorphisms of $\D{\lG}$. Let 
\begin{align*}
& G^{\Sigma_0} = G \rtimes \Sigma_0, \quad \lG^{\Sigma_0} = \lG \rtimes \Sigma_0, \\
& \D{G}^{\Sigma_0} = \D{G} \rtimes \D{\Sigma}_0, \quad \D{\lG}^{\Sigma_0} = \D{\lG} \rtimes \D{\Sigma}_0.
\end{align*}

For admissible representations $\lr_i$ of $\lG_i(F)$, we define the restriction of $\otimes_{i} \, \lr_{i}$ to $\lG(F)$ by $\tilde{\otimes}_i \, \lr_i$. Let $\cPkt{}(G(F))$ be the set of $\Sigma_{0}$-orbits in $\Pkt{}(G(F))$ and $\cP{G}$ the set of $\D{\Sigma}_{0}$-orbits in $\P{G}$. We will also consider $\Pkt{}(G^{\Sigma_0}(F))$ the set of isomorphism classes of irreducible admissible representation of $G^{\Sigma_0}(F)$. There is a natural surjection from $\Pkt{}(G^{\Sigma_0}(F))$ to $\cPkt{}(G(F))$. 

For $G$ in \eqref{eq: product}, the local Langlands correspondence is known modulo $\Sigma_0$-conjugation (cf. \cite[Theorem 1.5.1]{Arthur:2013}), namely we have a surjection with finite fibers
\[
\cPkt{}(G(F)) \rightarrow \cP{G}, \quad [\r] \mapsto [\p_{\r}].
\]
So we can define the $\D{\Sigma}_0$-conjugacy classes of infinitesimal characters for $\cPkt{}(G(F))$, which will be called $\Sigma_0$-infinitesimal characters. For $\r^{\Sigma_0} \in \Pkt{}(G^{\Sigma_0}(F))$, we define its $\Sigma_0$-infinitesimal character to that of its restriction to $G(F)$. We fix an infinitesimal character $\tilde{\lambda}$ of $\lG$ and let $\lambda = {\bold p} \circ \tilde{\lambda}$. There is a surjection 
\(
\P{\lG}_{\tilde{\lambda}} \rightarrow \P{G}_{\lambda}
\) 
by the composition with ${\bf p}$ (cf. \cite[Theorem 2.9]{Xu:2018}). This can also be seen from the natural isomorphism $V_{\tilde{\lambda}} \xrightarrow{\cong} V_{\lambda}$ under ${\bf p}$, which is compatible with the actions by $Z_{\D{\lG}}(\tilde{\lambda}) \rightarrow Z_{\D{G}}(\lambda)$. Let $\cP{G}_{\lambda}$ (resp. $\cP{\lG}_{\tilde{\lambda}}$) be the subset of $\cP{G}$ (resp. $\cP{\lG}$) with $\Sigma_0$-infinitesimal character $[\lambda]$ (resp. $[\tilde{\lambda}]$). Then we have the following commutative diagram.
\[
\xymatrix{& \P{G}_{\lambda} \ar[rr] \ar[d]^{\simeq} && \cP{G}_{\lambda} \ar[d]^{\simeq} \\
& V_{\lambda}/Z_{\D{G}}(\lambda) \ar[rr] && V_{\lambda}/Z_{\D{G}^{\Sigma_0}}(\lambda) \\
\P{\lG}_{\tilde{\lambda}} \ar[rr]  \ar[ruu] \ar[d]^{\simeq} && \cP{\lG}_{\tilde{\lambda}} \ar[uur] \ar[d]^{\simeq} & \\
V_{\tilde{\lambda}}/Z_{\D{\lG}}(\tilde{\lambda}) \ar[rr] \ar[uur] && V_{\tilde{\lambda}}/Z_{\D{\lG}^{\Sigma_0}}(\tilde{\lambda}) \ar[uur] & .
}
\]
Let $\bar{H}^{1}(W_{F}, \mathbb{C}^{\times}) := {\rm Im}\{H^{1}(W_{F}, \mathbb{C}^{\times}) \rightarrow H^{1}(W_{F}, Z(\D{\lG}))\}$. For $\lp \in \P{\lG}_{\tilde{\lambda}}$, let
\[
X(\lp) := \Big\{\omega \in \bar{H}^{1}(W_{F}, \mathbb{C}^{\times}) \, : \, \lp \otimes \omega = \lp \Big\}, \quad X^{\Sigma_0}(\lp) := \Big\{\omega \in \bar{H}^{1}(W_{F}, \mathbb{C}^{\times}) \, : \, [\lp \otimes \omega] = [\lp] \Big\},
\]
which only depend on the image of $\lp$ in $\cP{G}_{\lambda}$. Then
\[
\P{\lG}_{\tilde{\lambda}} / X(\tilde{\lambda}) \cong \P{G}_{\lambda}, \quad \cP{\lG}_{\tilde{\lambda}} / X^{\Sigma_0}(\tilde{\lambda}) \cong \cP{G}_{\lambda}.
\]
So the fiber of
\[
\P{\lG}_{\tilde{\lambda}} \rightarrow \P{G}_{\lambda} \quad \text{ \Big(resp. $\cP{\lG}_{\tilde{\lambda}} \rightarrow \cP{G}_{\lambda}$\Big)}
\]
over $\p$ (resp. $[\p]$) is in bijection with
\[
X(\tilde{\lambda}) / X(\lp) \quad \text{ \Big(resp. $X^{\Sigma_0}(\tilde{\lambda}) / X^{\Sigma_0}(\lp)$\Big). }
\]

Let us identify $V_{\lambda}$ with $V_{\tilde{\lambda}}$. Suppose $\lp_o \in \P{\lG}_{\tilde{\lambda}}$ corresponds to the unique open $Z_{\D{\lG}}(\tilde{\lambda})$-orbit in $V_{\tilde{\lambda}}$, then this orbit must also be the unique open $Z_{\D{G}}(\lambda)$-orbit in $V_{\lambda}$. Let $\p_{o} \in \P{G}_{\lambda}$ be the corresponding parameter. So $\tilde{\lambda}$ determines a unique lift $\lp_o$ of $\p_{o}$, hence
\(
X(\tilde{\lambda}) = X(\lp_o).
\)
By the uniqueness, we can also conclude that this orbit is invariant under $Z_{\D{G}^{\Sigma_0}}(\lambda)$. So $[\tilde{\lambda}]$ determines a unique lift $[\lp_o]$ of $[\p_{o}]$, hence $X^{\Sigma_0}(\tilde{\lambda}) = X^{\Sigma_0}(\lp_o)$. 

By the same argument as for Langlands parameters, one can show that there is a surjection $\Q{\lG}_{\tilde{\lambda}} \rightarrow \Q{G}_{\lambda}$ through the composition with ${\bf p}$. For $\lq \in \Q{\lG}_{\tilde{\lambda}}$, let
\[
X(\lq) := \Big\{\omega \in \bar{H}^{1}(W_{F}, \mathbb{C}^{\times}) \, : \, \lq \otimes \omega = \lq \Big\}, \quad X^{\Sigma_0}(\lq) := \Big\{\omega \in \bar{H}^{1}(W_{F}, \mathbb{C}^{\times}) \, : \, [\lq \otimes \omega] = [\lq] \Big\},
\]
which only depend on the image of $\lq$ in $\cQ{G}_{\lambda}$. The fiber of
\[
\Q{\lG}_{\tilde{\lambda}} \rightarrow \Q{G}_{\lambda} \quad \text{ \Big(resp. $\cQ{\lG}_{\tilde{\lambda}} \rightarrow \cQ{G}_{\lambda}$\Big)}
\]
over $\q$ (resp. $[\q]$) is in bijection with
\[
X(\tilde{\lambda}) / X(\lq) \quad \text{ \Big(resp. $X^{\Sigma_0}(\tilde{\lambda}) / X^{\Sigma_0}(\lq)$\Big). }
\]
From the inclusion $\Q{\lG} \hookrightarrow \P{\lG}$, we have $X(\lq) = X(\lp_{\q})$ and $X^{\Sigma_0}(\lq) = X^{\Sigma_0}(\lp_{\q})$. 

For $\lq \in \Q{\lG}$ and $\q = {\bf p} \circ \lq \in \Q{G}$, we define 
\[
S^{\Sigma_0}_{\q} = Z_{\D{G}^{\Sigma_0}}(\q), \quad S_{\lq}^{\Sigma_0} = Z_{\D{\lG}^{\Sigma_0}}(\lq).
\] 
The short exact sequence 
\[
\xymatrix{1 \ar[r] & \mathbb{C}^{\times} \ar[r] & \D{\lG}^{\Sigma_0} \ar[r]  & \D{G}^{\Sigma_0} \ar[r] & 1}
\]
induces a long exact sequence
\begin{align*}
\xymatrix{1 \ar[r] &  \mathbb{C}^{\times} \ar[r] & S_{\lq}^{\Sigma_0} \ar[r] & S_{\q}^{\Sigma_0} \ar[r]^{\delta \quad \quad} & H^{1}(W_{F}, \mathbb{C}^{\times}),}
\end{align*}
and hence 
\begin{align*}
\xymatrix{1 \ar[r] &  S_{\lq}^{\Sigma_0}/\mathbb{C}^{\times} \ar[r]^{\quad \iota} & S_{\q}^{\Sigma_0} \ar[r]^{\delta \quad \quad} & H^{1}(W_{F}, \mathbb{C}^{\times}).}
\end{align*}
Taking quotient by $Z(\D{G})^{\Gal{F}}$, we get
\begin{align*}
\xymatrix{1 \ar[r] &  \cS{\lq}^{\Sigma_0} \ar[r]^{\iota} & \cS{\q}^{\Sigma_0} \ar[r]^{\bar{\delta} \quad \quad \quad \quad }  & H^{1}(W_{F}, \mathbb{C}^{\times})/\delta(Z(\D{G})^{\Gamma{}}),}                 
\end{align*}
where 
\[
\cS{\lq}^{\Sigma_0} =  S_{\lq}^{\Sigma_0}/Z(\D{\lG})^{\Gal{F}}, \quad \cS{\q}^{\Sigma_0} =  S_{\q}^{\Sigma_0}/Z(\D{G})^{\Gal{F}}
\] 
One can show that $\Im \delta$ is finite as in \cite[Lemma 2.1]{Xu:2018}, then $\cS{\lq}^{0} = \cS{\q}^{0}$. After taking the quotients by the identity components, we get
\begin{align}
\label{eq: twisted endoscopic sequence}
\xymatrix{1 \ar[r] &  \S{\lq}^{\Sigma_0} \ar[r]^{\iota} & \S{\q}^{\Sigma_0} \ar[r]^{\bar{\delta} \quad \quad \quad \quad }  & H^{1}(W_{F}, \mathbb{C}^{\times})/\delta(Z(\D{G})^{\Gamma{}}),}                
\end{align}
where 
\[
\S{\lq}^{\Sigma_0} =  \cS{\lq}^{\Sigma_0} / \cS{\lq}^{0}, \quad \S{\q}^{\Sigma_0} =  \cS{\q}^{\Sigma_0} / \cS{\q}^{0}.
\]  
There are natural maps from $S^{\Sigma_0}_{\q}, \cS{\q}^{\Sigma_0}$, and $\S{\q}^{\Sigma_0}$ to $\D{\Sigma}_0$, and for $\theta \in \Sigma_0$, we denote the preimages of $\D{\theta} \in \D{\Sigma}_0$ by $S^{\theta}_{\q}, \cS{\q}^{\theta}$ and $\S{\q}^{\theta}$ respectively. The above discussion also applies to $\lp \in \P{\lG}$ and $\p = {\bf p} \circ \lp \in \P{G}$.

Next we will show $H^{1}(W_{F}, \mathbb{C}^{\times})/\delta(Z(\D{G})^{\Gamma{}}) \cong \bar{H}^{1}(W_{F}, \mathbb{C}^{\times})$. Consider
\[
\xymatrix{1 \ar[r] & \mathbb{C}^{\times} \ar[r] & Z(\D{\lG}) \ar[r]  & Z(\D{G}) \ar[r] & 1,}
\]
it induces a long exact sequence
\[
\xymatrix{Z(\D{G})^{\Gal{F}} \ar[r]^{\delta \quad} & H^{1}(W_{F}, \mathbb{C}^{\times}) \ar[r] & H^{1}(W_{F}, Z(\D{\lG})) \ar[r] & H^{1}(W_{F}, Z(\D{G})) \ar[r] & H^{2}(W_{F}, \mathbb{C}^{\times}) = 1.
}
\]
So
\[
H^{1}(W_{F}, \mathbb{C}^{\times})/\delta(Z(\D{G})^{\Gamma_{F}}) \cong \bar{H}^{1}(W_{F}, \mathbb{C}^{\times})
\]
and hence
\begin{align}
\label{eq: twist}
\xymatrix{1 \ar[r] & \bar{H}^{1}(W_{F}, \mathbb{C}^{\times}) \ar[r] & H^{1}(W_{F}, Z(\D{\lG})) \ar[r] & H^{1}(W_{F}, Z(\D{G})) \ar[r] & 1
}
\end{align}
is exact. Through this, we will identify 
\(
\bar{H}^{1}(W_{F}, \mathbb{C}^{\times}) \cong \Hom(\lG(F)/G(F), \C^{\times})
\)
(cf. \cite[Section 2.2]{Xu:2018}) and denote $\bar{\delta}$ by $\a$. Indeed, the image of $\a$ lies in $X := \Hom(\lG(F)/G(F) \lZ(F) , \C^{\times})$. Moreover, one can show the following result as in \cite[Lemma 3.3]{Xu:2016}.

\begin{lemma}
\label{lemma: twist invariant}
\begin{align*}
& X(\lq) = \a(\S{\q}), \quad X^{\Sigma_0}(\lq) = \a(\S{\q}^{\Sigma_0}), \\
& X(\lp) = \a(\S{\p}), \quad  X^{\Sigma_0}(\lp) = \a(\S{\p}^{\Sigma_0}).
\end{align*}
\end{lemma}

For $\theta \in \Sigma_0$, any semisimple $s \in \cS{\q}^{\theta}$ and $\x := \alpha(s)$, let $\D{H} := Z_{\D{G}}(s)^{0}$ and it can be equipped with a Galois action given by $\q$. This determines a quasisplit connected reductive group $H$, and $\q$ will factor through $\L{H}$ for some $\theta$-twisted endoscopic datum $(H, s, \iota)$ of $G$. Hence we get a parameter $\q_{H} \in \Q{H}$. In this way, we call $(H, \q_{H})$ corresponds to $(\q, s)$, and denote this relation by 
\begin{align}
\label{eq: endoscopy}
(H, \q_{H}) \rightarrow (\q, s).
\end{align}
Let $\D{\lif{H}}$ be the preimage of $\D{H}$ in $\D{\lG}$ and it can be equipped with a Galois action given by $\lq$. This determines a quasisplit connected reductive group $\lif{H}$, and $\lq$ will factor through $\L{\lif{H}}$ for some $(\theta, \x)$-twisted endoscopic datum $(\lif{H}, \tilde{s}, \tilde{\iota})$ of $\lG$. Hence we get a parameter $\lq_{H} \in \Q{\lif{H}}$. 

From \cite[Section 2.1]{Xu:Lpacket}, we have a factorization
\[
H \cong \prod_{i} GL(n_i) \times G'
\]
where $G'$ is of the form \eqref{eq: product}, and
\[
\lif{H} \cong \prod_{i} GL(n_i) \times \lG'
\]
We extend the action of $\Sigma_0$ on $\lG'$ trivially to $\lif{H}$. When $H = G'$, it is called elliptic. If $\q_{H} = \prod_{i} \q_{i} \times \q'$ for $\q_{i} \in \Q{GL(n_i)}$ and $\q' \in \Q{G'}$, then
\(
\S{\q_{H}}^{\Sigma_0} \cong \S{\q'}^{\Sigma_0}
\)
and $\S{\lq_{H}}^{\Sigma_0} \cong \S{\lq'}^{\Sigma_0}$.

\section{Parabolic induction and Jacquet module}
\label{sec: Jac}

From now on, let $G$ be a quasisplit symplectic or special even orthogonal group over $F$ and $\lG$ the corresponding similitude group. We define the split symplectic group $Sp(2n)$ (resp. split special even orthogonal group $SO(2n)$) in $GL(2n)$ with respect to 
\[
\begin{pmatrix} 
      0 & -J_{n} \\
      J_{n} &  0 
\end{pmatrix}\, \Big(\text{resp.} \begin{pmatrix} 
      0 & J_{n} \\
      J_{n} &  0 
\end{pmatrix}\Big) \, \text{ for } J_{n} = 
\begin{pmatrix}
        &&&1\\
        &&1&\\
        &\iddots&&\\
        1&&&
\end{pmatrix}
\]
We denote the outer twist of $SO(2n)$ with respect to a quadratic extension $E / F$ by $SO(2n, \eta_{E/F})$, where $\eta_{E/F}$ is the quadratic character associated to $E/F$ by the local class field theory. The Levi subgroup $\lM$ of $\lG$ is isomorphic to
\begin{align}
\label{eq: levi}
GL(n_{1}) \times \cdots \times GL(n_{r}) \times \lG_{-},
\end{align}
where $\lG_{-}$ is of the same type as $\lG$ with rank $n_{-} + 1$ and $n = \sum_{i = 1}^{r} n_{i} + n_{-}$. Throughout this paper we fix a Borel subgroup $\lif{B}$ of $\lG$ consisting of upper-triangular matrices and choose $\lM$ to be contained in the group 
\[
\begin{pmatrix}
GL(n_{1})&&&&&&0 \\
&\ddots &&&&& \\
&& GL(n_{r})&&&&\\
&&&\lG_{-} &&&\\
&&&&GL(n_{r})&& \\
&&&&&\ddots&\\
0&&&&&&GL(n_{1})
\end{pmatrix}.
\]
This gives all the standard Levi subgroups if $\lG$ is $GSp(2n)$ or $GSO(2n, \eta)$ ($\eta \neq 1$), and $GO(2n)$-conjugacy classes of standard Levi subgroups if $\lG$ is $GSO(2n)$. We fix an isomorphism from $\eqref{eq: levi}$ to $\lif{M}$ as follows
\[
(g_{1}, \cdots g_{r}, g) \longrightarrow \text{diag}\{g_{1}, \cdots, g_{r}, g, \c(g){}_tg^{-1}_{r}, \cdots, \c(g){}_tg^{-1}_{1}\}
\]
if $n_{-} > 0$, and
\[
(g_{1}, \cdots g_{r}, g) \longrightarrow \text{diag}\{g_{1}, \cdots, g_{r}, \c(g){}_tg^{-1}_{r}, \cdots, \c(g){}_tg^{-1}_{1}\}
\]
if $n_{-} = 0$.
Here ${}_tg_{i} = J_{n_{i}}{}^tg_{i}J^{-1}_{n_{i}}$ for $1 \leqslant i \leqslant r$. If $\lif{M}$ is $\theta_{0}$-stable, we write $\lif{M}^{\Sigma_{0}} := \lif{M} \rtimes \Sigma_{0}$. Otherwise, we let $\lif{M}^{\Sigma_{0}} = \lif{M}$. Suppose $\tilde{\sigma}^{\Sigma_{0}} \in \Rep(\lif{M}^{\Sigma_{0}}(F))$, $\lr^{\Sigma_{0}} \in \Rep(\lG^{\Sigma_{0}}(F))$, the categories of finite-length smooth representations, we define the parabolic induction and Jacquet module as follows.

\begin{itemize}

\item If $\lif{M}^{\theta_{0}} = \lif{M}$, we define the normalized parabolic induction $\Ind^{\lG^{\Sigma_{0}}(F)}_{\lif{P}^{\Sigma_{0}}(F)} \tilde{\sigma}^{\Sigma_{0}}$ to be the extension of the representation $\Ind^{\lG(F)}_{\lif{P}(F)}(\tilde{\sigma}^{\Sigma_{0}}|_{\lif{M}(F)})$ by an induced action of $\Sigma_{0}$, and we define the normalized Jacquet module $\Jac_{\lif{P}^{\Sigma_{0}}(F)} \lr^{\Sigma_{0}}$ to be the extension of the representation $\Jac_{\lif{P}(F)}(\lr^{\Sigma_{0}}|_{\lG(F)})$ by an induced action of $\Sigma_{0}$.

\item If $\lif{M}^{\theta_{0}} \neq \lif{M}$, we define the normalized parabolic induction $\Ind^{\lG^{\Sigma_{0}}(F)}_{\lif{P}^{\Sigma_{0}}(F)} \tilde{\sigma}^{\Sigma_{0}}$ to be $\Ind^{\lG^{\Sigma_{0}}(F)}_{\lG(F)} \Ind^{\lG(F)}_{\lif{P}(F)}(\tilde{\sigma}^{\Sigma_{0}}|_{\lif{M}(F)})$, and we define the normalized Jacquet module $\Jac_{\lif{P}^{\Sigma_{0}}(F)} \lr^{\Sigma_{0}}$ to be $\Jac_{\lif{P}(F)}(\lr^{\Sigma_{0}}|_{\lG(F)})$.

\end{itemize}

For $\tilde{\sigma} = \tau_{1} \otimes \cdots \otimes \tau_{r} \otimes \lr^{\Sigma_0}_{-}$, where $\tau_{i} \in {\rm Rep}(GL(n_{i}, F))$ for $1 \leqslant i \leqslant r$ and $\lr^{\Sigma_0}_{-} \in {\rm Rep}(\lG^{\Sigma_0}_{-}(F))$, we denote the normalized parabolic induction $\Ind_{\lif{P}^{\Sigma_0}(F)}^{\lG^{\Sigma_0}(F)}(\tilde{\sigma}^{\Sigma_0})$ by 
\(
\tau_{1} \times \cdots \times \tau_{r} \rtimes \lr^{\Sigma_0}_{-}
\) 
and its socle by $\langle \tau_{1} \times \cdots \times \tau_{r} \rtimes \lr^{\Sigma_0}_{-} \rangle$. We also denote by $\tau_{1} \times \cdots \times \tau_{r}$ the normalized parabolic induction in $GL(\sum_{i = 1}^{r} n_i, F)$ and its socle by $\langle \tau_{1} \times \cdots \times \tau_{r} \rangle$. 

Suppose $\rho$ is a unitary irreducible supercuspidal representation of $GL(d_{\rho}, F)$. For an increasing (resp. decreasing) sequence $\{x, \cdots, y\}$ of real numbers of common distance $1$, we denote
\[
\langle x, \cdots, y \rangle_{\rho}:= \langle \rho||^{x} \times \cdots \times \rho||^{y} \rangle.
\]
Let $St(\rho, a) := \langle \frac{a-1}{2}, \cdots, -\frac{a-1}{2} \rangle_{\rho}$. More generally, we denote 
\[
  \begin{pmatrix}
   x_{11} & \cdots & x_{1n} \\
   \vdots &  & \vdots \\
   x_{m1} & \cdots & x_{mn}
  \end{pmatrix}_{\rho}
:= \langle \times_{i \in [1, m]} \langle x_{i1}, \cdots, x_{in} \rangle_{\rho} \rangle
\]
where each row is decreasing (resp. increasing) and each column is increasing (resp. decreasing) of common distance $1$. This includes the case that
\[
Sp(St(\rho, a), b) := \langle St(\rho, a)||^{-(b-1)/2} \times St(\rho, a)||^{-(b-3)/2} \times \cdots \times St(\rho, a)||^{(b-1)/2} \rangle.
\]

Suppose $\lif{M} = GL(d_{\rho}) \times \lG_{-}$. For $\lr^{\Sigma_0} \in \Rep(\lG^{\Sigma_0}(F))$, we can decompose the semisimplification of the Jacquet module
\[
s.s. \Jac_{\lif{P}^{\Sigma_0}(F)}(\lr^{\Sigma_0}) = \bigoplus_{i} \tau_{i} \otimes \tilde{\sigma}^{\Sigma_0}_{i},
\] 
where $\tau_{i} \in \Rep(GL(d_{\rho}, F))$ and $\tilde{\sigma}^{\Sigma_0}_{i} \in \Rep(\lG^{\Sigma_0}_{-}(F))$, both of which are irreducible. We define $\Jac^{\rho}_{x} (\lr^{\Sigma_0})$ for any real number $x$ to be 
\[
\Jac^{\rho}_{x}(\lr^{\Sigma_0}) := \bigoplus_{\tau_{i} = \rho||^{x}} \tilde{\sigma}^{\Sigma_0}_{i}.
\]
By linearity, this induces a map on the Grothendieck groups
\[
\Jac^{\rho}_{x}: K\Rep(\lG^{\Sigma_0}(F)) \rightarrow K\Rep(\lG^{\Sigma_0}_{-}(F)).
\]
If we have an ordered sequence of real numbers $\{x_{1}, \cdots, x_{s}\}$, we can define
\[
\Jac^{\rho}_{x_{1}, \cdots, x_{s}} \lr^{\Sigma_0} := \Jac^{\rho}_{x_{s}} \circ \cdots \circ \Jac^{\rho}_{x_{1}} \lr^{\Sigma_0}.
\]

\begin{lemma}
\label{lemma: Jacquet module}
For $A, B \in \mathbb{R}$ such that $A - B \in \mathbb{Z} \backslash \{0\}$ and $\zeta$ a sign, let
\[
\lr^{\Sigma_0} = \langle \zeta B, \cdots, -\zeta A \rangle_{\rho} \rtimes \lr^{\Sigma_0}_{-}.
\]
Then
\begin{align*}
{\rm Jac}^{\rho}_{\zeta B} \, \lr^{\Sigma_0} & = \langle \zeta(B-1), \cdots, -\zeta A \rangle_{\rho} \rtimes \lr_{-}^{\Sigma_0} + \langle \zeta B, \cdots, -\zeta A \rangle_{\rho} \rtimes {\rm Jac}^{\rho}_{\zeta B} \, \lr^{\Sigma_0}_{-}, \\
{\rm Jac}^{\rho^{\vee}}_{\zeta A} \, \lr^{\Sigma_0} & = \langle \zeta B, \cdots, -\zeta (A-1) \rangle_{\rho} \rtimes (\lr_{-}^{\Sigma_0} \otimes \eta_{\rho} \omega^{-\zeta A}_{\rho}) + \langle \zeta B, \cdots, -\zeta A \rangle_{\rho} \rtimes {\rm Jac}^{\rho^{\vee}}_{\zeta A} \, \lr_{-}^{\Sigma_0},
\end{align*}
where $\omega_{\rho} = |\cdot|^{d_{\rho}}$ and $\eta_{\rho}$ is the central character of $\rho$.
\end{lemma}

\begin{proof}
It follows from \cite[Theorem 5.2]{Tadic:1995}.
\end{proof}

For a real valued matrix $X$, we denote by $|X|$ the sum of entries in $X$ and define $\Jac^{\rho}_{X} := \circ_{x \in X} \Jac^{\rho}_{x}$, where $x$ ranges over $X$ from top to bottom then left to right. In particular, we denote
\[
X^{T}_{(\rho, A, B, \zeta)} := \begin{bmatrix}
    \zeta (B + T) & \cdots & \zeta(B + 1) \\
              \vdots &  & \vdots \\
              \zeta (A + T) & \cdots & \zeta(A + 1)
  \end{bmatrix}
\]
and 
\[
\Jac_{(\rho, A + T, B + T, \zeta) \mapsto (\rho, A, B, \zeta)} := \Jac^{\rho}_{X^{T}_{(\rho, A, B, \zeta)}}.
\]

Let $\overline{\Rep}(\lG(F))$ be the category of finite-length smooth representations of $\lG(F)$ viewed as $\sH(\lG(F))$-modules. We denote the elements in $\overline{\Rep}(\lG(F))$ by $[\lr]$ for $\lr \in \Rep(\lG(F))$, and we call $[\lr]$ is irreducible if $\lr$ is irreducible. Let
\[
\overline{\Jac}_{\lif{P}(F)} = \begin{cases}
                         \Jac_{\lif{P}(F)} + \Jac_{\lif{P}(F)} \circ \theta_{0},  & \text{ if $\lG = GSO(2n)$ and $\lif{M}^{\theta_{0}} \neq \lif{M}$,} \\
                         \Jac_{\lif{P}(F)},         & \text{ otherwise. }
                         \end{cases}
\]
We can define parabolic induction and Jacquet module on $\overline{\Rep}(\lG(F))$ as follows
\[
\Ind^{\lG(F)}_{\lif{P}(F)} [\tilde{\sigma}] := [\Ind^{\lG(F)}_{\lif{P}(F)} \tilde{\sigma}] \text{ and } \overline{\Jac}_{\lif{P}(F)} [\lr] := [\overline{\Jac}_{\lif{P}(F)} \lr].
\]
We define  
\[
\overline{\Jac}^{\rho}_{x} = \begin{cases}
                         \Jac^{\rho}_{x} + \Jac^{\rho}_{x} \circ \theta_{0},  & \text{ if $\lG = GSO(2n)$ and $n = d_{\rho} \neq 1$, } \\
                         \Jac^{\rho}_{x},         & \text{ otherwise. }
                         \end{cases}
\]
The above discussion also applies to $G(F)$ (cf. \cite{Xu:Apacket}).

Let ${\rm Rep}(\lG^{\Sigma_0}(F), \omega)$ be the category of objects $(\lr^{\Sigma_0}, A)$, were $\lr^{\Sigma_0} \in {\rm Rep}(\lG^{\Sigma_0}(F))$ and $A$ is an intertwining operator between $\lr^{\Sigma_0} \otimes \omega$ and $\lr^{\Sigma_0}$ such that $A^{2} = {\rm id}$. A morphism between $(\lr^{\Sigma_0}_1, A_1)$ and $(\lr^{\Sigma_0}_2, A_2)$ is a morphism $\varphi: \lr^{\Sigma_0}_1 \rightarrow \lr^{\Sigma_0}_2$ in ${\rm Rep}(\lG^{\Sigma_0}(F))$ such that $\varphi \circ A_1 =  A_2 \circ \varphi$. The simple objects $(\lr^{\Sigma_0}, A)$ can be obtained as follows.

\begin{itemize}
\item $\lr^{\Sigma_0}$ is irreducible and $\lr^{\Sigma_0} \cong \lr^{\Sigma_0} \otimes \omega$. There are two choices for the intertwining operator, say $A, -A$. In particular, $(\lr^{\Sigma_0}, A) \ncong (\lr^{\Sigma_0}, -A)$.

\item $\lr^{\Sigma_0}$ is a direct sum of two irreducible representations $\lr^{\Sigma_0}_1$, $\lr^{\Sigma_0}_2$ such that $\lr^{\Sigma_0}_2 \cong \lr^{\Sigma_0}_1 \otimes \omega \ncong \lr^{\Sigma_0}_1$. For any choice of intertwining operator $A$, the isomorphism class of $(\lr^{\Sigma_0}, A)$ is the same.
\end{itemize}
Let $\overline{{\rm Rep}}(\lG(F), \theta, \omega)$ be the category of objects $(\lr, A)$, were $\lr \in {\rm Rep}(\lG(F))$ and $A$ is an intertwining operator between $\lr \otimes \omega$ and $\lr^{\theta}$ such that $A^{2} = {\rm id}$. A morphism between $(\lr_1, A_1)$ and $(\lr_2, A_2)$ is a morphism $\varphi: \lr_1 \rightarrow \lr_2$ as $\sH(\lG(F))$-modules such that $\varphi \circ A_1 =  A_2 \circ \varphi$. The simple objects $(\lr, A)$ can be obtained as follows.
\begin{itemize}
\item $\lr$ is irreducible and $\lr \cong \lr^{\theta} \otimes \omega$. There are two choices for the intertwining operator, say $A, -A$. In particular, $(\lr, A) \ncong (\lr, -A)$.

\item $\lr$ is a direct sum of two irreducible representations $\lr_1$, $\lr_2$ such that $\lr_2 \cong \lr^{\theta}_1 \otimes \omega \ncong \lr_1$. For any choice of intertwining operator $A$, the isomorphism class of $(\lr, A)$ is the same.
\end{itemize}
Our definition of parabolic induction and Jacquet module (including its variants) can be extended to ${\rm Rep}(\lG^{\Sigma_0}(F), \omega)$ and $\overline{{\rm Rep}}(\lG(F), \theta, \omega)$ by taking the induced intertwining operators.

Let $R(\lG^{\Sigma_0}(F), \omega)$ be the space of finite linear combinations of $\omega$-twisted characters of $\lG^{\Sigma_0}(F)$ and $\bar{R}(\lG(F), \theta, \omega)$ the restriction to $\sH(\lG(F))$ of the space of finite linear combinations of $(\theta, \omega)$-twisted characters of $\lG(F)$. We have the following diagram on the corresponding Grothendieck groups:
\[
\xymatrix{K{\rm Rep}(G^{\Sigma_0}(F)) & K{\rm Rep}(\lG^{\Sigma_0}(F)) \ar[l]_{r_1} \\
& K{\rm Rep}(\lG^{\Sigma_0}(F), \omega) \ar@{->>}[r]^{r_3} \ar[u]^{f} \ar[ul]^{p_1} \ar[dl]_{p_2} \ar@{->>}[d]_{\rm ch_1} & K\overline{\rm Rep}(\lG(F), \theta, \omega) \ar@{->>}[d]_{\rm ch_2} \\
R(G^{\Sigma_0}(F)) & R(\lG^{\Sigma_0}(F), \omega) \ar[l]^{r_2} \ar@{->>}[r]_{r_4} & \bar{R}(\lG(F), \theta, \omega)
}
\]
\begin{itemize}
\item $f$ forgets the intertwining operators; 

\item $r_{1}, r_{2}$ are restrictions to $G^{\Sigma_0}(F)$, and $p_1 = r_1 \circ f, p_2 = r_2 \circ f$;

\item $r_3: (\lr^{\Sigma_0}, A) \mapsto (\lr^{\Sigma_0}|_{\lG}, \, \lr^{\Sigma_0}(\theta) \circ A)$;

\item $r_4$ is the restriction to $\lG(F) \rtimes \theta$;

\item ${\rm ch}_1, {\rm ch}_2$ are maps to the associated twisted characters.
\end{itemize}

\begin{lemma}
Elements in $K{\rm Rep}(\lG^{\Sigma_0}(F), \omega)$ can be determined by their images under $f$ and ${\rm ch}_1$.  
\end{lemma}

\begin{proof}
Suppose $Z \in K{\rm Rep}(\lG^{\Sigma_0}(F), \omega)$ satisfies that $f(Z) = {\rm ch}_1(Z) = 0$. For any irreducible object $(\lr^{\Sigma_0}, A)$ in ${\rm Rep}(\lG^{\Sigma_0}(F), \omega)$, we would like to show that it does not contribute to $Z$. If $\lr^{\Sigma_0}$ is reducible, then $(\lr^{\Sigma_0}, A)$ is uniquely determined by $f((\lr^{\Sigma_0}, A))$, so the result is clear. If $\lr^{\Sigma_0}$ is irreducible, then $f((\lr^{\Sigma_0}, A)) = f((\lr^{\Sigma_0}, -A))$ and ${\rm ch}_1((\lr^{\Sigma_0}, A)) = -{\rm ch}_1((\lr^{\Sigma_0}, -A))$. Then we can solve for the multiplicities of both $(\lr^{\Sigma_0}, A), (\lr^{\Sigma_0}, -A)$ in $Z$, which are necessarily zero. 
\end{proof}

\begin{corollary}
\label{cor: determination}
Suppose $Z \in K{\rm Rep}(\lG^{\Sigma_0}(F), \omega)$ is supported on irreducible representation $\lr^{\Sigma_0}$ with some fixed infinitesimal character $\tilde{\lambda}$ and $X(\lr^{\Sigma_0}) = X^{\Sigma_0}(\tilde{\lambda})$, then $Z$ can be determined by its image under $p_1, p_2$.
\end{corollary}

\begin{proof}
By the assumption $f(Z), {\rm ch}_1(Z)$ can be determined by $p_1(Z), p_2(Z)$ respectively. Then the result follows from the previous lemma.
\end{proof}

\section{Local Langlands correspondence for similitude groups}
\label{sec: LLC}

In order to define the $\Sigma_0$-infinitesimal characters of $\cPkt{}(\lG(F))$, we need the local Langlands correspondence for $\lG$ modulo $\theta_0$-conjugation. From \cite[Theorem 8.12]{Xu:Lpacket}, we can construct such a correspondence, which depends on the choices for simple parameters of $G$. Although this correspondence is not canonical, it is compatible with parabolic induction and twisted endoscopic transfer. For our application, we would further require that the $\Sigma_0$-infinitesimal characters are preserved under parabolic inductions, cf. Proposition~\ref{prop: induction preserves infinitesimal character}. So we will make the following choices for simple parameters. Let $\p_a := \rho \otimes \nu_a \in \cPsm{G_a}$, where $\rho$ is a self-dual irreducible unitary $d_{\rho}$-dimensional representation of $W_{F}$ and $\nu_a$ is the irreducible $a$-dimensional representation of $SL(2, \C)$. By the local Langlands correspondence for general linear groups \cite{HarrisTaylor:2001} \cite{Henniart:2000} \cite{Scholze:2013}, we can associate $\rho$ with a self-dual irreducible unitary supercuspidal representation of $GL(d_{\rho}, F)$, denoted again by $\rho$. Let $\eta_{\rho}$ be the central character of $\rho$. When $\rho$ is of orthogonal type, it parametrizes an $L$-packet $\cPkt{\rho}$ of $G_{\rho}(F)$, which is a singleton for $\S{\rho} = 1$. Then we fix a correspondence 
\(
\tilde{\rho} \mapsto \cPkt{\tilde{\rho}} 
\)
such that 
\begin{enumerate}

\item $\cPkt{\tilde{\rho} \otimes \omega} = \cPkt{\tilde{\rho}} \otimes \omega$ for any $\omega \in H^{1}(W_{F}, Z(\D{\lG}_{\rho})) \cong {\rm Hom}(\lG_{\rho}(F), \mathbb{C}^{\times})$,

\item the central character $\chi_{\tilde{\rho}}$ of $\cPkt{\tilde{\rho}}$ is given by the composition of $\tilde{\rho}$ with $\D{\lG}_{\rho} \rightarrow \D{Z}_{\lG_{\rho}}$. 

\end{enumerate}
Next we construct the correspondence for $\lp_a$ by induction on $a$. Note $\cPkt{\p_a}$ is also a singleton and we have
\[
\cPkt{\p_{a}} \hookrightarrow \rho||^{\frac{a-1}{2}} \rtimes \cPkt{\p_{a-2}}
\] 
induced from $M := GL(d_{\rho}) \times G_{a-2}$ (cf. \cite[Proposition 8.1]{Xu:cusp}). Suppose we have associated $\lp_{a-2}$ with a packet $\cPkt{\lp_{a-2}}$. Let $\lp_{a}$ be the lift of $\p_{a}$ such that the $\Sigma_0$-infinitesimal character is the same as that of
\(
\rho||^{\frac{a-1}{2}} \rtimes \lp_{a-2}.
\)
Then we associate $\lp_a$ with $\cPkt{\lp_{a}}$ such that
\[
\cPkt{\lp_{a}} \hookrightarrow \rho||^{\frac{a-1}{2}} \rtimes \cPkt{\lp_{a-2}}.
\]

\begin{proposition}
\label{prop: induction preserves infinitesimal character}
For any Levi subgroup $\lM$ of $\lG$ and $\Sigma_0$-infinitesimal character $[\lambda_{\lM}]$ for $\lM$, we have
\[
{\rm Ind}^{\lG(F)}_{\lif{P}(F)} \, \cPkt{}(\lM(F))_{\lambda_{\lM}} \subseteq \cPkt{}(\lG(F))_{\tilde{\lambda}}
\]
for
\(
[\tilde{\lambda}] = [\iota_{\lM} \circ \lambda_{\lM}],
\)
where
\(
\iota_{\lM}: \L{\lM} \rightarrow \L{\lG}.
\)
\end{proposition}

This kind of statement is due to Haines \cite{Haines:2014}. In the case of classical groups, this has been proved by Moussaoui \cite{Moussaoui:2017}. In our case, the critical step is the following lemma.

\begin{lemma}
\label{lemma: cuspidal support}
For $[\lr] \in \cPkt{}(\lG(F))_{\tilde{\lambda}}$ with cuspidal support $(\lif{L}, \lr_{\rm cusp})$, if $[\lambda_{\lif{L}}]$ is the $\Sigma_0$-infinitesimal character of $\lr_{\rm cusp}$, then 
\(
[\tilde{\lambda}] = [\iota_{\lif{L}} \circ \lambda_{\lif{L}}].
\)
\end{lemma}

\begin{proof}
Let $[\lp] \in \cPbd{\lG}$ be the Langlands parameter of $[\lr]$, then $\lp$ factors through $\lp_{M} \otimes \xi \in \P{\lif{M}}$ for some parabolic subgroup $\lif{P}$ of $\lG$ with Levi factor $\lif{M}$ and $\lr$ is the Langlands quotient of ${\rm Ind}^{\lG(F)}_{\lif{P}(F)} (\tilde{\sigma} \otimes \xi)$ for some tempered representation $[\tilde{\sigma}] \in \cPkt{\lp_{M}}$. It suffices to prove the lemma for $[\tilde{\sigma}]$. So we can assume $\lr$ is tempered. In this case, $\lp$ factors through $\lp_{M} \in \Pdt{\lif{M}}$ and $\lr$ is a direct summand of ${\rm Ind}^{\lG(F)}_{\lif{P}(F)} (\tilde{\sigma})$ for some discrete series representation $[\tilde{\sigma}] \in \cPkt{\lp_M}$. Therefore we can further reduce to the discrete series case.

Suppose $\lr$ is a non-cuspidal discrete series representation. Let $[\p] := [{\bold p} \circ \lp] \in \cPdt{G}$ and $\r$ be an irreducible constituent of $\lr|_{G}$. Then $\r$ is also a non-cuspidal discrete series representation. By \cite[Theorem 3.3, Proposition 8.1]{Xu:cusp}, there exists a self-dual irreducible unitary supercuspidal representation $\rho$ of $GL(d_{\rho}, F)$ and integer $a$ such that $\r_{-} := \overline{{\rm Jac}}^{\rho}_{_{\frac{a-1}{2}}} \r$ is an irreducible tempered representation of $G_{-}(F)$. Then $\lr_{-} := \overline{{\rm Jac}}^{\rho}_{_{\frac{a-1}{2}}} \lr$ is an irreducible representation of $\lG_{-}(F)$ containing $\r_{-}$ in its restriction to $G_{-}(F)$, in particular, it is also tempered. By Frobenius reciprocity, $\lr$ or $\lr^{\theta_0}$ is an irreducible constituent of $\rho||^{\frac{a-1}{2}} \rtimes \lr_{-}$. By the induction on the rank of groups, we can assume the Lemma holds for $\lG_{-}$. Then it suffices to show the $\Sigma_0$-infinitesimal character of $\rho||^{\frac{a-1}{2}} \boxtimes \lr_{-}$ induces that of $\lr$. For this purpose we need to find the $L$-parameter of $\lr_{-}$.

If $\p = \rho \otimes \nu_a := \p_a$, then the lemma follows from our construction of the correspondence for $\lp_a$. Otherwise, we can factor $\p$ through $\p_{H} := \p_{a} \times \p' \in \cPbd{H}$ for a twisted endoscopic group $H := G_{a} \times G'$ of $G$, where $\p_{a} = \rho \otimes \nu_a$ if $d_{\rho}$ is even (resp. $\p_{a} = (\rho \otimes \eta_{\rho}) \otimes \nu_{a}$ if $d_{\rho}$ is odd). Correspondingly, $\lp$ factors through $\lp_{H} = {\bold p} \circ (\lp_{a} \times \lp') \in \cPbd{\lif{H}}$ for a twisted endoscopic group $\lif{H}$ of $\lG$. Choose $\p_{-} \in \cPbd{G_{-}}$ factoring through $\p_{H_{-}} = \p_{a-2} \times \p'$ for a twisted endoscopic group $H_{-} := G_{a-2} \times G'$ of $G_{-}$. Choose $\lp_{-} \in \cPbd{\lG_{-}}$ factoring through $\lp_{H_{-}} = {\bold p} \circ (\lp_{a-2} \times \lp')$ for the corresponding twisted endoscopic group $\lif{H}_{-}$. Here $\lp_{a-2}$ is so chosen that $\lp_{a}$ has the same $\Sigma_0$-infinitesimal character as that of
$\rho||^{\frac{a-1}{2}} \rtimes \lp_{a-2}$ if $d_{\rho}$ is even (resp. $\rho \otimes \eta_{\rho} ||^{\frac{a-1}{2}} \rtimes \lp_{a-2}$ if $d_{\rho}$ is odd). The twisted endoscopic embeddings are so chosen that
\[
\xymatrix{\L{\lif{H}} \ar[r] & \L{\lG} \\
\L{(GL(d_{\rho}) \times \lif{H}_{-})} \ar[u] \ar[r] & \L{(GL(d_{\rho}) \times \lG_{-})} \ar[u]
}
\]
commutes. Then $\lp$ has the same $\Sigma_0$-infinitesimal character as that of 
\(
\rho||^{\frac{a-1}{2}} \rtimes \lp_{-}.
\)
It remains to show that $[\lp_{-}]$ is the Langlands parameter of $[\lr_{-}]$. First of all, 
\(
\cPkt{\lp_{a-2}} = \overline{{\rm Jac}}^{\rho}_{\frac{a-1}{2}} \cPkt{\lp_{a}}
\)
(resp. $\cPkt{\lp_{a-2}} = \overline{{\rm Jac}}^{\rho \otimes \eta_{\rho}}_{\frac{a-1}{2}} \cPkt{\lp_{a}}
$)
from our construction. Secondly, $\cPkt{\lp} =  {\rm Tran} \, (\cPkt{\lp_{a}} \, \tilde{\otimes} \, \cPkt{\lp'})$ and $\cPkt{\lp_{-}} = {\rm Tran} \, (\cPkt{\lp_{a-2}} \, \tilde{\otimes} \, \cPkt{\lp'})$ (cf. \cite[Theorem 8.12]{Xu:Lpacket}). At last, it follows from the compatibility of the twisted endoscopic transfer with Jacquet module \cite[Appendix C]{Xu:cusp} that $\cPkt{\lp_{-}} = \overline{{\rm Jac}}^{\rho}_{\frac{a-1}{2}} \cPkt{\lp}$. Hence $[\lr_{-}] \in \cPkt{\lp_{-}}$. 
\end{proof}

\begin{remark}
In \cite[Appendix C]{Xu:cusp}, we have only considered the twist by automorphism. Our case can be deduced from there by considering $\lif{M}^{+} := \lG \times GL(1)$ with the automorphism $\theta^{+}(g, x) := (\theta(g), \lambda(g)x^{-1})$, where $\theta \in \Sigma_0$. Note a representation $\lr \boxtimes \omega$ of $\lM^{+}$ is $\theta^{+}$-invariant if and only if $\lr^{\theta} \otimes \x \cong \lr$, in which case $\omega^2 = 1$.
\end{remark}

Now we can complete the proof of Proposition~\ref{prop: induction preserves infinitesimal character}. 

\begin{proof}

Suppose the cuspidal support of $[\tilde{\sigma}] \in \cPkt{}(\lif{M}(F))_{\lambda_{\lif{M}}}$ is $(\lif{L}, \tilde{\sigma}_{\rm cusp})$. Then this is also the cuspidal support of $\lr$. Let $[\lambda_{\lif{L}}]$ be the $\Sigma_0$-infinitesimal character of $\tilde{\sigma}_{\rm cusp}$. Then by this lemma, the $\Sigma_0$-infinitesimal character of $\lr$ is
\(
[\tilde{\lambda}] = [\iota_{\lif{L}} \circ \lambda_{\lif{L}}] = [\iota_{\lif{M}} \circ \iota^{\lif{M}}_{\lif{L}} \circ \lambda_{\lif{L}}] = [\iota_{\lif{M}} \circ \lambda_{\lif{M}}].
\)
\end{proof}

\begin{corollary}
\label{cor: cuspidal support}
For any Levi subgroup $\lM$ of $\lG$ and $\Sigma_0$-infinitesimal character $[\lambda_{\lM}]$ for $\lM$, we have
\[
{\rm Ind}^{\lG^{\Sigma_0}(F)}_{\lif{P}^{\Sigma_0}(F)} \, \Pkt{}(\lM^{\Sigma_0}(F))_{\lambda_{\lM}} \subseteq \Pkt{}(\lG^{\Sigma_0}(F))_{\tilde{\lambda}}
\]
for
\(
[\tilde{\lambda}] = [\iota_{\lM} \circ \lambda_{\lM}],
\)
where
\(
\iota_{\lM}: \L{\lM} \rightarrow \L{\lG}.
\)
\end{corollary}

\begin{corollary}
\label{cor: Jac preserves infinitesimal character}
For $[\lr] \in \bar{\Pi}(\lG(F))_{\tilde{\lambda}}$ (resp. $\lr^{\Sigma_0} \in \Pi(\lG^{\Sigma_0}(F))_{\tilde{\lambda}}$), let $\tilde{\sigma}$ (resp. $\tilde{\sigma}^{\Sigma_0}$) be any irreducible constituent of ${\rm Jac}_{\lP(F)} \lr$ (resp. ${\rm Jac}_{\lP^{\Sigma_0}(F)} \lr^{\Sigma_0}$) for a parabolic subgroup $\lP$ of $\lG$, then the $\Sigma_0$-infinitesimal character of ${\rm Ind}^{\lG(F)}_{\lP(F)} \, \tilde{\sigma}$ (resp. ${\rm Ind}^{\lG^{\Sigma_0}(F)}_{\lP^{\Sigma_0}(F)} \, \tilde{\sigma}^{\Sigma_0}$) is $\tilde{\lambda}$. 
\end{corollary}

\begin{proof}
It suffices to treat $[\lr] \in \bar{\Pi}(\lG(F))_{\tilde{\lambda}}$. Suppose $\lr$ has cuspidal support $(\lif{L}, \lr_{\rm cusp})$. Let $\lif{Q}$ be a parabolic subgroup of $\lG$ with Levi factor $\lif{L}$. Then $\lr$ is an irreducible constitutent of ${\rm Ind}^{\lG(F)}_{\lif{Q}(F)} \, \lr_{\rm cusp}$. By the geometric lemma,
\[
{\rm Jac}_{\lP(F)} \, {\rm Ind}^{\lG(F)}_{\lif{Q}(F)} (\lr_{\rm cusp}) = \sum_{w \in W^{\lP, \lif{Q}} : \lif{L} \subseteq w \lM} {\rm Ind}^{\lM(F)}_{w^{-1} \lif{Q}(F) \cap \lM(F)} \, (w^{-1} \, \lr_{\rm cusp})
\]
in the Grothendieck group, where
\[
W^{\lP, \lif{Q}} := \{w \in W^{\lG} : w(\lif{M} \cap \lif{B}) \subseteq \lif{B} \text{ and } w^{-1}(\lif{L} \cap \lif{B}) \subseteq \lif{B}\}.
\]
It follows that $\lif{\sigma}$ is an irreducible constituent of ${\rm Ind}^{\lM(F)}_{w^{-1} \lif{Q}(F) \cap \lM(F)} \, (w^{-1} \, \lr_{\rm cusp})$ for some $w$. So the $\Sigma_0$-infinitesimal character of ${\rm Ind}^{\lG(F)}_{\lP(F)} (\tilde{\sigma})$ is the same as that of ${\rm Ind}^{\lG(F)}_{(w^{-1} \lif{Q}(F) \cap \lM(F))N_{\lP}(F)} \, (w^{-1} \, \lr_{\rm cusp})$. The rest follows from Lemma~\ref{lemma: cuspidal support}.

\end{proof}

At last, for $\lr^{\Sigma_0} \in \Pkt{}(\lG^{\Sigma_0}(F))$ and $\lr \in \Pkt{}(\lG(F))$ such that $\lr \subseteq \lr^{\Sigma_0}|_{\lG(F)}$, let 
\begin{align*}
& X(\lr^{\Sigma_0}) = \{\x \in X \, | \, \lr^{\Sigma_0} \otimes \x \cong \lr^{\Sigma_0}\}, \\
& X(\lr) =  \{\x \in X \, | \, \lr \otimes \x  \cong \lr \}, \\
& X([\lr]) = \{\omega \in X \, | \, [\lr \otimes \omega] =  [\lr] \}.
\end{align*}
We have $X(\lr^{\Sigma_0}) = X([\lr])$.

\begin{lemma}
Suppose $\lp \in \P{\lG}$ and $\lr^{\Sigma_0} \in \Pkt{\lp}^{\Sigma_0}$, $[\lr] \in \cPkt{\lp}$ such that $\lr \subseteq \lr^{\Sigma_0}|_{\lG(F)}$, then 
\[
X(\lr^{\Sigma_0}) = X^{\Sigma_0}(\lp), \quad X(\lr) = X(\lp).
\]
\end{lemma}

\begin{proof}
It follows from Lemma~\ref{lemma: twist invariant} and \cite[Corollary 4.2]{Xu:2018}.
\end{proof}

\section{Statement of main results}
\label{sec: main results}

From now on we will fix a $\theta_0$-stable Whittaker datum for $G$. For $[\q] \in \cQ{G}$ and $\lambda = \lambda_{\q}$, Arthur \cite[Theorem 1.5.1]{Arthur:2013} associated it with a multi-set $\cPkt{\q}$ over $\cPkt{}(G(F))$. It follows from M{\oe}glin's multiplicity one result \cite{Moeglin1:2011} and \cite[(2.2.12)]{Arthur:2013} that $\cPkt{\q}$ is a subset of $\cPkt{}(G(F))_{\lambda}$. It is equipped with a map
\begin{align}
\label{eq: Whittaker parametrization G}
\cPkt{\q} \longrightarrow \D{\S{\q}}, \quad [\r] \mapsto \langle \cdot, \r \rangle_{W}
\end{align}
with respect to our choice of Whittaker datum, such that 
\begin{enumerate}

\item
\[
f_{W}(\q) := \sum_{[\r] \in \cPkt{\q}} \langle s_{\q}, \r \rangle_{W}  \, f_{G}(\r), \quad f \in \sH(G(F))
\]
is stable.

\item For semisimple $s \in \cS{\q}$ and $(H ,\q_{H}) \rightarrow (\q, s)$, we have
\begin{align}
\label{eq: character relation G}
f^{H}_{W}(\q_{H}) = \sum_{[\r] \in \cPkt{\q}} \langle ss_{\q}, \r \rangle_{W} \, f_{G}(\r), \quad f \in \sH(G(F)),
\end{align}
\end{enumerate}
where $f^{H}$ is the transfer of $f$ and $f^{H}_{W}(\q)$ is defined with respect to $\cPkt{\q_H} := \otimes_{i} \Pkt{\q_{i}} \otimes \cPkt{\q'}$ for $\q_{H} = \prod_{i} \q_{i} \times \q'$ with $\q_{i} \in \Q{GL(n_i)}$ and $\q' \in \Q{G'}$. For $\bar{\e} \in \D{\S{\q}}$, let $\r_{W}(\q, \bar{\e})$ be the direct sum of the preimages of $\bar{\e}$ under \eqref{eq: Whittaker parametrization G}.

We define $\Pkt{\q}^{\Sigma_{0}}$ to be the set of irreducible representations of $G^{\Sigma_{0}}(F)$, whose restriction to $G(F)$ belong to $\cPkt{\q}$. If $\S{\q}^{\Sigma_{0}} \neq \S{\q}$, then $\r^{\theta_{0}} \cong \r$ for any irreducible constituent $[\r]$ in $\r_{W}(\q, \bar{\e})$ (cf. \cite[Section 8]{Xu:Apacket}). So we can define
\begin{align}
\label{eq: Whittaker parametrization full orthogonal}
\Pkt{\q}^{\Sigma_0} \longrightarrow \D{\S{\q}^{\Sigma_0}}, \quad \r^{\Sigma_0} \mapsto \langle \cdot, \r^{\Sigma_0} \rangle_{W}
\end{align}
so that the following properties are satisfied. For $\e \in \D{\S{\q}^{\Sigma_0}}$, let $\r^{\Sigma_0}_{W}(\q, \e)$ be the direct sum of the preimages of $\e$ under \eqref{eq: Whittaker parametrization full orthogonal} and $\bar{\e} = \e|_{\S{\q}}$, then
\begin{itemize}

\item $[\r^{\Sigma_{0}}_{W}(\q, \e)|_{G(F)}] = 2 \r_{W}(\q , \bar{\e})$ if $G$ is special even orthogonal and $\S{\q}^{\Sigma_{0}} = \S{\q}$, or $\r_{W}(\q ,\bar{\e})$ otherwise,

\item for any semisimple $s \in \cS{\q}^{\Sigma_{0}}$ but not in $\cS{\q}$ and $(H, \q_{H}) \rightarrow (\q, s)$, the following identity holds

\begin{align}
\label{eq: twisted character relation classical}
f^{H}_{W}(\q_{H}) = \sum_{[\r] \in \cPkt{\q}} \langle ss_{\q}, \r^{\Sigma_0} \rangle_{W} \, f_{G}(\r^{\Sigma_0}), \,\,\,\,\,\,\,\,\,\,\, f \in C^{\infty}_{c}(G(F) \rtimes \theta_{0}),
\end{align}
where $\r^{\Sigma_0}$ is any preimage of $[\r]$.
\end{itemize}
Recall that there is an exact sequence

\begin{align}
\label{eq: local twisted endoscopic sequence}
\xymatrix{1 \ar[r] &  \S{\lq}^{\Sigma_0} \ar[r]^{\iota} & \S{\q}^{\Sigma_0} \ar[r]^{\a  \quad \quad \quad \quad}  & \Hom(\lG(F)/G(F), \C^{\times})}.               
\end{align}

\begin{lemma}
\label{lemma: conjugation} \,
\begin{enumerate} 
\item $\r_W(\q, \bar{\e}_1), \r_W(\q, \bar{\e}_2)$ are conjugate under $\lG(F)$ if and only if $\bar{\e}_1 / \bar{\e}_2$ is trivial on $\S{\lq}$. 
\item $\r^{\Sigma_0}_W(\q, \e_1), \r^{\Sigma_0}_W(\q, \e_2)$ are conjugate under $\lG^{\Sigma_0}(F)$ if and only if $\e_1 / \e_2$ is trivial on $\S{\lq}^{\Sigma_0}$. 
\end{enumerate}
\end{lemma}

\begin{proof}
Since $f_{W}(\q)$ is stable, then $\cPkt{\q}$ is stable under $\lG(F)$-conjugation. Moreover, the character relation \eqref{eq: character relation G} implies that for any $g \in \lG(F)$ and $x \in \S{\q}$,
\[
\sum_{\bar{\e} \in \widehat{\S{\q}}} \bar{\e}(s_{\q} x) \, \r_W(\q, \bar{\e})^{g^{-1}} = \sum_{\bar{\e} \in \widehat{\S{\q}}} \omega_x(g) \bar{\e}(s_{\q}x) \, \r_W(\q, \bar{\e}),
\]
where $\omega_x(g) = \alpha(x)(g)$ (cf. \cite[Lemma 3.13]{Xu:2016}). Note $s_{\q} = s_{\lq}$, so $\omega_{s_{\q}}(\cdot) = 1$. It follows that
\begin{align}
\label{eq: conjugate}
\r_W(\q, \bar{\e})^{g} = \r_W(\q, \omega_{(\cdot)}(g) \, \bar{\e}).
\end{align}
Since $\alpha$ induces an injection $\S{\q}/\S{\lq} \hookrightarrow {\rm Hom}(\lG(F)/G(F), \mathbb{C}^{\times})$, then
\[
\lG(F)/G(F) \rightarrow {\rm Hom}(\S{\q}/\S{\lq}, \, \Two), \quad g \mapsto \omega_{(\cdot)}(g)
\]
is surjective. Part (1) follows from this. Part (2) also follows when $\S{\q}^{\Sigma_0} = \S{\q}$.

Now let us suppose $\S{\q}^{\Sigma_0} \neq \S{\q}$, then $\r^{\theta_0} \cong \r$ for any irreducible constituent $[\r]$ in $\r_{W}(\q, \bar{\e})$. The character relation \eqref{eq: twisted character relation classical} implies that for any $g \in \lG(F)$ and $x \in \S{\q}^{\theta_0}$,
\[
\sum_{\bar{\e} \in \widehat{\S{\q}}} \e(s_{\q} x) \, f_{G}(\r^{\Sigma_0}_W(\q, \e)^{g^{-1}}) = \sum_{\bar{\e} \in \widehat{\S{\q}}} \omega_{x}(g) \e(s_{\q}x) \, f_{G}(\r^{\Sigma_0}_W(\q, \e)), \quad f \in C^{\infty}_{c}(G(F) \rtimes \theta_0).
\] 
Combining with \eqref{eq: conjugate}, we get
\begin{align}
\label{eq: conjugate full orthogonal}
\r^{\Sigma_0}_W(\q, \e)^{g} = \r^{\Sigma_0}_W(\q, \omega_{(\cdot)}(g) \, \e).
\end{align}
Since $\alpha$ induces an injection $\S{\q}^{\Sigma_0}/\S{\lq}^{\Sigma_0} \hookrightarrow {\rm Hom}(\lG(F)/G(F), \mathbb{C}^{\times})$, then
\[
\lG(F)/G(F) \rightarrow {\rm Hom}(\S{\q}^{\Sigma_0}/\S{\lq}^{\Sigma_0}, \, \Two), \quad g \mapsto \omega_{(\cdot)}(g)
\]
is surjective. Part (2) follows from this.
\end{proof}

Let $\tilde{\zeta}$ be a character of $\lZ(F)$ whose restriction to $\Z(F)$ is the central character of $\cPkt{\q}$. Let $\clPkt{\q, \tilde{\zeta}}$ (resp. $\lPkt{\q, \tilde{\zeta}}^{\Sigma_0}$) be the subset of $\cPkt{}(\lG(F))$ (resp. $\Pkt{}(\lG^{\Sigma_0}(F))$) with central character $\tilde{\zeta}$, such that the restrictions to $G(F)$ (resp, $G^{\Sigma_0}(F)$) belong to $\cPkt{\q}$ (resp. $\Pkt{\q}^{\Sigma_0}$). 

\begin{corollary}
\label{cor: Whittaker pairing}
There exist unique maps
\begin{align}
\label{eq: Whittaker pairing}
\clPkt{\q, \tilde{\zeta}} \rightarrow \D{\S{\lq}} \quad \text{ and } \quad \lPkt{\q, \tilde{\zeta}}^{\Sigma_0} \rightarrow \D{\S{\lq}^{\Sigma_0}}
\end{align}
such that 
\[
\xymatrix{\clPkt{\q, \tilde{\zeta}} \ar[r] \ar[d] & \D{\S{\lq}} \\
\cPkt{\q} \ar[r] & \D{\S{\q}} \ar[u]
} \quad 
\text{ and } \quad \xymatrix{\lPkt{\q, \tilde{\zeta}}^{\Sigma_0} \ar[r] \ar[d] & \D{\S{\lq}^{\Sigma_0}} \\
\Pkt{\q}^{\Sigma_0} \ar[r] & \D{\S{\q}^{\Sigma_0}} \ar[u]
} 
\]
commute respectively.

\end{corollary}

\begin{proof}
For $\lr^{\Sigma_0} \in \lPkt{\q, \tilde{\zeta}}^{\Sigma_0}$, we choose $\r^{\Sigma_0}$ in the restriction of $\lr^{\Sigma_0}$. Then we define 
\[
\langle \cdot , \lr^{\Sigma_0} \rangle_{W} : = \langle \cdot , \r^{\Sigma_0} \rangle_{W} |_{\S{\lq}^{\Sigma_0}}.
\] 
By Lemma~\ref{lemma: conjugation}, we see that $\langle \cdot , \lr^{\Sigma_0} \rangle_{W}$ is independent of the choice of $\r^{\Sigma_0}$. The commutativity of the diagram is clear from our definition. The uniqueness is also clear. By restriction, we can deduce the other case.
\end{proof}

\begin{lemma}
\label{lemma: twist}
For $[\r] \in \cPkt{\q}$, 
\(
X(\lq) \subseteq X(\lp_{\r})
\)
and $X^{\Sigma_0}(\lq) \subseteq X^{\Sigma_0}(\lp_{\r})$. Moreover for any $\x \in \bar{H}^{1}(W_{F}, \mathbb{C}^{\times})$,
\begin{align}
\label{eq: twist}
\lq^{\theta_0} = \lq \otimes \x \Rightarrow \lp^{\theta_0}_{\r} = \lp_{\r} \otimes \x.
\end{align}
\end{lemma}

\begin{proof}
Let $\lG(\r)$ (resp. $\lG([\r])$) be the stabilizer of $\r$ (resp. $[\r]$) in $\lG(F)$ under the conjugation action. By the refinement of $L$-packets in the case of even orthogonal groups \cite[Theorem 8.4.1]{Arthur:2013}, we have $\lG(\r) = \lG([\r])$. Since 
\[
\lG([\r]) \subseteq \lG(\r_{W}(\q, \e)) = \bigcap_{\omega \in \a(\S{\q})} \, {\rm Ker} \, \omega,
\]
then
\[
X(\lp_{\r}) = X(\lr) = (\lG/\lG(\r))^{*} \supseteq (\lG/\lG(\r_{W}(\q, \e)))^{*} = \a(\S{\q}) = X(\lq).
\]
This also settles \eqref{eq: twist} in the case when $\theta_0 = id$. For the remaining cases, it suffices to assume $\S{\q}^{\Sigma_0} \neq \S{\q}$. Then $\r^{\theta_0} \cong \r$. Suppose $\lq^{\theta_0} = \lq \otimes \x$, there exists $x_0 \in \S{\q}^{\theta_0}$ such that $\x = \alpha(x_0) = \omega_{x_{0}}$. From \eqref{eq: conjugate full orthogonal}, we can deduce that for $g \in \lG(\r)$ and $f \in C^{\infty}_{c}(G(F) \rtimes \theta_0)$,
\[
\langle x_{0}, (\r^{\Sigma_0})^{g} \rangle_{W}\, f_{G}(\r^{\Sigma_0}) =   \x_{x_{0}}(g)\langle x_{0}, \r^{\Sigma_0}\rangle_{W} \, f_{G}(\r^{\Sigma_0}) =  \x_{x_{0}}(g) \langle x_{0}, (\r^{\Sigma_0})^{g} \rangle_{W} \, f_{G}((\r^{\Sigma_0})^{g}).
\]
Hence
\[
f_{G}((\r^{\Sigma_0})^{g}) = \x_{x_{0}}(g)^{-1} f_{G}(\r^{\Sigma_0}).
\]
From this we can conclude that
\(
\lr^{\theta_0} \cong \lr \otimes \omega_{x_0}
\)
as in the proof of \cite[Proposition 6.16]{Xu:2016}. This completes the proof of \eqref{eq: twist}. As a consequence, we get $X^{\Sigma_0}(\lq) \subseteq X^{\Sigma_0}(\lp_{\r})$.
\end{proof}




Let $\lif{Z}_{F}$ be a closed subgroup of $Z_{\lG}(F)$ such that $\c(\lif{Z}_{F})$ has finite index in $F^{\times}$ and $Z_{F} = \lif{Z}_{F} \cap G(F)$. Let $\tilde{\chi} = \tilde{\zeta}|_{\lif{Z}_{F}}$ and $\chi = \tilde{\chi}|_{Z_{F}}$. Define $\sH(\lG(F), \tilde{\chi})$ (resp. $\sH(G(F), \chi)$) to be the space $\tilde{\chi}^{-1}$ (resp. $\chi^{-1}$)-equivariant $\theta_0$-invariant smooth compactly supported functions on $\lG(F)$ (resp. $G(F)$).

\begin{theorem}
\label{thm: Apacket}

For $[\q] \in \cQ{G}$, there exists a subset $\cPkt{\lq}$ of $\clPkt{\q, \tilde{\zeta}}$ unique up to twisting by $X$, such that the following properties are satisfied.

\begin{enumerate}

\item 
For $\tilde{\bar{\e}} \in \D{\S{\lq}}$, let $\lr_{W}(\lq, \tilde{\bar{\e}})$ be the direct sum of the preimages of $\tilde{\bar{\e}}$ in $\cPkt{\lq}$ under \eqref{eq: Whittaker pairing}, then 
\begin{align}
\label{eq: restriction 1}
\lr_{W}(\lq, \tilde{\bar{\e}})|_{G(F)} = \bigoplus_{\bar{\e} \in \D{\S{\q}} : \, \bar{\e}|_{\S{\lq}} = \tilde{\bar{\e}}} \, \r_{W}(\q, \bar{\e}).
\end{align}

\item
\[
\lf_{W}(\lq) := \sum_{[\lr] \in \cPkt{\lq}} \langle s_{\lq}, \lr \rangle_{W} \, \lf_{\lG}(\lr), \quad \lf \in \sH(\lG(F), \tilde{\chi})
\]
is stable.

\item Suppose $\theta \in \Sigma_0$, semisimple $s \in \cS{\q}^{\theta}$ with $\x := \a(s)$ and $(H, \q_{H}) \rightarrow (\q, s)$, fix $\cPkt{\lq_{H}} := \otimes_{i} \Pkt{\q_{i}} \otimes \cPkt{\lq'}$ for $\q_{H} = \prod_{i} \q_{i} \times \q'$ with $\q_{i} \in \Q{GL(n_i)}$ and $\q' \in \Q{G'}$, then we can choose $\cPkt{\lq}$ such that
\begin{align}
\label{eq: theta twisted character relation}
\lf^{\tilde{H}}_{W}(\lq_{H}) = \sum_{[\lr] \in \cPkt{\lq}} \lf_{\lG^{\theta}, W}(\lr, \x), \,\,\,\,\,\, \lf \in \sH(\lG(F), \tilde{\chi})  
\end{align}
where $\lf_{\lG^{\theta}, W}(\lr, \x) = tr (\lr(\lf) \circ A_{\lr}(\theta, \x))$, and $A_{\lr}(\theta, \x)$ is an intertwining operator between $\lr \otimes \x$ and $\lr^{\theta}$, which is normalized in a way so that if $f$ is the restriction of $\lf$ on $G(F)$, then 
\begin{align}
\label{eq: theta twisted intertwining operator}
(\lf|_{\lif{Z}_{F}G(F)})_{\lG^{\theta}, W}(\lr, \x) = \sum_{\r \subseteq \lr|_{G(F)}} \langle s_{\q}s, \r^+ \rangle_{W} \, f_{G^{\theta}}(\r)  
\end{align}
where $\r^{+}$ is an extension of $\r$ to $G^{+}(F) = G(F) \rtimes \langle \theta \rangle$ and $f_{G^{\theta}}(\r) = tr(\r(f) \circ \r^{+}(\theta))$.

\end{enumerate}
\end{theorem}

Let $\lr_W(\lq, \bar{\e}) := \lr_W(\lq, \tilde{\bar{\e}})$ for $\bar{\e} \in \D{\S{\q}}$, where $\tilde{\bar{\e}} = \bar{\e}|_{\S{\lq}}$. Then
\(
\lr_W(\lq, \bar{\e}_1) = \lr_W(\lq, \bar{\e}_2)
\)
if and only if $\bar{\e}_1/ \bar{\e}_2$ is trivial on $\S{\lq}$. Moreover,
\[
\lr_{W}(\lq, \bar{\e})|_{G(F)} = \bigoplus_{\bar{\e}' \in (\S{\q}/\S{\lq})^{\wedge}} \, \r_{W}(\q, \bar{\e} \bar{\e}').
\]

We define $\Pkt{\lq}^{\Sigma_{0}}$ to be the set of irreducible representations of $\lG^{\Sigma_0}(F)$, whose restrictions to $\lG(F)$ belong to $\cPkt{\lq}$. If $\S{\lq}^{\Sigma_{0}} \neq \S{\lq}$, i.e., $\lq^{\theta_0} = \lq$, then $\lr^{\theta_{0}} \cong \lr$ for any $[\lr] \in \clPkt{\q, \tilde{\zeta}}$ by \eqref{eq: twist}. By Corollary~\ref{cor: Whittaker pairing}, we can define
\begin{align}
\label{eq: Whittaker parametrization disconnected}
\Pkt{\lq}^{\Sigma_0} \longrightarrow \D{\S{\lq}^{\Sigma_0}}, \quad \lr^{\Sigma_0} \mapsto \langle \cdot, \lr^{\Sigma_0} \rangle_{W}.
\end{align}
For $\tilde{\e} \in \D{\S{\lq}^{\Sigma_0}}$, let $\lr^{\Sigma_0}_{W}(\lq, \tilde{\e})$ be the direct sum of the preimages of $\tilde{\e}$ under \eqref{eq: Whittaker parametrization disconnected} and $\tilde{\bar{\e}} = \tilde{\e}|_{\S{\lq}}$, then
\begin{itemize}

\item
$[\lr^{\Sigma_{0}}_{W}(\lq, \tilde{\e})|_{\lG}] = 2 \lr_{W}(\lq , \tilde{\bar{\e}})$ if $G$ is special even orthogonal and $\S{\lq}^{\Sigma_{0}} = \S{\lq}$, or $\lr_{W}(\lq ,\tilde{\bar{\e}})$ otherwise.

\item for any semisimple $s \in \cS{\lq}^{\Sigma_{0}}$ but not in $\cS{\lq}$ and $(H, \q_{H}) \rightarrow (\q, s)$, the following identity holds

\begin{align*}
\lf^{\lif{H}}_{W}(\lq_{H}) = \sum_{[\lr] \in \cPkt{\lq}} \langle ss_{\lq}, \lr^{\Sigma_0} \rangle_{W} \,  \lf_{\lG}(\lr^{\Sigma_0}), \,\,\,\,\,\,\,\,\,\,\, \lf \in C^{\infty}_{c}(\lG(F) \rtimes \theta_{0}).
\end{align*}

\end{itemize}
Let $\lr^{\Sigma_0}_W(\lq, \e) := \lr^{\Sigma_0}_W(\lq, \tilde{\e})$ for $\e \in \D{\S{\q}^{\Sigma_0}}$, where $\tilde{\e} = \e|_{\S{\lq}^{\Sigma_0}}$. Then
\(
\lr^{\Sigma_0}_W(\lq, \e_1) = \lr^{\Sigma_0}_W(\lq, \e_2)
\)
if and only if $\e_1/\e_2$ is trivial on $\S{\lq}^{\Sigma_0}$. Moreover,
\[
\lr^{\Sigma_0}_{W}(\lq, \e)|_{G^{\Sigma_0}} = \bigoplus_{\e' \in (\S{\q}^{\Sigma_0}/\S{\lq}^{\Sigma_0})^{\wedge}} \, \r^{\Sigma_0}_{W}(\q, \e \e').
\]

\subsection{Proving a special case}
\label{subsec: special case}

Suppose $[\q] \in \cQ{G}$ satisifies $X^{\Sigma_0}(\tilde{\lambda}) = X^{\Sigma_0}(\lq)$. Then for any $[\r] \in \cPkt{\q}$, we have $X^{\Sigma_0}(\tilde{\lambda}) = X^{\Sigma_0}(\lp_{\r})$ by Lemma~\ref{lemma: twist}, i.e., there exists a unique $[\lr] \in \cPkt{}(\lG(F))_{\tilde{\lambda}}$ such that $[\r]$ is in the restriction of $[\lr]$ (cf. Section~\ref{sec: infinitesimal character}). So we can define
\[
\cPkt{\lq} := \Big\{[\lr] \in \cPkt{}(\lG(F))_{\tilde{\lambda}} \, : \, [\lr]|_{G(F)} \subseteq \cPkt{\q} \Big\}.
\]

\begin{theorem}
\label{thm: special case}
For $[\q] \in \cQ{G}$ such that $X^{\Sigma_0}(\tilde{\lambda}) = X^{\Sigma_0}(\lq)$, $\cPkt{\lq}$ satisfies (1), (2) in Theorem~\ref{thm: Apacket}. 
\end{theorem}

\begin{proof}
Since $X^{\Sigma_0}(\tilde{\lambda}) = X^{\Sigma_0}(\lq)$, the elements in $\cPkt{\lq}$ can be distinguished from $\cPkt{}(\lG(F))_{\tilde{\lambda}}$ by their restrictions to $G(F)$. Let us write
\[
f_{W}(\q) = \sum_{[\p] \in \cP{G}_{\lambda}} m([\q], [\p]) f(\p), \quad f \in \sH(G(F), \chi),
\]
where $\p$ factors through $\p_{M} \otimes \xi$ for $[\p_{M}] \in \cPbd{M}$ and ${\rm Re} \, \xi \in (\mathfrak{a}^{*}_{P})^{+}$,
\[
f(\p) = \sum_{[\sigma] \in \cPkt{\p_M}} f_{G}( {\rm Ind}^{G(F)}_{P(F)}\, \sigma \otimes \xi )
\]
(cf. \cite[(2.2.12)]{Arthur:2013}). We form a stable distribution on $\lG(F)$ by
\begin{align}
\label{eq: stable expansion 1}
\sum_{[\lp] \in \cP{\lG}_{\tilde{\lambda}}} m([\lq], [\lp]) \lf(\lp), \quad \lf \in \sH(\lG(F), \tilde{\chi})
\end{align}
where
\(
m([\lq], [\lp]) = |X^{\Sigma_0}(\tilde{\lambda})/X^{\Sigma_0}(\lp)|^{-1} m([\q], [\p])
\)
and $\lp$ factors through $\lp_{M} \otimes \xi$ for $[\lp_{M}] \in \cPbd{\lif{M}}$,
\[
\lf(\lp) = \sum_{[\tilde{\sigma}] \in \cPkt{\lp_{M}}} \lf_{\lG}( {\rm Ind}^{\lG(F)}_{\lif{P}(F)} \, \tilde{\sigma} \otimes \xi).
\]
It is clear that the restriction of \eqref{eq: stable expansion 1} to $\lif{Z}_{F}G(F)$ is $f_{W}(\q)$. We would like to show that it is equal to $\lf_{W}(\lq)$.
By Proposition~\ref{prop: induction preserves infinitesimal character}, the irreducible constituents of ${\rm Ind}^{\lG(F)}_{\lif{P}(F)} \, \tilde{\sigma} \otimes \xi$ for $[\tilde{\sigma}] \in \cPkt{\lp_{\lM}}$ and $[\p] \in \cP{G}_{\lambda}$ all belong to $\cPkt{}(\lG(F))_{\tilde{\lambda}}$. Since \eqref{eq: stable expansion 1} is invariant under $X^{\Sigma_0}(\tilde{\lambda})$, then for any $[\lr] \in \cPkt{}(\lG(F))_{\tilde{\lambda}}$ that contributes to $\eqref{eq: stable expansion 1}$, $[\lr \otimes \omega]$ must have the same coefficients in \eqref{eq: stable expansion 1} for all $\omega \in X^{\Sigma_0}(\tilde{\lambda})$. So their restrictions to $G(F)$ do not cancel. This means that $[\lr] \in \cPkt{\lq}$. The rest is clear by considering the restriction of \eqref{eq: stable expansion 1} to $\lif{Z}_{F}G(F)$ again.

\end{proof}

\begin{theorem}
\label{thm: special case 1}
Suppose $[\q] \in \cQ{G}$ satisifies $X^{\Sigma_0}(\tilde{\lambda}) = X^{\Sigma_0}(\lq)$. For $\theta \in \Sigma_0$, semisimple $s \in \cS{\q}^{\theta}$ with $\x := \a(s)$ and $(H, \q_{H}) \rightarrow (\q, s)$, let $\lambda_{H} = \lambda_{\q_{H}}$ and choose $\tilde{\lambda}_{H}$ such that $\tilde{\lambda} = \iota_{\lif{H}} \circ \tilde{\lambda}_{H}$, where $\iota_{\lif{H}}$ is the twisted endoscopic embedding. If $X^{\Sigma_0}(\lq_{H}) = X^{\Sigma_0}(\tilde{\lambda}_{H})$, then \eqref{eq: theta twisted character relation} holds.
\end{theorem}

\begin{proof}
Let us write
\[
f^{H}_{W}(\q_{H}) = \sum_{\p_{H} \in \P{H}_{\lambda_{H}}} m(\p_{H}, \q_{H}) f^{H}(\p_{H}), \quad f \in \sH(G(F), \chi),
\]
where $f^{H}$ is the transfer of $f$, then we can show as in Theorem~\ref{thm: special case} that
\[
\lf^{\lif{H}}_{W}(\lq_{H}) = \sum_{\lp_{H} \in \cP{\lif{H}}_{\tilde{\lambda}_{H}}} m(\lq_{H}, \lp_{H}) \lf^{\lif{H}}(\lp_{H}), \quad f \in \sH(\lG(F), \tilde{\chi}),
\]
where 
\(
m(\lq_H, \lp_H) = |X(\tilde{\lambda}_H)/X(\lp_H)|^{-1} m(\q_H, \p_H)
\)
and $\lf^{\lif{H}}$ is the transfer of $\lf$. Note the restriction of $\lf^{\lif{H}}_{W}(\lq_{H})$ to $\lif{Z}_{F}G(F)$ is $f^{H}_{W}(\q_{H})$. Since $\lf^{\lif{H}}_{W}(\lq_{H})$ is invariant under $X^{\Sigma_0}(\tilde{\lambda})$, we can further show as before that it is supported on $\cPkt{\lq}$. Then the rest is clear by considering the restriction of $\lf^{\lif{H}}_{W}(\lq_{H})$ to $\lif{Z}_{F}G(F)$ again.
\end{proof}

\begin{remark}
In case that $X^{\Sigma_{0}}(\tilde{\lambda}) = 1$, the condition in Theorem~\ref{thm: special case 1} is always satisfied.
\end{remark}

\section{Construction}
\label{sec: construction}
 
M{\oe}glin \cite{Moeglin:2006} \cite{Moeglin:2009} gives a construction of Arthur packets for classical groups, which reduces to the tempered case. Since we have already constructed the tempered packets of $\lG^{\Sigma_0}(F)$ in \cite{Xu:2018}, the idea is to extend M{\oe}glin's construction to $\lG^{\Sigma_0}(F)$. First we will give a combinatorial description of Arthur parameters (cf. Subsection~\ref{subsec: Arthur parameters}), from which we will divide our construction into three cases: {\bf elementary case}, {\bf case of discrete diagonal restriction} and the {\bf general case}. The elementary case (cf. Subsection~\ref{subsec: elementary case}) will follow from the special case treated in Subsection~\ref{subsec: special case}. For later purpose we will also relate this case with the case of discrete $L$-packets by a generalized version of Aubert involution following an idea of M{\oe}glin \cite{Moeglin:2006}. The case of discrete diagonal restriction (cf. Subsection~\ref{subsec: DDR}) is the most difficult one and a more detailed outline will be given in the beginning of that subsection. At last, the general case (cf. Subsection~\ref{subsec: general case}) can be easily reduced to the case of discrete diagonal restriction by using Jacquet modules.

\subsection{Combinatorial description of Arthur parameters}
\label{subsec: Arthur parameters}

Let $N = 2n + 1$ if $G = Sp(2n)$ and $N = 2n$ if $G = SO(2n, \eta_{E/F})$. Let $\x_0$ be the character of $\lG^{\Sigma_0}(F)/\lG(F) \cong G^{\Sigma_0}(F)/G(F)$. For $[\q] \in \cQ{G}$, we get an equivalence class of $N$-dimensional self-dual representation of $W_{F} \times SL(2, \mathbb{C}) \times SL(2, \mathbb{C})$ by composing with a $GL(N, \mathbb{C})$-conjugacy class of twisted endoscopic embedding $\iota_{G}: \L{G} \longrightarrow GL(N, \mathbb{C})$ such that
\begin{itemize}
\item $\iota_{G}|_{\D{G}}$ is the standard representation of $\D{G}$,
\item $\iota_{G}|_{W_{F}}$ is trivial if $N$ is odd, and factors through $\Gal{E/F}$ with the nontrivial element sent to a reflection if $N$ is even.
\end{itemize}
Then we can decompose $\q$ into a direct sum of irreducible subrepresentations
\begin{align}
\label{eq: Arthur parameter}
\q = \bigoplus_{i=1}^{r} l_{i} \q_{i} = \bigoplus_{i=1}^{r}l_{i}(\rho_{i} \otimes \nu_{a_{i}} \otimes \nu_{b_{i}}).
\end{align}
Here $\rho_{i}$ are equivalence classes of irreducible unitary representations of $W_{F}$ and $\nu_{a_{i}}$ (resp. $\nu_{b_{i}}$) are the $(a_{i}-1)$-th (resp. $(b_{i}-1)$-th) symmetric power representations of $SL(2, \mathbb{C})$. The irreducible constituent $\rho_{i} \otimes \nu_{a_{i}} \otimes \nu_{b_{i}}$ has dimension $n_{i} = n_{(\rho_{i}, a_{i}, b_{i})}$ and multiplicity $l_{i}$. Let $\eta_i = \eta_{\rho_i}^{a_ib_i}$. We define the multi-set of Jordan blocks for $\q$ as follows,
\[
Jord(\q) := \{(\rho_{i}, a_{i}, b_{i}) \text{ with multiplicity } l_{i}: 1 \leqslant i \leqslant r \}.
\]
Moreover, for any  $\rho$ let us define
\[
Jord_{\rho}(\q) := \{(\rho', a', b') \in Jord(\q): \rho' = \rho\}.
\]
One can define the parity for self-dual irreducible unitary representations $\rho$ of $W_{F}$ as in (\cite{Xu:cusp}, Section 3). Then we say $(\rho_{i}, a_{i}, b_{i}) $ is of {\bf orthogonal type} if $a_{i} + b_{i}$ is even when $\rho_{i}$ is of orthogonal type, and $a_{i} + b_{i}$ is odd when $\rho_{i}$ is of symplectic type. Similarly we say $(\rho_{i}, a_{i}, b_{i})$ is of {\bf symplectic type} if $a_{i} + b_{i}$ is odd when $\rho_{i}$ is of orthogonal type, and $a_{i} + b_{i}$ is even when $\rho_{i}$ is of symplectic type. Let $\q_{p}$ be the parameter whose Jordan blocks consist of those in $Jord(\q)$ with the same parity as $\D{G}$, and 
let $\q_{np}$ be any parameter of general linear groups such that 
\[
\q = \q_{np} \+ \q_{p} \+ \q_{np}^{\vee},
\]
where $\q_{np}^{\vee}$ is the dual of $\q_{np}$. We also denote by $Jord(\q)_{p}$ the set of Jordan blocks in $Jord(\q_{p})$ without multiplicity. Let $s_{0} = (s_{0,i}) \in \Two^{Jord(\q)_{p}}$ be defined as $s_{0,i} = 1$ if $l_{i}$ is even and $s_{0,i} = -1$ if $l_{i}$ is odd. Then
\[
\S{\q}^{\Sigma_{0}} \cong \{s  = (s_{i}) \in \Two^{Jord(\q)_{p}} \} / \langle s_{0} \rangle,
\]
and
\[
\S{\q} \cong \{ s  = (s_{i}) \in \Two^{Jord(\q)_{p}} : \prod_{i}(s_{i})^{n_{i}} = 1\} / \langle s_{0} \rangle
\] 
if $G$ is special even orthogonal. Under this interpretation, one can show
\begin{align}
\label{eq: combinatorial description}
\alpha: \S{\q}^{\Sigma_{0}} \rightarrow {\rm Hom}(\lG(F)/G(F), \mathbb{C}^{\times}), \quad s \mapsto (\prod_{i : s_i = -1} \eta_i) \circ \lambda_{\lG}
\end{align}
as in \cite[Lemma 6.9]{Xu:2018}).

There is a natural inner product on $\Two^{Jord(\q)_{p}}$ which identify its dual with itself. Let $\e = (\e_{i})$ and $s = (s_{i})$ be two elements in $\Two^{Jord(\q)_{p}}$, then their inner product is defined by $\e(s) = \prod_{i}(\e_{i} \ast s_{i})$, where
\[
\e_{i} \ast s_{i} = \begin{cases}
                           -1, & \text{ if } \e_{i} = s_{i} = -1, \\
                           1,  & \text{ otherwise. } 
                           \end{cases}
\]
So on the dual side, 
\[
\D{\S{\q}^{\Sigma_{0}}} \cong \{\e = (\e_{i}) \in \Two^{Jord(\q)_{p}}: \prod_{i} \e_{i}^{l_{i}} = 1\}.
\]
When $G$ is special even orthogonal, let $\e_{0} = (\e_{0,i}) \in \Two^{Jord(\q)_{p}}$ be defined as $\e_{0,i} = 1$ if $n_{i}$ is even, or $\e_{0,i} = -1$ if $n_{i}$ is odd, then $\e_{0} \in \D{\S{\q}^{\Sigma_{0}}}$ is always trivial when restricted to $\S{\q}$, and  
\[
\D{\S{\q}} \cong \{\e = (\e_{i}) \in \Two^{Jord(\q)_{p}}: \prod_{i} \e_{i}^{l_{i}} = 1\} / \langle \e_{0} \rangle.
\]
In general, we can let $\e_{0} = 1$ if $G$ is not special even orthogonal. 

There is a natural projection  
\begin{align}
\label{eq: projection}
\xymatrix{ \Two^{Jord(\q_{p})} \ar[rr]^{Cont} && \Two^{Jord(\q)_{p}} \\
                 s \ar@{|->}[rr] && s'}
\end{align}
such that 
\[
s'(\rho, a, b) = \prod_{\substack{(\rho', a', b') \in Jord(\q_{p}) \\ (\rho', a', b') = (\rho, a, b) \text{ in } Jord(\q)_{p} }} s(\rho', a', b')
\]
for $(\rho, a, b) \in Jord(\q)_{p}$. In particular, $s_{0}$ has a natural representative $s_{0}^{>}$ in $\Two^{Jord(\q_{p})}$ given by $s_{0}^{>}(\rho, a, b) = -1$ for all $(\rho, a, b) \in Jord(\q_{p})$. We define
\[
\S{\q^{>}}^{\Sigma_{0}} = \{s(\cdot) \in \Two^{Jord(\q_{p})} \} / \langle s_{0}^{>} \rangle,
\]
and 
\[
\S{\q^{>}} = \{ s(\cdot) \in \Two^{Jord(\q_{p})} : \prod_{(\rho, a, b) \in Jord(\q_{p})} s(\rho, a, b)^{n_{(\rho, a, b)}} = 1\} / \langle s_{0}^{>} \rangle
\] 
if $G$ is special even orthogonal. Then there are surjections $\S{\q^{>}}^{\Sigma_{0}} \rightarrow \S{\q}^{\Sigma_{0}}$ and $\S{\q^{>}} \rightarrow \S{\q}$.

On the dual side, we have a natural inclusion
\[
\xymatrix{ \Two^{Jord(\q)_{p}} \ar@{^{(}->}[rr]^{Ext} && \Two^{Jord(\q_{p})} \\
                 \e \ar@{|->}[rr] && \e'}
\]
such that 
\[
\e'(\rho, a, b) = \e(\rho, a, b)
\]
for $(\rho, a, b) \in Jord(\q_{p})$. We define an inner product on $\Two^{Jord(\q_{p})}$ as for $\Two^{Jord(\q)_{p}}$. Then
\[
\D{\S{\q^{>}}^{\Sigma_{0}}} \cong \{\e(\cdot) \in \Two^{Jord(\q_{p})} : \prod_{(\rho, a, b) \in Jord(\q_{p})} \e(\rho, a, b) = 1\}, 
\]
and 
\[
\D{\S{\q^{>}}} \cong \{\e(\cdot) \in \Two^{Jord(\q_{p})} : \prod_{(\rho, a, b) \in Jord(\q_{p})} \e(\rho, a, b) = 1\} / \langle \e_{0} \rangle
\] 
if $G$ is special even orthogonal. 
There are inclusions $\D{\S{\q}^{\Sigma_{0}}} \hookrightarrow \D{\S{\q^{>}}^{\Sigma_{0}}}$ and $\D{\S{\q}} \hookrightarrow \D{\S{\q^{>}}}$. For $\e \in \D{\S{\q^{>}}^{\Sigma_{0}}}$, we denote its image in $\D{\S{\q^{>}}}$ by $\bar{\e}$. 

Suppose $\q = \q_{p}$. For $(\rho, a, b) \in Jord(\q)$, let us write $A = (a+b)/2-1$, $B = |a-b|/2$, and set $\zeta = \zeta_{a, b} = \text{Sign}(a - b)$ if $a \neq b$ and arbitrary otherwise. Then we can denote $(\rho, a, b)$ also by $(\rho, A, B, \zeta)$. We would like to impose some total order $>_{\q}$ on $Jord_{\rho}(\q)$ for each $\rho$. We say $>_{\q}$ is ``admissible" if it satisfies
\begin{align*}
(\mathcal{P}): \quad & \forall (\rho, A, B, \zeta), (\rho, A', B', \zeta') \in Jord(\q) \text{ with } A > A', B > B' \text{ and } \zeta = \zeta', \\
& \text{ then } (\rho, A, B, \zeta) >_{\q} (\rho, A', B', \zeta').
\end{align*}
For a fixed admissible order $>_{\q}$, we have introduced $\e^{MW/W}_{\q}, \e^{M/MW}_{\q} \in \D{\S{\q^{>}}^{\Sigma_{0}}}$ in \cite{Xu:Apacket}. For $[\r] \in \cPkt{\q}$ and $s \in \D{\S{\q^{>}}^{\Sigma_{0}}}$, we define 
\[
\langle s, \r \rangle_{MW} := \e^{MW/W}_{\q}(s) \langle s, \r \rangle_{W}, \quad \langle s, \r \rangle_{M} := \e^{M/MW}_{\q}(s) \langle s, \r \rangle_{MW}.
\]

For $s \in \S{\q^{>}}^{\Sigma_0}$, it gives a partition of  
\[
Jord(\q) = Jord_{+} \sqcup Jord_{-}.
\]
depending on $s(\rho, a, b) = 1$ or $-1$. Let
\[
N_{I} = \sum_{(\rho, a, b) \in Jord_{+}} n_{(\rho, a, b)}, \quad N_{II} = \sum_{(\rho, a, b) \in Jord_{-}} n_{(\rho, a, b)}
\]
and
\[
\eta_{I} = \prod_{(\rho, a, b) \in Jord_{+}}\eta_{(\rho, a, b)}, \quad \eta_{II} = \prod_{(\rho, a, b) \in Jord_{-}}\eta_{(\rho, a, b)}.
\]
Then $\q$ factors through $\q_{H} := \q_{I} \times \q_{II} \in \Q{H}$ for a twisted elliptic endoscopic group $H = G_{I} \times G_{II}$ of $G$, where 
\[
\D{G}_{I} = SO(N_I), \quad \D{G}_{II} = SO(N_{II})
\]
and $\q_{I} \in \Q{G_{I}}, \q_{II} \in \Q{G_{II}}$ such that
\[
Jord(\q_{I}) = \begin{cases} \{(\rho \otimes \eta_{I}, a, b) :  (\rho, a, b) \in Jord_{I}\} & \text{ if $N_{I}$ is odd } \\
 Jord_{I}  & \text{ if $N_{I}$ is even }
 \end{cases}
\]
\[
Jord(\q_{II}) = \begin{cases} \{(\rho \otimes \eta_{II}, a, b) :  (\rho, a, b) \in Jord_{II}\} & \text{ if $N_{II}$ is odd } \\
 Jord_{II}  & \text{ if $N_{II}$ is even }
 \end{cases}
\]
Comparing with \eqref{eq: endoscopy}, this amounts to choosing a special representative of $s$ in $\cS{\q}^{\Sigma_0}$. So we will say that $(H, \q_{H}) \rightarrow (\q, s)$.



\subsection{Elementary case}
\label{subsec: elementary case}

Suppose $[\q] \in \cQ{G}$ is elementary, i.e., $[\q_d] \in \cPdt{G}$ and $A = B$ for all $(\rho, A, B, \zeta) \in Jord(\q)$ (cf. \cite[Section 6]{Xu:Apacket}). For simplicity, we denote by $Jord_{\rho}(\q_{d})$ the set of integers $\alpha$ such that $(\rho, \alpha, 1) \in Jord(\q_{d})$, and we write $(\rho, \alpha, \delta_{\alpha})$ for $(\rho, (\alpha-1)/2, (\alpha-1)/2, \delta_{\alpha}) \in Jord(\q)$. We always impose the natural order $>_{\q}$, i.e., $(\rho, \alpha, \delta_{\alpha}) >_{\q} (\rho, \alpha', \delta_{\alpha'})$ if $\alpha > \alpha'$. Let $\lambda = \lambda_{\q}$. In this case, we have
\[
X(\lq) = X(\lq_d) = X(\tilde{\lambda}) \quad \text{ and } \quad X^{\Sigma_0}(\lq) = X^{\Sigma_0}(\lq_d) = X^{\Sigma_0}(\tilde{\lambda}).
\]
by Lemma~\ref{lemma: twist invariant} and \eqref{eq: combinatorial description}. So this case has been covered in Theorem~\ref{thm: special case}. Theorerm~\ref{thm: special case 1} also applies to this case for $\q_{H} = \q_{I} \times \q_{II}$, where $\q_I, \q_{II}$ are both elementary. Moreover, 
\[
\cPkt{\q} \longrightarrow \D{\S{\q}} \quad \text{ and } \quad  \Pkt{\q}^{\Sigma_0} \longrightarrow \D{\S{\q}^{\Sigma_0}}
\]
are bijections (cf. \cite[Section 6.4]{Xu:Apacket}). Let us define 
\begin{align*}
& \r^{\Sigma_0}_{MW}(\q, \e) := \r^{\Sigma_0}_{W}(\q, \e \e^{MW/W}_{\q}), \quad \r^{\Sigma_0}_{M}(\q, \e) := \r^{\Sigma_0}_{MW}(\q, \e \e^{M/MW}_{\q}), \\
& \lr^{\Sigma_0}_{MW}(\lq, \e) := \lr^{\Sigma_0}_{W}(\lq, \e \e^{MW/W}), \quad \lr^{\Sigma_0}_{M}(\lq, \e) := \lr^{\Sigma_0}_{MW}(\lq, \e\e^{M/MW}).
\end{align*}
Similarly, we define $\r_{MW}(\q, \bar{\e}), \r_{MW}(\q, \bar{\e}), \lr_{MW}(\lq, \bar{\e}), \lr_{MW}(\lq, \bar{\e})$.

M{\oe}glin constructs $\Pkt{\q}^{\Sigma_0}$ from the discrete $L$-packet $\Pkt{\q_{d}}^{\Sigma_0}$ by a generalized version of Aubert involution in \cite{Moeglin:2006}. Let $\rho$ be a self-dual irreducible unitary supercuspidal representation of $GL(d_{\rho}, F)$ and $\mathcal{P}^{\Sigma_{0}}_{d_{\rho}}$ the set of $\Sigma_{0}$-conjugacy classes of standard parabolic subgroups $P$ of $G$, whose Levi component $M$ is of the form 
\begin{align}
\label{eq: Aubert dual for G Levi}
GL(a_{1}d_{\rho}) \times \cdots \times GL(a_{l}d_{\rho}) \times G_{-},
\end{align}
where $a_{l}d_{\rho} \neq 1$ if $G = SO(2n)$ and $G_{-} = 1$. For $P \in \mathcal{P}^{\Sigma_{0}}_{d_{\rho}}$, $\sigma^{\Sigma_{0}} \in \Rep(M^{\Sigma_{0}}(F))$ and $x_0 \in \mathbb{R}_{\geqslant 0}$, we denote by $(\sigma^{\Sigma_{0}})^{\rho}_{< x_{0}}$ the direct sum of  irreducible constitutes of $\sigma^{\Sigma_0}$ whose cuspidal support on the general linear factors consist only of $\rho||^{x}$ with $|x| < x_{0}$. If $G_{-} = SO(2)$, we further define ${}^1(\sigma^{\Sigma_0})^{\rho}_{< x_0}$ to be the direct summand of $(\sigma^{\Sigma_0})^{\rho}_{< x_0}$, whose irreducible constituents are isomorphic to $\rho||^{x'}$ on $SO(2) \cong GL(1)$ for $|x'| < x_0$, and ${}^2(\sigma^{\Sigma_0})^{\rho}_{< x_0}$ to be its complement. Define
\[
\widetilde{\Jac}_{P^{\Sigma_{0}}(F)}(\r^{\Sigma_{0}})^{\rho}_{< x_0} = \begin{cases}
                                                         {}^1\Jac_{P^{\Sigma_{0}}(F)}(\r^{\Sigma_{0}})^{\rho}_{< x_0} \otimes \x_{0} - {}^2\Jac_{P^{\Sigma_{0}}(F)}(\r^{\Sigma_{0}})^{\rho}_{< x_0} ,  & \text{ if } G_{-} = SO(2), \\
                                                         \Jac_{P^{\Sigma_{0}}(F)}(\r^{\Sigma_{0}})^{\rho}_{< x_0}, & \text{ otherwise. }  
                                                         \end{cases}
\]
For $\r^{\Sigma_0} \in \Rep(G^{\Sigma_{0}}(F)), X_0 \in \mathbb{N}, x_0 = (X_0 - 1)/2$, define 
\[
{\rm inv}^{\rho}_{< X_{0}}(\r^{\Sigma_{0}}) := \sum_{P \in \mathcal{P}^{\Sigma_{0}}_{d_{\rho}}} (-1)^{{\rm dim} \, A_{M}} \Ind^{G^{\Sigma_{0}}(F)}_{P^{\Sigma_{0}}(F)}(\widetilde{\Jac}_{P^{\Sigma_{0}}(F)} (\r^{\Sigma_{0}})^{\rho}_{< x_{0}}).
\] 

\begin{remark}
The term ${}^2\Jac_{P^{\Sigma_{0}}(F)}(\r^{\Sigma_{0}})^{\rho}_{< x_0}$ is missing in the definition in \cite{Xu:Apacket}. That is incorrect for it would fail to be an involution (cf. \cite[Proposition 6.5]{Xu:Apacket}). However, the results are unchanged when applying to elementary Arthur packets.
\end{remark}

Define ${\rm inv}^{\rho}_{\leqslant X_0} (\r^{\Sigma_0})$ similarly. Let $\q^{\sharp}$ be obtained from $\q$ by changing $\delta_{\alpha}$ to $-\delta_{\alpha}$ for all $\alpha \in Jord_{\rho}(\q)$ such that $\alpha < X_0$. M{\oe}glin \cite[Theorem 5]{Moeglin:2006} showed that
\[
\e(s_\q) \beta(\q, \rho, < X_0) \, {\rm inv}^{\rho}_{< X_0} \, \r^{\Sigma_0}_M(\q, \e) = \e(s_{\q^{\sharp}}) \r^{\Sigma_0}_M(\q^{\sharp}, \e),
\]
where $\beta(\q, \rho, < X_{0})$ is some sign (cf. \cite[Section 6.2]{Xu:Apacket}). By taking $|{\rm inv}^{\rho}_{< X_0}|$ (resp. $|{\rm inv}^{\rho}_{\leqslant X_0}|$), we forget the sign. Then
\begin{align}
\label{eq: elementary classical}
\r^{\Sigma_0}_{M}(\q, \e) = \circ_{\substack{(\rho, a, \delta_a) \in Jord(\q) \\ \delta_a = -1}} (|{\rm inv}^{\rho}_{< a}| \circ |{\rm inv}^{\rho}_{\leqslant a}|) \, \r^{\Sigma_0}_W(\q_d, \e).
\end{align}

Next we would like to extend this result to $\lG^{\Sigma_0}(F)$. For $\lr^{\Sigma_0} \in \Rep(\lG^{\Sigma_{0}}(F)), X_0 \in \mathbb{N}, x_0 = (X_0 - 1)/2$, define 
\[
{\rm inv}^{\rho}_{< X_0} (\lr^{\Sigma_0}) := \sum_{P \in \mathcal{P}^{\Sigma_0}_{d_{\rho}}} (-1)^{{\rm dim} \, A_M} {\rm Ind}^{\lG^{\Sigma_0}(F)}_{\lP^{\Sigma_0}(F)} \, (\widetilde{\Jac}_{\lP^{\Sigma_0}(F)} (\lr^{\Sigma_0})^{\rho}_{< x_{0}}),
\]
where $\widetilde{{\Jac}}_{\lP^{\Sigma_0}(F)} (\lr^{\Sigma_0})^{\rho}_{< x_0}$ is defined with respect to the restriction to $M^{\Sigma_0}(F)$. We also define ${\rm inv}^{\rho}_{ \leqslant X_0} (\lr^{\Sigma_0})$ similarly.
Denote by ${\rm inv}^{\rho}_{\infty}$ if $X_0$ is taken to be infinity. Note ${\rm inv}^{\rho}_{< X_0}$ preserves the $\Sigma_0$-infinitesimal characters by Corollary~\ref{cor: cuspidal support} and Corollary~\ref{cor: Jac preserves infinitesimal character},  

\begin{lemma}
\label{lemma: involution}
\[
\e(s_\q) \beta(\q, \rho, < X_0) \, {\rm inv}^{\rho}_{< X_0} \, \lr^{\Sigma_0}_M(\lq, \e) =  \e(s_{\q^{\sharp}}) \lr^{\Sigma_0}_M(\lq^{\sharp}, \e).
\]
\end{lemma}

\begin{proof}
Since ${\rm inv}^{\rho}_{< X_0}$ preserves the $\Sigma_0$-infinitesimal characters, it suffices to show that any irreducible constituent $\lr^{\Sigma_0}$ in ${\rm Ind}^{\lG^{\Sigma_0}(F)}_{\lP^{\Sigma_0}(F)} \, (\widetilde{{\rm Jac}}_{\lP^{\Sigma_0}(F)} \lr^{\Sigma_0}_M(\lq, \tilde{\e}))^{\rho}_{< x_0}$ such that $X(\lr^{\Sigma_0}) \subsetneq X^{\Sigma_0}(\tilde{\lambda})$ is cancelled in ${\rm inv}^{\rho}_{< X_0} \, \lr^{\Sigma_0}_M(\lq, \e)$. If not, then it follows from
\[
\Big\{\omega \in X \, | \, {\rm inv}^{\rho}_{< X_0} \, \lr^{\Sigma_0}_M(\lq, \e) \otimes \omega = {\rm inv}^{\rho}_{< X_0} \, \lr^{\Sigma_0}_M(\lq, \e) \Big\} = X^{\Sigma_0}(\tilde{\lambda})
\]
that the restriction of $\lr^{\Sigma_0}$ to $G^{\Sigma_0}(F)$ is not cancelled in
\begin{align*}
& (\e(s_\q) \beta(\q, \rho, < X_0) \, {\rm inv}^{\rho}_{< X_0} \, \lr^{\Sigma_0}_M(\lq, \e))|_{G^{\Sigma_0}(F)} = \e(s_\q) \beta(\q, \rho, < X_0) \, {\rm inv}^{\rho}_{< X_0} \, (\lr^{\Sigma_0}_M(\lq, \e)|_{G^{\Sigma_0}}(F)) \\
= & \sum_{\e' \in (\S{\q}^{\Sigma_0}/\S{\lq}^{\Sigma_0})^{\wedge}}  \e(s_\q) \beta(\q, \rho, < X_0) \,  {\rm inv}^{\rho}_{< X_0} \, \r^{\Sigma_0}_M(\q, \e\e') = \sum_{\e' \in (\S{\q^{\sharp}}^{\Sigma_0}/\S{\lq^{\sharp}}^{\Sigma_0})^{\wedge}}  \e(s_{\q^{\sharp}}) \r^{\Sigma_0}_M(\q^{\sharp}, \e\e').
\end{align*}
Here we have used the fact that $s_{\q} = s_{\lq}$ and $s_{\q^{\sharp}} = s_{\lq^{\sharp}}$. Since $X^{\Sigma_0}(\lq^{\sharp}) = X^{\Sigma_0}(\tilde{\lambda})$, then the irreducible constituents of $\lr^{\Sigma_0}|_{G^{\Sigma_0}(F)}$ can not belong to $\r^{\Sigma_0}_M(\q^{\sharp}, \e\e')$. So we get a contradiction.
\end{proof}

As a consequence, we have
\begin{corollary}
\begin{align}
\label{eq: involution}
\lr^{\Sigma_0}_{M}(\lq, \e) = \circ_{\substack{(\rho, a, \delta_a) \in Jord(\q) \\ \delta_a = -1}} (|{\rm inv}^{\rho}_{<a}| \circ |{\rm inv}^{\rho}_{\leqslant a}|) \, \lr^{\Sigma_0}_W(\lq_d, \e).
\end{align}
\end{corollary}

For $[\lr] \in \overline{\Rep}(\lG(F)), X_0 \in \mathbb{N}, x_0 = (X_0 - 1)/2$, we define 
\[
\overline{\rm inv}^{\rho}_{< X_0} ([\lr]) := \sum_{P \in \mathcal{P}^{\Sigma_0}_{d_{\rho}}} (-1)^{{\rm dim} \, A_M} {\rm Ind}^{\lG(F)}_{\lP(F)} \, (\widetilde{\overline{\Jac}}_{\lP(F)} ([\lr])^{\rho}_{< x_{0}}),
\]
where 
\[
\widetilde{\overline{\Jac}}_{\lP(F)}([\lr])^{\rho}_{< x_0} = \begin{cases}
                                                         {}^1 \overline{\Jac}_{\lP(F)}([\lr])^{\rho}_{< x_0} - {}^2 \overline{\Jac}_{\lP(F)}([\lr])^{\rho}_{< x_0} ,  & \text{ if } G_{-} = SO(2), \\
                                                         \overline{\Jac}_{\lP(F)}([\lr])^{\rho}_{< x_0}, & \text{ otherwise. }  
                                                         \end{cases}
\]
So for $\lr^{\Sigma_0} \in \Rep(\lG^{\Sigma_0}(F))$,
\[
[{\rm inv}^{\rho}_{< X_0} (\lr^{\Sigma_0})|_{G}] = \overline{\rm inv}^{\rho}_{< X_0} ([\lr^{\Sigma_0}|_{G}]).
\]
It follows from \eqref{eq: involution} that
\[
\lr_{M}(\lq, \bar{\e}) = \circ_{\substack{(\rho, a, \delta_a) \in Jord(\q) \\ \delta_a = -1}} (|\overline{\rm inv}^{\rho}_{<a}| \circ |\overline{\rm inv}^{\rho}_{\leqslant a}|) \, \lr_W(\lq_d, \bar{\e}).
\]


By using the same argument as for the Aubert involution in \cite{Aubert:1995}, one can show that ${\rm inv}^{\rho}_{\infty}$ is an involution on the Grothendieck group of ${\rm Rep}(\lG^{\Sigma_0}(F))$ and preserves irreducible representations up to signs. Moreover, for $\lr^{\Sigma_0} \in {\rm Rep}(\lG^{\Sigma_0}(F))$ and self-dual irreducible unitary supercuspidal representations $\rho$, $\rho'$ of $GL(d_{\rho}, F)$, $GL(d_{\rho'}, F)$ respectively and $x \in \mathbb{R}_{\geqslant 0}$, one can show
\begin{align}
\label{eq: involution and Jac}
{\rm Jac}^{\rho'}_{x} \circ {\rm inv}^{\rho}_{\infty} (\lr^{\Sigma_0}) = \begin{cases} - \, {\rm inv}^{\rho}_{\infty} \circ {\rm Jac}^{\rho'}_{-x} (\lr^{\Sigma_0}) \otimes \omega^{-x}_{\rho} \eta_{\rho} \quad & \text{ if $\rho' = \rho$,} \\
{\rm inv}^{\rho}_{\infty} \circ {\rm Jac}^{\rho'}_{x} (\lr^{\Sigma_0}) \quad & \text{ if $\rho' \neq \rho$,}
\end{cases}
\end{align}
which is a consequence of the analogue of \cite[Theorem 1.7 (2)]{Aubert:1995}. Hence,
\begin{align}
\label{eq: involution and Jac 1}
\overline{\rm Jac}^{\rho'}_{x} \circ \overline{\rm inv}^{\rho}_{\infty} ([\lr]) = \begin{cases} - \, \overline{\rm inv}^{\rho}_{\infty} \circ \overline{\rm Jac}^{\rho'}_{-x} ([\lr]) \otimes \omega^{-x}_{\rho} \eta_{\rho} \quad & \text{ if $\rho' = \rho$,} \\
\overline{\rm inv}^{\rho}_{\infty} \circ \overline{\rm Jac}^{\rho'}_{x} ([\lr]) \quad & \text{ if $\rho' \neq \rho.$}
\end{cases}
\end{align}

\subsection{Case of discrete diagonal restriction}
\label{subsec: DDR}

We say $[\q] \in \cQ{G}$ has discrete diagonal restriction if $[\q_d] \in \cPdt{G}$, i.e.,  $\q = \q_{p}$ and for any $\rho$ the intervals $[B, A], [B', A']$ do not intersect for any $(\rho, A, B, \zeta), (\rho, A', B', \zeta')$ in $Jord_{\rho}(\q)$. We will give a recursive formula for elements in $\Pkt{\lq}^{\Sigma_0}$ (\eqref{eq: case 1}, \eqref{eq: case 2}). The key point of this formula is to reduce 
\[
\sum_{(\rho, A, B, \zeta) \in Jord(\q)} (A - B).
\]
When the sum is zero, we are back to the elementary case. The advantage of this formula is that it is easy to show stability and twisted character relation (Theorem~\ref{thm: discrete endoscopy}). The challenge is to show that the formula does give a representation instead of a virtual representation, which is also the main theorem of this subsection (Theorem~\ref{thm: discrete diagonal restriction}). The basic idea of the proof is to find the desired representations in the formula (Proposition~\ref{prop: existence}), and show the other representations in the formula are all cancelled. The hard part is to show the cancellation. One idea is to characterize the irreducible representations by their Jacquet modules. Then it suffices to show the cancellation of their Jacquet modules. This is carried out in Proposition~\ref{prop: cancellation} and the key computation is done in Lemma~\ref{lemma: cancellation of Jacquet module}. For the remaining cancellations, we would like to reduce to the results of M{\oe}glin for $G^{\Sigma_0}$. To do so we need to require the irreducible representation $\lr^{\Sigma_0}$ to satisfy $X(\lr^{\Sigma_0}) = X^{\Sigma_0}(\tilde{\lambda})$ (cf. the proof of Proposition~\ref{prop: special case}), i.e., it can be distinguished from the recursive formula by its restriction to $G^{\Sigma_0}(F)$. At last, we would like to argue that these are the only representations left to consider (cf. Lemma~\ref{lemma: special case}).

\subsubsection{Construction}
\label{subsubsec: construction}

Suppose $[\q] \in \cQ{G}$ has discrete diagonal restriction, i.e., $[\q_d] \in \cPdt{G}$ (cf. \cite[Section 7]{Xu:Apacket}). We always impose the natural order i.e., 
\(
(\rho, A, B, \zeta) >_{\q} (\rho, A', B', \zeta')
\)
if $A > A'$. Let 
\[
\r^{\Sigma_0}_{MW}(\q, \e) := \r^{\Sigma_0}_{W}(\q, \e\e^{MW/W}_{\q}), \quad \r^{\Sigma_0}_{M}(\q, \e) := \r^{\Sigma_0}_{MW}(\q, \e\e^{M/MW}_{\q}).
\]
Suppose there exists $(\rho, A, B, \zeta) \in Jord(\q)$ such that $A > B$. For $\e \in \D{\S{\q}^{\Sigma_0}}$, let $\eta_0 = \e(\rho, A, B, \zeta)$ and M{\oe}glin \cite{Moeglin:2009}  (also see \cite[Theorem 7.14]{Xu:Apacket}) showed that
\begin{align*}
\r^{\Sigma_0}_{M}(\q, \e) = & \oplus_{C \in ]B, A]} (-1)^{A-C} \langle \zeta B, \cdots, -\zeta C \rangle_{\rho} \rtimes {\rm Jac}^{\rho}_{\zeta(B+2), \cdots, \zeta C} \, \r^{\Sigma_0}_{M}\Big(\q', \e', (\rho, A, B+2, \zeta; \eta_0) \Big) \\ & \oplus_{\eta = \pm} (-1)^{[(A - B +1)/2]} \eta^{A-B+1} \eta^{A-B}_0 \r^{\Sigma_0}_M \Big(\q', \e', (\rho, A, B+1, \zeta; \eta), (\rho, B, B, \zeta; \eta\eta_0) \Big).
\end{align*}
Let $J(\q)$ be the set of $\rho$ appearing in $Jord(\q)$. Our idea is to construct a family of Arthur packets for the similitude groups of various ranks, so that they satisfy similar recursive formulas. Let
\begin{align*}
&\lif{X}_{C} := \langle \zeta B, \cdots, -\zeta C \rangle_{\rho} \rtimes  {\rm Jac}^{\rho}_{\zeta(B+2), \cdots, \zeta C} \, \lr^{\Sigma_0}_{M}\Big(\lq', \e', (\rho, A, B+2, \zeta; \eta_0)\Big) \otimes \eta^{C}_{\rho} \omega^{\zeta \gamma_{B}(C)}_{\rho} \chi_{\tilde{\rho}}, \\
& \lif{X}_{\eta} := \lr^{\Sigma_0}_M \Big(\lq', \e', (\rho, A, B+1, \zeta; \eta), (\rho, B, B, \zeta; \eta\eta_0) \Big),
\end{align*}
where $\omega_{\rho} = |\cdot|^{d_{\rho}}$,
\(
\gamma_B(C) = \frac{B + 1}{2} + \Big((B + 2) + \cdots + C \Big),
\)
and $\chi_{\tilde{\rho}}$ is the central character of $\cPkt{\tilde{\rho}}$ for $d_{\rho}$ even (resp. $\cPkt{\widetilde{\rho \otimes \eta_{\rho}}}$ for $d_{\rho}$ odd) when $\rho$ is of orthogonal type and trivial otherwise. If $\lif{X}_{+} \neq \lif{X}_{-}$, we define
\begin{align}
\label{eq: case 1}
\lr^{\Sigma_0}_{M}(\lq, \e) = \oplus_{C \in ]B, A]} (-1)^{A-C} \lif{X}_{C} \oplus_{\eta = \pm} (-1)^{[(A - B +1)/2]} \eta^{A-B+1} \eta^{A-B}_0 \lif{X}_{\eta}.
\end{align}
If $\lif{X}_{+} = \lif{X}_{-}$, then it is necessary that $\rho$ is of orthogonal type with $\eta_{\rho} \neq 1$ and $a, b$ are even, so $A - B + 1$ is also even, and we define
\begin{align}
\label{eq: case 2}
\lr^{\Sigma_0}_{M}(\lq, \e) = \oplus_{C \in ]B, A]} (-1)^{A-C} \lif{X}_{C} \oplus (-1)^{(A - B +1)/2} \eta_0 \lif{X}_{+}.
\end{align}

The construction reduces to the elementary case, which further reduces to the case of discrete $L$-packets by \eqref{eq: involution}. So it suffices to specify a family of discrete $L$-packets. If $\rho \in J(\q)$ is of orthogonal type, we fix $\cPkt{\tilde{\rho}}$ for $d_{\rho}$ even (resp. $\cPkt{\widetilde{\rho \otimes \eta_{\rho}}}$ for $d_{\rho}$ odd). In the case that $\rho \in J(\q)$ and $\p_a = \rho \otimes \nu_a$ for $d_{\rho}$ even (resp. $\p_a = (\rho \otimes \eta_{\rho}) \otimes \nu_a$ for $d_{\rho}$ odd), we define $\cPkt{\lp_{a}}$ inductively by requiring
\begin{align*}
& \cPkt{\lp_{a}} \hookrightarrow \rho||^{\frac{a-1}{2}} \rtimes (\cPkt{\lp_{a - 2}} \otimes \omega_{\rho}^{-\frac{a - 1}{4}}) \text{ for $d_{\rho}$ even} \\
\Big({\rm resp}. \quad & \cPkt{\lp_{a}} \hookrightarrow  \rho \otimes \eta_{\rho}||^{\frac{a-1}{2}} \rtimes (\cPkt{\lp_{a - 2}} \otimes \omega_{\rho}^{-\frac{a - 1}{4}}) \text{ for $d_{\rho}$ odd} \Big).
\end{align*}
For $\p \in \cPdt{G}$ such that $J(\p) \subseteq J(\q)$, we will define $\cPkt{\lp}$ through a sequence of twisted endoscopic transfers. According to \cite[Section 8.3]{Xu:Lpacket}, we need to specify the factorization of $\p$ and liftings of the twisted endoscopic embeddings, which depend on some choices of $1$-cochains of $W_{F}$ in $\mathbb{C}^{\times}$. When $J(\p) = \{\rho\}$, we choose $1$-cochains of $\Gamma_{E/F}$ in $\mathbb{C}^{\times}$ for $E/F$ associated with $\eta_{\rho}$. By \cite[Corollary 8.5 and Theorem 8.12]{Xu:Lpacket}, the resulting packet $\cPkt{\lp}$ is independent of the choice of factorizations of $\p$. In general, we fix an order on $J(\q)$ and decompose $\p = \oplus_{\rho \in J(\q)} \p_{\rho}$ such that $J(\p_{\rho}) = \{\rho\}$. Then we will factor $\p$ by factoring out $\p_{\rho}$ one by one according to this order. To lift the twisted endoscopic embeddings, we will choose 1-cochains of $\Gamma_{E/F}$ in $\mathbb{C}^{\times}$ for $E/F$ associated with $\eta_{\p_{\rho}}$ whenever possible, and fix some arbitrary choices depending on $\rho$ otherwise. We can further make these choices independent of the discrete parameters $\p$ relevant in the reductions of $\cPkt{\q}$. Denote the product of $1$-cochains for defining $\cPkt{\lq_d}$ by $c_{\q}$. At last, $\cPkt{\lp}$ determines $\Pkt{\lp}^{\Sigma_0}$. 
\begin{remark}
\label{rk: elementary}
Let $\p^{\sharp}_a = \rho \otimes \nu_1 \otimes \nu_a$ for $d_{\rho}$ even (resp. $\p^{\sharp}_a = (\rho \otimes \eta_{\rho}) \otimes \nu_1 \otimes \nu_a$ for $d_{\rho}$ odd). Then by \eqref{eq: involution and Jac 1},
\begin{align*}
& \cPkt{\lp^{\sharp}_{a}} \hookrightarrow \rho||^{-\frac{a-1}{2}} \rtimes (\cPkt{\lp^{\sharp}_{a - 2}} \otimes \omega_{\rho}^{\frac{a - 1}{4}}) \text{ for $d_{\rho}$ even} \\
\Big({\rm resp}. \quad & \cPkt{\lp^{\sharp}_{a}} \hookrightarrow  \rho \otimes \eta_{\rho}||^{-\frac{a-1}{2}} \rtimes (\cPkt{\lp^{\sharp}_{a - 2}} \otimes \omega_{\rho}^{\frac{a - 1}{4}}) \text{ for $d_{\rho}$ odd} \Big).
\end{align*}
We claim the elementary packets $\cPkt{\lq}$ for $\q_d \in \cPdt{G}$ with $J(\q_d) \subseteq J(\q)$ can also be obtained through the same sequence of twisted endoscopic transfers as for $\cPkt{\lq_d}$. This is because the elementary packets can be determined by their infinitesimal characters, and the twisted endoscopic transfers preserve the infinitesimal characters, cf. Theorem~\ref{thm: special case 1}.
\end{remark}

Under these constructions of discrete $L$-packets, we have the analogue of \cite[Proposition 9.3]{Xu:cusp} for $\lG^{\Sigma_0}(F)$. 

\begin{proposition}
\label{prop: parabolic reduction full orthogonal group}
Let $\p \in \cPdt{G}$ such that $J(\p) \subseteq J(\q)$ and $\e \in \D{\S{\p}^{\Sigma_{0}}}$. For any $(\rho, a, 1) \in Jord(\p)$, we denote by $a_{-}$ the biggest positive integer smaller than $a$ in $Jord_{\rho}(\p)$, denote by $a_{min}$ the minimum of $Jord_{\rho}(\p)$. If $a = a_{min}$, we let $a_{-} = 0$ if $a$ is even, and $-1$ otherwise. In this case, we always assume $\e(\rho, a,1)\e(\rho, a_{-}, 1) = -1$.
\begin{enumerate}

\item If $\e(\rho, a, 1)\e(\rho, a_{-}, 1) = -1$ and $a_{-} < a - 2$, then 
\begin{align}
\label{eq: parabolic reduction 1 full orthogonal group}
\lr^{\Sigma_{0}}(\p, \e) \hookrightarrow \langle (a-1)/2, \cdots, (a_{-} + 3)/2 \rangle_{\rho} \rtimes \lr^{\Sigma_{0}}(\p', \e') \otimes \omega^{-((a-1) + \cdots + (a_{-} + 3))/4}_{\rho}
\end{align}
as the unique irreducible subrepresentation, where 
\[
Jord(\p') = Jord(\p) \cup \{(\rho, a_{-} + 2, 1)\} \backslash \{(\rho, a, 1)\},
\]
and 
\[
\e'(\cdot)= \e(\cdot) \text{ over } Jord(\p) \backslash \{(\rho, a, 1)\}, \quad \quad \e'(\rho, a_{-}+2, 1) = \e(\rho, a, 1).
\]

\item If $\e(\rho, a, 1)\e(\rho, a_{-},1) = 1$, then 
\begin{align}
\label{eq: parabolic reduction 2 full orthogonal group}
\lr^{\Sigma_{0}}(\p, \e) \hookrightarrow \langle (a-1)/2, \cdots, -(a_{-} -1)/2 \rangle_{\rho} \rtimes \lr^{\Sigma_{0}}(\p', \e') \otimes \omega^{-((a - 1) + \cdots + (a_{-} + 1))/4}_{\rho} \chi_{\tilde{\rho}},
\end{align}
where 
\[
Jord(\p') = Jord(\p) \backslash \{(\rho, a, 1), (\rho, a_{-}, 1)\},
\]
and $\e'(\cdot)$ is the restriction of $\e(\cdot)$. 

\item If $\e(\rho, a_{min}, 1) = 1$ and $a_{min}$ is even, then 
\begin{align}
\label{eq: parabolic reduction 3 full orthogonal group}
\lr^{\Sigma_{0}}(\p, \e) \hookrightarrow \langle (a_{min}-1)/2, \cdots, 1/2 \rangle_{\rho} \rtimes \lr^{\Sigma_{0}}(\p', \e') \otimes \omega_{\rho}^{-((a_{min} - 1) + \cdots + 1)/4}
\end{align}
as the unique irreducible subrepresentation, where 
\[
Jord(\p') = Jord(\p) \backslash \{(\rho, a_{min}, 1)\},
\] 
and $\e'(\cdot)$ is the restriction of $\e(\cdot)$. 

\end{enumerate}

\end{proposition}

\begin{proof}
Comparing with \cite[Proposition 9.3]{Xu:cusp}, it suffices to show 
\[
\overline{\Jac}^{\rho}_{(a - 1)/2, \cdots, (a_{-} + 3)/2} \, \cPkt{\lp} = \cPkt{\lp'} \otimes \omega^{-((a-1) + \cdots + (a_{-} + 3))/4}_{\rho}
\]
for (1), and
\[
\overline{\Jac}^{\rho}_{(a - 1)/2, \cdots, (a_{-} + 1)/2} \, \cPkt{\lp} = \langle (a_{-} -1)/2, \cdots, -(a_{-} -1)/2 \rangle_{\rho} \rtimes \cPkt{\lp'} \otimes \omega^{-((a-1) + \cdots + (a_{-} + 1))/4}_{\rho}
\]
for (2), and
\[
\overline{\Jac}^{\rho}_{(a_{min} - 1)/2, \cdots, 1/2} \, \cPkt{\lp} = \cPkt{\lp'} \otimes \omega_{\rho}^{-((a_{min} - 1) + \cdots + 1)/4}
\]
for (3). Here (1) and (3) follow directly from our construction of the $L$-packets and the compatibility of twisted endoscopic transfer with Jacquet module. By the same argument, we know in case (2) that
\[
\overline{\Jac}^{\rho}_{(a - 1)/2, \cdots, (a_{-} + 1)/2} \, \cPkt{\lp} = {\rm Tran} \, \cPkt{\lp_{H}} \otimes \omega^{-((a-1) + \cdots + (a_{-} + 1))/4}_{\rho}
\]
where $\lp_{H} = {\bold p} \circ (\lp_{a_{-}} \times \lp'')$ and $Jord(\p'') = Jord(\p) \backslash \{(\rho, a)\}$. By \cite[Lemma 8.2]{Xu:Lpacket}, ${\rm Tran} \, \cPkt{\lp_{H}}$ differs from 
\[
\langle (a_{-} -1)/2, \cdots, -(a_{-} -1)/2 \rangle_{\rho} \rtimes \cPkt{\lp'}
\] 
by the twist of the central character $\chi_{\lp_{a_{-}}}$of $\cPkt{\lp_{a_{-}}}$. Since $\chi_{\lp_{a}} = \eta_{\rho} \chi_{\lp_{a - 2}}$ for any $a \in \mathbb{Z}$, then $\chi_{\lp_{a_{-}}} = \eta^{(a_{-} - 1)/2}_{\rho} \chi_{\tilde{\rho}}$. Since of $\lr^{\Sigma_{0}}(\p, \e) \otimes \eta_{\rho} \cong \lr^{\Sigma_{0}}(\p, \e)$, we can also drop $\eta^{(a_{-} - 1)/2}_{\rho}$.
\end{proof}

We could also define $\lr_{M}(\q, \bar{\e})$ in a similar way. Let 
\[
Jord(\q^{1}) = Jord(\q') \cup \{(\rho, A, B + 2, \zeta)\},
\] 
and 
\[
Jord(\q^{2}) = Jord(\q') \cup \{(\rho, A, B + 1, \zeta), (\rho, B, B, \zeta)\}.
\] 
We can identify $\S{\q} \cong \S{\q^{1}}$ by sending $(\rho, A, B, \zeta)$ to $(\rho, A, B+2, \zeta)$ if $A > B +1$. It induces $\S{\lq} \cong \S{\lq^{1}}$. We also map $s \in \S{\q}$ into $\S{\q^{2}}$ by letting 
\[
s(\rho, A, B + 1, \zeta) = s(\rho, B, B, \zeta) := s(\rho, A, B, \zeta).
\]
Then $\S{\q} \hookrightarrow \S{\q^{2}}$ is of index $1$ or $2$. It also induces $\S{\lq} \hookrightarrow \S{\lq^{2}}$ of index $1$ or $2$. We denote the image of $\bar{\e}$ in $\D{\S{\q^{1}}}$ by $\bar{\e}_{1}$. Let us define
\begin{align*}
\lr_{M}(\lq, \bar{\e}) := & \+_{C \in ]B, A]} (-1)^{A-C} \langle \zeta B, \cdots, -\zeta C \rangle_{\rho} \rtimes \overline{\Jac}^{\rho}_{\zeta (B+2), \cdots, \zeta C} \lr_{M}(\lq^{1}, \bar{\e}_{1}) \otimes \eta^{C}_{\rho} \omega^{\zeta \gamma_{B}(C)}_{\rho} \chi_{\tilde{\rho}} \\
& \+_{\tilde{\bar{\e}} \leftarrow \tilde{\bar{\e}}_{2} \in \D{\S{\lq^{2}}}} (-1)^{[(A - B + 1) / 2]} \e_{2}(\rho, A, B+1, \zeta)^{A - B + 1} \e(\rho, A, B, \zeta)^{A - B}  \lr_{M}(\lq^{2}, \bar{\e}_{2}).
\end{align*} 
By induction, one can show $[\lr^{\Sigma_0}_{M}(\lq, \e)|_{\lG}] = 2 \lr_{M}(\lq, \bar{\e})$ if $G$ is special even orthogonal and $\S{\lq}^{\Sigma_0} = \S{\lq}$, or $\lr_{M}(\lq, \bar{\e})$ otherwise.

\begin{lemma}
\label{lemma: independence}
The definitions of $\lr^{\Sigma_0}_{M}(\lq, \e)$ and $\lr_{M}(\lq, \bar{\e})$ are independent of the choice of $(\rho, A, B, \zeta) \in Jord(\q)$ such that $A > B$.
\end{lemma}

\begin{proof}
We can prove this by induction on $\sum_{(\rho, A, B, \zeta) \in Jord(\q)} (A - B)$. Suppose there exists another Jordan block $(\rho', A', B', \zeta') \in Jord(\q)$ such that $A' > B'$. By induction assumption, we can substitute in $\lr^{\Sigma_0}_{M}(\lq, \e)$ the recursive formulas for $\lr^{\Sigma_0}_{M}(\lq', \e', (\rho, A, B+2, \zeta; \eta_0))$ and $\lr^{\Sigma_0}_M (\lq', \e', (\rho, A, B+1, \zeta; \eta), (\rho, B, B, \zeta; \eta\eta_0) )$ with respect to $(\rho', A', B', \zeta')$. To simplify the result, we can use the facts that
\begin{align*}
& {\rm Jac}^{\rho}_{\zeta(B+2), \cdots, \zeta C} \, \Big(\langle \zeta' B, \cdots, -\zeta' C' \rangle_{\rho'} \rtimes {\rm Jac}^{\rho'}_{\zeta'(B'+2), \cdots, \zeta' C'} \,\lr^{\Sigma_0}_{M}\Big(\lq'', \e'', (\rho, A, B+2, \zeta; \eta_0), (\rho', A', B'+2, \zeta'; \eta'_0)\Big) \Big) \\
= & \langle \zeta' B, \cdots, -\zeta' C' \rangle_{\rho'} \rtimes {\rm Jac}^{\rho}_{\zeta(B+2), \cdots, \zeta C} \circ {\rm Jac}^{\rho'}_{\zeta' (B'+2), \cdots, \zeta' C'} \,\lr^{\Sigma_0}_{M}\Big(\lq'', \e'', (\rho, A, B+2, \zeta; \eta_0), (\rho', A', B'+2, \zeta'; \eta'_0)\Big)
\end{align*}
and
\[
\langle \zeta B, \cdots, -\zeta C \rangle_{\rho} \times \langle \zeta' B', \cdots, -\zeta' C' \rangle_{\rho'} \cong \langle \zeta' B', \cdots, -\zeta' C' \rangle_{\rho'} \times \langle \zeta B, \cdots, -\zeta C \rangle_{\rho}
\]
and ${\rm Jac}^{\rho}_{\zeta(B+2), \cdots, \zeta C}$ commutes with ${\rm Jac}^{\rho'}_{\zeta' (B'+2), \cdots, \zeta' C'}$. Then the result would be the same if we first define $\lr^{\Sigma_0}_{M}(\lq, \e)$ with respect to $(\rho', A', B', \zeta')$ and expand further with respect to $(\rho, A, B, \zeta)$. The case of $\lr_{M}(\lq, \bar{\e})$ follows by restricting $\lr^{\Sigma_0}_{M}(\lq, \e)$ to $\lG(F)$.

\end{proof}

\begin{lemma}
\label{lemma: restriction}
\,
\begin{enumerate}

\item For $\e' \in (\S{\q}^{\Sigma_0}/\S{\lq}^{\Sigma_0})^{\wedge}$ (resp. $\bar{\e}' \in (\S{\q}/\S{\lq})^{\wedge}$), 
\[
\lr^{\Sigma_0}_M(\lq, \e) = \lr^{\Sigma_0}_M(\lq, \e \e') \quad \text{ (resp. $\lr_M(\lq, \bar{\e}) = \lr_M(\lq, \bar{\e} \bar{\e}')$ ).}
\]

\item
\begin{align}
\label{eq: restriction}
\lr^{\Sigma_0}_{M}(\lq, \e)|_{G^{\Sigma_0}(F)} & = \bigoplus_{\e' \in (\S{\q}^{\Sigma_0}/\S{\lq}^{\Sigma_0})^{\wedge}} \, \r^{\Sigma_0}_{M}(\q, \e \e'). 
\end{align}

\begin{align}
\lr_{M}(\lq, \bar{\e})|_{G(F)} & = \bigoplus_{\bar{\e}' \in (\S{\q}/\S{\lq})^{\wedge}} \, \r_{M}(\q, \bar{\e} \bar{\e}').
\end{align}

\end{enumerate}
\end{lemma}

\begin{proof}
It follows from the recursive formulas and the results in the discrete case.
\end{proof}

So far we have only defined $\lr^{\Sigma_0}_{M}(\lq, \e)$ in the Grothendieck group, and it is by no means clear that this defines a representation. Indeed, M{\oe}glin first defined $\r^{\Sigma_0}_M(\q, \e)$ by the recursive formula and then showed that it is a representation by direct computation (cf. \cite[Theorem 4.1]{Moeglin:2009}). We will follow the same strategy below.

\begin{theorem}
\label{thm: discrete diagonal restriction}
Let
\[
\lr^{\Sigma_0}_{M}(\lq, \e)_{\rm main} = \Big\langle \langle \zeta B, \cdots ,-\zeta A \rangle_{\rho} \rtimes \lr^{\Sigma_0}_M \Big(\lq', \e', (\rho, A- 1, B+1, \zeta; \eta_0)\Big) \otimes \eta_{\rho}^{A} \omega_{\rho}^{\zeta((B+1) + \cdots + A)/2} \chi_{\tilde{\rho}} \Big\rangle
\]
and 
\[
\lr^{\Sigma_0}_{M}(\lq, \e)_{{\rm com}, \eta} = \lr^{\Sigma_0}_{M}\Big(\lq', \e', \cup_{C \in [B, A]}(\rho, C, C, \zeta; (-1)^{C-B}\eta) \Big)
\]
for
\[
\eta \in S_{\eta_0} := \Big\{\eta = \pm \, | \, \eta_0 = \prod_{C \in [B, A]} (-1)^{C-B} \eta \Big\}. 
\]
If
\(
\lr^{\Sigma_0}_{M}(\lq, \e)_{{\rm com}, \eta} \neq \lr^{\Sigma_0}_{M}(\lq, \e)_{com, -\eta},
\) 
then
\begin{align}
\label{eq: case 1 resolved}
\lr^{\Sigma_0}_{M}(\lq, \e) = \lr^{\Sigma_0}_{M}(\lq, \e)_{\rm main} \oplus_{\eta \in S_{\eta_0}} \lr^{\Sigma_0}_{M}(\lq, \e)_{{\rm com}, \eta}.
\end{align}
Otherwise,
\begin{align}
\label{eq: case 2 resolved}
\lr^{\Sigma_0}_{M}(\lq, \e) = \lr^{\Sigma_0}_{M}(\lq, \e)_{\rm main} \oplus \lr^{\Sigma_0}_{M}(\lq, \e)_{{\rm com}, \eta}
\end{align}
for any $\eta \in S_{\eta_0}$.
\end{theorem}

\begin{corollary}
$\lr^{\Sigma_0}_{M}(\lq, \e)$ is a representation of $\lG^{\Sigma_0}(F)$.
\end{corollary}


We will state some consequences of this theorem. Define $\Pkt{\lq}^{\Sigma_0}$ to be the set of irreducible constituents of
\[
\bigoplus_{\tilde{\e} \in \D{\S{\lq}^{\Sigma_0}}} \lr^{\Sigma_0}_{M}(\lq, \e).
\]
From \cite[Theorem 4.2]{Moeglin:2009} (also see \cite[Theorem 7.8]{Xu:Apacket}), the irreducible constitutes of $\r^{\Sigma_{0}}_{M}(\q, \e)$ can be parametrized by pairs of integer-valued functions $(\ul, \ueta)$ over $Jord(\q)$, such that 
\begin{align}
\label{eq: refined parameter}
\ul(\rho, A, B, \zeta) \in [0, [(A-B+1)/2]] \text{ and } \ueta(\rho, A, B, \zeta) \in \{\pm 1\},
\end{align}
and
\begin{align}
\label{eq: endoscopic character formula}
\e(\rho, A, B, \zeta) = \e_{\ul, \ueta}(\rho, A, B, \zeta) := \ueta(\rho, A, B, \zeta)^{A - B + 1} (-1)^{[(A-B+1)/2] + \ul(\rho, A, B, \zeta)}.
\end{align}
Moreover,
\begin{align*}
\r_{M}^{\Sigma_{0}}(\q, \ul, \ueta) & \hookrightarrow \times_{(\rho, A, B, \zeta) \in Jord(\q)} 
       \begin{pmatrix}
              \zeta B & \cdots & -\zeta A \\
              \vdots &  & \vdots \\
              \zeta (B + \ul(\rho, A, B, \zeta) - 1) & \cdots & - \zeta (A - \ul(\rho, A, B, \zeta) + 1)
       \end{pmatrix}_{\rho} \\
& \times \r_{M}^{\Sigma_{0}}\Big(\cup_{(\rho, A, B, \zeta) \in Jord(\q)} \cup_{C \in [B + \ul(\rho, A, B, \zeta), A - \ul(\rho, A, B, \zeta)]} (\rho, C, C, \zeta; \ueta(\rho, A, B, \zeta)(-1)^{C - B - \ul(\rho, A, B, \zeta)})\Big)
\end{align*}
as the unique irreducible subrepresentation. There is an obvious equivalence relation to be made here on pairs $(\ul, \ueta)$, namely 
\[
(\ul, \ueta) \sim_{\Sigma_{0}} (\ul', \ueta')
\]
if and only if $\ul = \ul'$ and $(\ueta/\ueta') (\rho, A, B, \zeta)= 1$ unless $\ul(\rho, A, B, \zeta) = (A - B +1)/2$. Then 
\begin{align}
\label{eq: refined decomposition classical}
\r^{\Sigma_{0}}_{M}(\q, \e) = \bigoplus_{\{(\ul, \ueta): \, \e = \e_{\ul, \ueta} \}/\sim_{\Sigma_{0}}} \r^{\Sigma_{0}}_{M}(\q, \ul, \ueta).
\end{align}
We define $\lr^{\Sigma_0}_{M}(\lq, \ul, \ueta)$ to be the element in $\Pkt{\lq}^{\Sigma_0}$ containing $\r^{\Sigma_0}_{M}(\q, \ul, \ueta)$ in its restriction to $G^{\Sigma_0}(F)$. Then 
\begin{align*}
\lr_{M}^{\Sigma_{0}}(\lq, \ul, \ueta) & \hookrightarrow \times_{(\rho, A, B, \zeta) \in Jord(\q)} 
       \begin{pmatrix}
              \zeta B & \cdots & -\zeta A \\
              \vdots &  & \vdots \\
              \zeta (B + \ul(\rho, A, B, \zeta) - 1) & \cdots & - \zeta (A - \ul(\rho, A, B, \zeta) + 1)
       \end{pmatrix}_{\rho} \\
& \rtimes \lr_{M}^{\Sigma_{0}}\Big(\cup_{(\rho, A, B, \zeta) \in Jord(\q)} \cup_{C \in [B + \ul(\rho, A, B, \zeta), A - \ul(\rho, A, B, \zeta)]} (\rho, C, C, \zeta; \ueta(\rho, A, B, \zeta)(-1)^{C - B - \ul(\rho, A, B, \zeta)})\Big) \\
& \otimes_{(\rho, A, B, \zeta) \in Jord(\q)} \eta^{(A - (\ul(\rho, A, B, \zeta)-1)/2)\ul(\rho, A, B, \zeta)}_{\rho} \omega^{(A + B + 1)(A - B - \ul(\rho, A, B, \zeta) +1)\ul(\rho, A, B, \zeta)/2}_{\rho} \chi^{\ul(\rho, A, B, \zeta)}_{\tilde{\rho}}
\end{align*}
as the unique irreducible subrepresentation. Moreover, $\lr_{M}^{\Sigma_{0}}(\lq, \ul, \ueta) = \lr_{M}^{\Sigma_{0}}(\lq, \ul', \ueta')$ if and only if $\ul = \ul'$ and 
\begin{align*}
& \lr_{M}^{\Sigma_{0}}\Big(\cup_{(\rho, A, B, \zeta) \in Jord(\q)} \cup_{C \in [B + \ul(\rho, A, B, \zeta), A - \ul(\rho, A, B, \zeta)]} (\rho, C, C, \zeta; \ueta(\rho, A, B, \zeta)(-1)^{C - B - \ul(\rho, A, B, \zeta)})\Big) \\
= \, & \lr_{M}^{\Sigma_{0}}\Big(\cup_{(\rho, A, B, \zeta) \in Jord(\q)} \cup_{C \in [B + \ul'(\rho, A, B, \zeta), A - \ul'(\rho, A, B, \zeta)]} (\rho, C, C, \zeta; \ueta'(\rho, A, B, \zeta)(-1)^{C - B - \ul'(\rho, A, B, \zeta)})\Big).
\end{align*}
This defines an equivalence relation $\sim_{\lG^{\Sigma_0}}$ on $(\ul, \ueta)$. Hence
\begin{align}
\label{eq: refined decomposition}
\lr^{\Sigma_{0}}_{M}(\lq, \e) = \bigoplus_{\{(\ul, \ueta): \, \e = \e_{\ul, \ueta} \}/\sim_{\lG^{\Sigma_{0}}}} \lr^{\Sigma_{0}}_{M}(\lq, \ul, \ueta).
\end{align}

We define $\lr_{M}(\lq, \ul, \ueta)$ to be the irreducible representation of $\lG(F)$ viewed as $\sH(\lG(F))$-module in the restriction of $\lr_{M}^{\Sigma_{0}}(\lq, \ul, \ueta)$ to $\lG(F)$. Then
\begin{align*}
\lr_{M}(\lq, \ul, \ueta) & \hookrightarrow \times_{(\rho, A, B, \zeta) \in Jord(\q)} 
       \begin{pmatrix}
              \zeta B & \cdots & -\zeta A \\
              \vdots &  & \vdots \\
              \zeta (B + \ul(\rho, A, B, \zeta) - 1) & \cdots & - \zeta (A - \ul(\rho, A, B, \zeta) + 1)
       \end{pmatrix} \\
& \rtimes \lr_{M}\Big(\cup_{(\rho, A, B, \zeta) \in Jord(\q)} \cup_{C \in [B+\ul(\rho, A, B, \zeta), A - \ul(\rho, A, B, \zeta)]} (\rho, C, C, \zeta; \ueta(\rho, A, B, \zeta)(-1)^{C - B - \ul(\rho, A, B, \zeta)})\Big) \\
& \otimes_{(\rho, A, B, \zeta) \in Jord(\q)} \eta^{(A - (\ul(\rho, A, B, \zeta)-1)/2)\ul(\rho, A, B, \zeta)}_{\rho} \omega^{(A + B + 1)(A - B - \ul(\rho, A, B, \zeta) +1)\ul(\rho, A, B, \zeta)/2}_{\rho} \chi^{\ul(\rho, A, B, \zeta)}_{\tilde{\rho}}
\end{align*}
as the unique irreducible element in $\overline{\Rep}(\lG(F))$ forming an $\sH(\lG(F))$-submodule. Moreover, $\lr_{M}(\lq, \ul, \ueta) = \lr_{M}(\lq, \ul', \ueta')$ if and only if $\ul = \ul'$ and 
\begin{align*}
& \lr_{M}\Big(\cup_{(\rho, A, B, \zeta) \in Jord(\q)} \cup_{C \in [B + \ul(\rho, A, B, \zeta), A - \ul(\rho, A, B, \zeta)]} (\rho, C, C, \zeta; \ueta(\rho, A, B, \zeta)(-1)^{C - B - \ul(\rho, A, B, \zeta)})\Big) \\
= \, & \lr_{M}\Big(\cup_{(\rho, A, B, \zeta) \in Jord(\q)} \cup_{C \in [B + \ul'(\rho, A, B, \zeta), A - \ul'(\rho, A, B, \zeta)]} (\rho, C, C, \zeta; \ueta'(\rho, A, B, \zeta)(-1)^{C - B - \ul'(\rho, A, B, \zeta)})\Big).
\end{align*}
This defines an equivalence relation $\sim_{\lG}$ on $(\ul, \ueta)$. Hence
\begin{align}
\label{eq: refined decomposition G}
\lr_{M}(\lq, \bar{\e}) = \bigoplus_{\{(\ul, \ueta): \, \bar{\e} = \bar{\e}_{\ul, \ueta} \}/\sim_{\lG}} \lr_{M}(\lq, \ul, \ueta).
\end{align}

In the rest of this sub-subsection, we will prove Theorem~\ref{thm: discrete diagonal restriction}. Let us assume the theorem holds for 
\[
\sum_{(\rho, A, B, \zeta) \in Jord(\q)} (A - B) < K.
\]
In our discussions below, we consider $\sum_{(\rho, A, B, \zeta) \in Jord(\q)} (A - B) \leqslant K$. The next two lemmas are preparations.

\begin{lemma}
\label{lemma: independence resolved}
The representations on the right hand sides of \eqref{eq: case 1 resolved}, \eqref{eq: case 2 resolved} are independent of the choice of $(\rho, A, B, \zeta) \in Jord(\q)$ for $A > B$. 
\end{lemma}

\begin{proof}
By the induction assumption, we can conclude that the right hand sides of \eqref{eq: case 1 resolved}, \eqref{eq: case 2 resolved} are equal to that of \eqref{eq: refined decomposition}, which is independent of the choice of $(\rho, A, B, \zeta) \in Jord(\q)$ for $A > B$.
\end{proof}

\begin{lemma}
\label{lemma: shift}
Suppose $\q_{T^{0}}$ is obtained from $\q$ by shifting $(\rho, A, B, \zeta)$ to $(\rho, A+T_{0}, B+T_{0}, \zeta)$ such that it has discrete diagonal restriction and
any Jordan block in $Jord_{\rho}(\q)$ between the two has sign $-\zeta$. Then
\begin{align}
\label{eq: shift}
{\rm Jac}_{(\rho, A+T_0, B+T_0, \zeta) \mapsto (\rho, A, B, \zeta)} \, \lr^{\Sigma_0}_{M}(\lq_{T_0}, \e_{T_0}) = \lr^{\Sigma_0}_{M}(\lq, \e) \otimes \omega^{-|X^{T_{0}}_{(\rho, A, B, \zeta)}|/2}_{\rho}.
\end{align}
where $\e_{T_0}$ is related to $\e$ by the change of order formula (cf. \cite[Theorem 6.3]{Xu:Comb}).
\end{lemma}

\begin{proof}
We will prove this by induction on $\sum_{(\rho, A, B, \zeta) \in Jord(\q)} (A - B)$. If $\sum_{(\rho, A, B, \zeta) \in Jord(\q)} (A - B) = 0$, i.e., $\q$ is elementary, then $\lr^{\Sigma_0}_{M}(\lq_{T_0}, \e_{T_0})$ is a representation by definition (cf. Theorem~\ref{thm: special case}). Since
\[
{\rm Jac}_{(\rho, A+T_0, B+T_0, \zeta) \mapsto (\rho, A, B, \zeta)} \, \r^{\Sigma_0}_{M}(\q_{T_0}, \e_{T_0}) = \r^{\Sigma_0}_{M}(\q, \e),
\]
by the change of order formula (cf. \cite[Theorem 6.3]{Xu:Comb}), it suffices to know that
\[
\overline{{\rm Jac}}_{(\rho, A+T_0, B+T_0, \zeta) \mapsto (\rho, A, B, \zeta)} \, \cPkt{\lq_{T_0}} = \cPkt{\lq} \otimes \omega^{-|X^{T_0}_{(\rho, A, B, \zeta)}|/2}_{\rho}.
\]
Since the elementary packets can also be constructed through a sequence of twisted endoscopic transfers (see Remark~\ref{rk: elementary}), then this follows from the compatibility of twisted endoscopic transfer with Jacquet module as in the proof of \eqref{eq: parabolic reduction 1 full orthogonal group}. If there exists $(\rho', A', B', \zeta') \in Jord(\q)$ distinct from $(\rho, A, B, \zeta)$ such that $A' > B'$. Then we can expand $\lr^{\Sigma_0}_{M}(\lq_{T_0}, \e_{T_0})$ according to the recursive formula with respect to $(\rho', A', B', \zeta')$. Since we can take ${\rm Jac}_{(\rho, A+T_0, B+T_0, \zeta) \mapsto (\rho, A, B, \zeta)}$ in front of 
\[
\lr^{\Sigma_0}_{M}\Big(\lq'_{T_0}, \e'_{T_0}, (\rho, A', B'+2, \zeta'; \eta'_{T_0, 0})\Big), \quad \lr^{\Sigma_0}_M \Big(\lq'_{T_0}, \e'_{T_0}, (\rho, A', B'+1, \zeta'; \eta'), (\rho, B', B', \zeta'; \eta' \eta'_{T_0, 0}) \Big),
\]
then the result follows from the induction assumption.

Now we may assume $(\rho, A, B, \zeta)$ is the only Jordan block such that $A > B$. For $0 \leqslant T \leqslant T_0$, let us define
\begin{align}
\label{eq: shift intermediate}
\lr^{\Sigma_0}_{M}(\lq_{T}, \e_{T}) :=  {\rm Jac}_{(\rho, A + T_{0}, B + T_{0}, \zeta) \mapsto (\rho, A + T, B + T, \zeta)} \, \lr^{\Sigma_0}_{M}(\lq_{T_0}, \e_{T_0}) \otimes \omega_{\rho}^{|X^{T_0 - T}_{(\rho, A + T, B+T, \zeta)}|/2}.
\end{align}
Note the restriction of $\lr^{\Sigma_0}_{M}(\lq_{T}, \e_{T})$ to $G^{\Sigma_0}(F)$ is 
\[
\bigoplus_{\e' \in (\S{\q_{T}}^{\Sigma_0}/\S{\lq_{T}}^{\Sigma_0})^{\wedge}} \, \r^{\Sigma_0}_{M}(\q_{T}, \e_{T} \e').
\]

We claim that $\lr^{\Sigma_0}_{M}(\lq_{T}, \e_{T})$ can be expressed as \eqref{eq: case 1} or \eqref{eq: case 2}, where
\begin{align*}
\lif{X}_{T, C} & := \langle \zeta (B + T), \cdots, -\zeta (C + T) \rangle_{\rho} \rtimes  {\rm Jac}^{\rho}_{\zeta(B + T +2), \cdots, \zeta (C + T)}  \\
& \lr^{\Sigma_0}_{M}\Big(\lq'_{T}, \e'_{T}, (\rho, A + T, B+T+2, \zeta; \eta_{T, 0})\Big) \otimes \eta^{C + T}_{\rho} \omega^{\zeta \gamma_{B+T}(C + T)}_{\rho} \chi_{\tilde{\rho}}
\end{align*}
for $B + 1 \leqslant C \leqslant A$ and
\begin{align*}
& \lr^{\Sigma_0}_{M}\Big(\lq'_{T}, \e'_{T}, (\rho, A + T, B + T + 2, \zeta; \eta_{T, 0})\Big) =  {\rm Jac}_{(\rho, A + T_{0}, B + T_{0} +2, \zeta) \mapsto (\rho, A + T, B + T +2, \zeta)} \\
& \lr^{\Sigma_0}_{M}\Big(\lq'_{T}, \e'_{T}, (\rho, A + T_0, B + T_0 + 2, \zeta; \eta_{T, 0})\Big) \otimes \omega_{\rho}^{|X^{T_0 - T}_{(\rho, A + T, B+T +2, \zeta)}|/2}
\end{align*}
and
\begin{align*}
\lif{X}_{T, \eta} := & {\rm Jac}_{(\rho, A + T_{0}, B + T_{0}, \zeta) \mapsto (\rho, A + T, B + T, \zeta)} \, \lif{X}_{T_0, \eta} \otimes \omega_{\rho}^{|X^{T_0 - T}_{(\rho, A + T, B+T, \zeta)}|/2} \\
= & {\rm Jac}_{(\rho, A + T_{0}, B + T_{0} + 1, \zeta) \mapsto (\rho, A + T, B + T + 1, \zeta)} \circ {\rm Jac}_{(\rho, B + T_{0}, B + T_{0}, \zeta) \mapsto (\rho, B + T, B + T, \zeta)} \\\
&  \lif{X}_{T_0, \eta} \otimes \omega_{\rho}^{|X^{T_0 - T}_{(\rho, B + T, B+T, \zeta)}|/2} \omega_{\rho}^{|X^{T_0 - T}_{(\rho, A + T, B+T + 1, \zeta)}|/2}.
\end{align*}
Then the lemma follows from the case $T = 0$ and the induction assumption.

At last, we shall prove our claim by induction on $T_0 - T$. It suffices to show that
\begin{align}
\label{eq: shift 2}
{\rm Jac}^{\rho}_{\zeta(B+T), \cdots, \zeta(A+T)} \, \lif{X}_{T, C} = \lif{X}_{T - 1, C} \otimes \omega_{\rho}^{-\zeta((B+T) + \cdots +(A+T))/2}
\end{align}
and
\begin{align}
\label{eq: shift 3}
{\rm Jac}^{\rho}_{\zeta(B+T), \cdots, \zeta(A+T)} \, \lif{X}_{T, \eta} = \lif{X}_{T - 1, \eta} \otimes \omega_{\rho}^{-\zeta((B+T) + \cdots +(A+T))/2}.
\end{align}
The equality \eqref{eq: shift 3} is clear from the definition. For \eqref{eq: shift 2}, we write
\[
{\rm Jac}^{\rho}_{\zeta (B+T), \cdots, \zeta (A+T)} \lif{X}_{T, C} = {\rm Jac}^{\rho}_{\zeta (C + T + 1), \cdots, \zeta (A + T)} \circ {\rm Jac}^{\rho}_{\zeta (B + T + 1) , \cdots, \zeta (C + T)} \circ {\rm Jac}^{\rho}_{\zeta (B + T)} \lif{X}_{T, C}.
\]
First we have
\begin{align*}
{\rm Jac}^{\rho}_{\zeta(B+T)} \lif{X}_{T, C} & = \langle \zeta (B + T - 1), \cdots, -\zeta (C + T) \rangle_{\rho} \rtimes  {\rm Jac}^{\rho}_{\zeta(B + T +2), \cdots, \zeta (C + T)}  \\
& \lr^{\Sigma_0}_{M}\Big(\lq'_{T}, \e'_{T}, (\rho, A + T, B+T+2, \zeta; \eta_{T, 0})\Big) \otimes \eta^{C + T}_{\rho} \omega^{\zeta \gamma_{B+T}(C + T)}_{\rho} \chi_{\tilde{\rho}}.
\end{align*}
Next we claim that
\begin{align}
\label{eq: vanishing 2}
{\rm Jac}^{\rho}_{\zeta (B + T + 1) , \cdots, \zeta (C + T)} \circ  {\rm Jac}^{\rho}_{\zeta(B + T +2), \cdots, \zeta (C + T)} \lr^{\Sigma_0}_{M}\Big(\lq'_{T}, \e'_{T}, (\rho, A + T, B+T+2, \zeta; \eta_{T, 0})\Big) = 0.
\end{align}
Since $\lr^{\Sigma_0}_{M}\Big(\lq'_{T}, \e'_{T}, (\rho, A + T, B+T+2, \zeta; \eta_{T, 0})\Big)$ is a representation, it suffices to show
\[
{\rm Jac}^{\rho}_{\zeta (B + T + 1) , \cdots, \zeta (C + T)} \circ  {\rm Jac}^{\rho}_{\zeta(B + T +2), \cdots, \zeta (C + T)} \r^{\Sigma_0}_{M}\Big(\q'_{T}, \e'_{T}, (\rho, A + T, B+T+2, \zeta; \eta_{T, 0})\Big) = 0.
\]
For any irreducible constituent $\r^{\Sigma_0}$ in $\r^{\Sigma_0}_{M}\Big(\q'_{T}, \e'_{T}, (\rho, A + T, B+T+2, \zeta; \eta_{T, 0})\Big)$, we have 
\[
\r^{\Sigma_0} \hookrightarrow \begin{pmatrix} \zeta(B+T +2) & \zeta(B+T+1) \\
\vdots & \vdots \\
\zeta(C+T) & \zeta(C+T - 1)
\end{pmatrix}_{\rho} \rtimes \sigma^{\Sigma_0}.
\]
If ${\rm Jac}^{\rho}_{\zeta x} \sigma^{\Sigma_0} = 0$ for all $x \in [B + T +1, C + T]$, the vanishing statement is clear. Suppose ${\rm Jac}^{\rho}_{\zeta x} \sigma^{\Sigma_0} \neq 0$ for some $x \in [B + T +1, C + T]$, then ${\rm Jac}^{\rho}_{\zeta x} \r^{\Sigma_0} \neq 0$. It is necessary that $x = B + T +2$. So ${\rm Jac}^{\rho}_{\zeta (B+T + 2) , \zeta (B+T+2)} \r^{\Sigma_0} \neq 0$. But this is impossible by \cite[Proposition 8.3]{Xu:Apacket}. At last, it follows from \eqref{eq: vanishing 2} that
\begin{align*}
& {\rm Jac}^{\rho}_{\zeta (B+T), \cdots, \zeta (A+T)} \lif{X}_{T, C}  = \langle \zeta (B + T - 1), \cdots, -\zeta (C +T - 1) \rangle_{\rho} \rtimes {\rm Jac}^{\rho}_{\zeta (C + T + 1), \cdots, \zeta (A+T)} \circ {\rm Jac}^{\rho}_{\zeta (B+T + 1) , \cdots, \zeta (C +T - 1)} \\
& \circ {\rm Jac}^{\rho}_{\zeta (B + T +2), \cdots, \zeta (C + T)} \lr^{\Sigma_0}_{M}\Big(\lq'_{T}, \e'_{T}, (\rho, (A+T), (B+T+2), \zeta; \eta_{T, 0})\Big) \otimes \eta^{C+T}_{\rho} \omega^{\zeta \gamma_{B+T}(C+T)}_{\rho} \chi_{\tilde{\rho}} \cdot \eta_{\rho} \omega_{\rho}^{-\zeta (C+T)} 
\\
& = \langle \zeta (B + T - 1), \cdots, -\zeta (C +T - 1) \rangle_{\rho} \rtimes {\rm Jac}_{\zeta (B+T + 1) , \cdots, \zeta (C +T - 1)} \\
& \circ {\rm Jac}^{\rho}_{\zeta (B + T +2), \cdots, \zeta (A + T)} \lr^{\Sigma_0}_{M}\Big(\lq'_{T}, \e'_{T}, (\rho, A+T, B+T+2, \zeta; \eta_{T, 0})\Big) \otimes \eta^{C+T}_{\rho} \omega^{\zeta \gamma_{B+T}(C+T)}_{\rho} \chi_{\tilde{\rho}} \cdot \eta_{\rho} \omega_{\rho}^{-\zeta (C+T)} 
\\
& = \langle \zeta (B +T - 1), \cdots, -\zeta (C +T - 1) \rangle_{\rho} \rtimes {\rm Jac}_{\zeta (B +T + 1) , \cdots, \zeta (C +T - 1)}  \\
& \lr^{\Sigma_0}_{M}\Big(\lq'_{T}, \e'_{T}, (\rho, A + T - 1, B + T +1, \zeta; \eta_{T, 0})\Big) \otimes \eta^{C + T - 1}_{\rho} \omega^{\zeta \gamma_{B + T}(C + T - 1)}_{\rho} \chi_{\tilde{\rho}}  \cdot \omega_{\rho}^{-\zeta((B+T+2) + \cdots +(A+T))/2}
\end{align*}
Note 
\[
\gamma_{B+T}(C + T - 1) = \gamma_{B + T - 1}(C + T - 1) - (B + T +1)/2 - (B + T)/2.
\]
This finishes the proof of our claim.
\end{proof}

\begin{remark}
We will only need the special case that there are no Jordan blocks between $(\rho, A, B, \zeta)$ and $(\rho, A + T_0, B+T_0, \zeta)$ until the end of the proof. The idea of shifting the Jordan block is motivated by the construction in the general case (cf. Subsection~\ref{subsec: general case}).
\end{remark}




Now we can start the heart part of the proof. First of all, let us determine the $\Sigma_0$-infinitesimal characters of irreducible representations in $\lr^{\Sigma_0}_{M}(\lq, \e)$.

\begin{proposition}
The $\Sigma_0$-infinitesimal characters of irreducible constituents in $\lif{X}_{C}, \lif{X}_{\eta}$ are the same as that of $\Pkt{\lq_{d}}^{\Sigma_0}$.
\end{proposition}

\begin{proof}
We will prove this by induction on $\sum_{(\rho, A, B, \zeta) \in Jord(\q)} (A - B)$. The case of $\lif{X}_{\eta}$ follows from the induction assumption directly. For $\lif{X}_{C}$, it is the same to determine that of
\[
\langle  B, \cdots,  -C \rangle_{\rho} \rtimes  {\rm Jac}^{\rho}_{B+2, \cdots, C} \, \lr^{\Sigma_0}_{M}(\lq^{1}_d, 1) \otimes \eta^{C}_{\rho} \omega^{ \gamma_{B}(C)}_{\rho} \chi_{\tilde{\rho}},
\]
which is the same as 
\[
\langle  C, \cdots, - B \rangle_{\rho} \rtimes  {\rm Jac}^{\rho}_{B+2, \cdots,  C} \, \lr^{\Sigma_0}_{M}(\lq^{1}_d, 1) \otimes \eta^{C}_{\rho} \omega^{ \gamma_{B}(C)}_{\rho} \chi_{\tilde{\rho}} \cdot  \omega_{\rho}^{-((B+1) + \cdots + C)} \eta_{\rho}^{C - B}.
\]
Note
\[
{\rm Jac}^{\rho}_{B+2, \cdots, C} \, \lr^{\Sigma_0}_{M}(\lq^{1}_d, 1) = \lr^{\Sigma_0}_{M}\Big(\lq'_d, 1, \cup_{C' \in [B+1, A] \backslash \{C\}}(\rho, C', C', +; +1) \Big) \otimes \omega_{\rho}^{-((B+2) + \cdots + C)/2}
\]
(cf. proof of \eqref{eq: parabolic reduction 1 full orthogonal group}). On the other hand, one can show 
\begin{align*}
& \lr^{\Sigma_0}_{M}\Big(\lq'_d, 1, \cup_{C' \in [C + 1, A]}(\rho, C', C', -; +1), (\rho, C, C, +; (-1)^{C - B - 1}), \cup_{C' \in [B + 1, C - 1]}(\rho, C', C', -; -1), (\rho, B, B, +; +1)\Big) \\
& \hookrightarrow  \langle  C, \cdots, - B \rangle_{\rho} \rtimes \lr^{\Sigma_0}_{M}\Big(\lq'_d, 1, \cup_{C' \in [B+1, A] \backslash \{C\}}(\rho, C', C', -; +1) \Big) \otimes  \omega_{\rho}^{-(C + \cdots + (B+1))/2} \chi_{\tilde{\rho}},
\end{align*}
by the compatibility of twisted endoscopic transfer with Jacquet module as in the proof of \eqref{eq: parabolic reduction 2 full orthogonal group}. So the $\Sigma_0$-infinitesimal character is the same as that of $\Pkt{\lq_d}^{\Sigma_0}$.

\end{proof}

This proposition explains the appearance of twists by $\omega_{\rho}$ and $\chi_{\tilde{\rho}}$ in $\lif{X}_{C}$. The twist by $\eta_{\rho}$ is more subtle, since it does not change the $\Sigma_0$-infinitesimal character. Next we would like to distinguish $\lr^{\Sigma_0}_{M}(\lq, \e)_{\rm main}$ and $\lr^{\Sigma_0}_{M}(\lq, \e)_{{\rm com}, \eta}$ from $\lif{X}_{C}, \lif{X}_{\eta}$. 

\begin{proposition}
\label{prop: existence}
The irreducible constituents of $\lr^{\Sigma_0}_{M}(\lq, \e)_{\rm main}$ (resp. $\lr^{\Sigma_0}_{M}(\lq, \e)_{{\rm com}, \eta}$) only appear in $\lif{X}_{A}$ (resp. $\lif{X}_{\eta'}$ for some $\eta'$) with multiplicity one.
\end{proposition}

\begin{proof}
By definition, $\lr^{\Sigma_0}_{M}(\lq, \e)_{\rm main}$ is the socle of $\lif{X}_{A}$. To show multiplicity one, we can apply ${\rm Jac}^{\rho}_{\zeta B, \cdots, -\zeta A}$ to $\lif{X}_{A}$. By Lemma~\ref{lemma: shift}, 
\begin{align*}
{\rm Jac}^{\rho}_{\zeta B, \cdots, -\zeta A} \lif{X}_{A} & = {\rm Jac}^{\rho}_{\zeta(B+2), \cdots, \zeta A}\lr_{M}\Big(\lq', \e', (\rho, A, B+2, \zeta, \eta_0)\Big) \otimes \eta^{A}_{\rho} \omega^{\zeta \gamma_{B}(A)}_{\rho} \chi_{\tilde{\rho}}  \\
& = \lr^{\Sigma_0}_{M}\Big(\lq', \e', (\rho, A-1, B+1, \zeta, \eta_0)\Big) \otimes \eta^{A}_{\rho} \omega^{\zeta \gamma_{B}(A)}_{\rho} \chi_{\tilde{\rho}} \cdot \omega_{\rho}^{-\zeta((B+2) + \cdots + A)/2} \\
& = \lr^{\Sigma_0}_M \Big(\lq', \e', (\rho, A- 1, B+1, \zeta, \eta_0)\Big) \otimes \eta_{\rho}^{A} \omega_{\rho}^{\zeta((B+1) + \cdots + A)/2} \chi_{\tilde{\rho}}
\end{align*}
which is multiplicity free. This shows multiplicity one. Moreover, 
\[
{\rm Jac}^{\rho}_{\zeta B, \cdots, -\zeta A} \lif{X}_{C}, \, \lif{X}_{\eta} = 0
\]
for $C \neq A$. So the irreducible constituents of $\lr^{\Sigma_0}_{M}(\lq, \e)_{\rm main}$ can not appear in $\lif{X}_{C}, \lif{X}_{\eta}$ for $C \neq A$. 

Next we consider $\lr^{\Sigma_0}_{M}(\lq, \e)_{{\rm com}, \eta}$, where 
\[
\eta_0 = \prod_{C \in [B, A]} (-1)^{C-B} \eta = \eta^{A - B +1} (-1)^{(A - B)(A - B + 1)/2}.
\]
We would like to show that $\lr^{\Sigma_0}_{M}(\lq, \e)_{{\rm com}, \eta}$ ($\eta \in S_{\eta_0}$) appears in 
\[
\begin{cases}
\oplus_{\eta' = \pm} (-1)^{[(A - B +1)/2]} \eta'^{A-B+1} \eta^{A-B}_0 \lif{X}_{\eta'} \quad & \text{ if $\lif{X}_+ \neq \lif{X}_-$} \\
 (-1)^{(A - B +1)/2} \eta_0 \lif{X}_{+} \quad & \text{ if $\lif{X}_+ = \lif{X}_-$ }
\end{cases}
\]
with coefficient $+1$. If $A - B$ is odd, then we have $\eta_0 =  (-1)^{(A - B + 1)/2}$. The sign in front of $\lif{X}_{\eta'}$ becomes
\[
(-1)^{[(A - B +1)/2]} \eta'^{A-B+1} \eta^{A-B}_0 = (-1)^{(A - B + 1)/2} \eta_0 = 1.
\]
By induction, we can assume
\[
(\lif{X}_{\eta'})_{com, \eta''} =\lr^{\Sigma_0}_{M}\Big(\lq', \e', \cup_{C \in [B+1, A]}(\rho, C, C, \zeta, (-1)^{C - (B+1)}\eta''), (\rho, B, B, \zeta, \eta' \eta_0) \Big)
\]
where
\[
\eta' = \prod_{C \in [B+1, A]} (-1)^{C - (B+1)}\eta'' = \eta'' (-1)^{(A - B - 1)/2}.
\]
One checks that
\(
\eta'' = - \eta' \eta_0.
\)
So
\[
(\lif{X}_{\eta'})_{com, \eta''} = \lr^{\Sigma_0}_{M}(\lq, \e)_{{\rm com}, \eta},
\]
where $\eta = \eta'\eta_0$. Moreover, $\lif{X}_+ = \lif{X}_-$ if and only if $\lr^{\Sigma_0}_{M}(\lq, \e)_{{\rm com}, \eta} = \lr^{\Sigma_0}_{M}(\lq, \e)_{com, -\eta}$. 
If $A - B$ is even, then we have $\eta_0 =  (-1)^{(A - B)/2} \eta$. The sign in front of $\lif{X}_{\eta'}$ becomes
\[
(-1)^{[(A - B +1)/2]} \eta'^{A-B+1} \eta^{A-B}_0 = (-1)^{(A - B)/2} \eta'.
\]
By induction, we can assume
\[
(\lif{X}_{\eta'})_{com, \eta''} =\lr^{\Sigma_0}_{M}\Big(\lq', \e', \cup_{C \in [B+1, A]}(\rho, C, C, \zeta, (-1)^{C - (B+1)}\eta''), (\rho, B, B, \zeta, \eta' \eta_0) \Big)
\]
where 
\[
\eta' = \prod_{C \in [B+1, A]} (-1)^{C - (B+1)}\eta'' = (-1)^{(A - B)/2}.
\]
For such $\eta'$, the sign in front of $\lif{X}_{\eta'}$ is positive and we have $\eta'\eta_0 = \eta$. Choose $\eta''$ such that
\(
\eta'' = - \eta' \eta_0,
\)
then
\[
(\lif{X}_{\eta'})_{com, \eta''} = \lr_{M}(\lq, \e)_{com, \eta}.
\]
At last, we still need to show the irreducible constituents of $\lr^{\Sigma_0}_{M}(\lq, \e)_{{\rm com}, \eta}$ do not appear in $\lif{X}_{C}$. This follows from M{\oe}glin's proof of \cite[Theorem 4.1]{Moeglin:2009} that the irreducible constituents of $\lr^{\Sigma_0}_{M}(\lq, \e)_{{\rm com}, \eta} |_{G^{\Sigma_0}(F)}$ do not appear in $\lif{X}_{C} |_{G^{\Sigma_0}(F)}$.

\end{proof}

\begin{corollary}
For any irreducible constituent $\lr^{\Sigma_0}$ in $\lif{X}_{C}, \lif{X}_{\eta}$, it is in $\lr^{\Sigma_0}_{M}(\lq, \e)_{\rm main}$ if and only if
\(
{\rm Jac}^{\rho}_{\zeta B, \cdots, -\zeta A} \lr^{\Sigma_0} \neq 0.
\)
\end{corollary}

Next, we would like to show any irreducible constituent $\lr^{\Sigma_0}$ in $\lif{X}_{C}, \lif{X}_{\eta}$ excluded from $\lr^{\Sigma_0}_{M}(\lq, \e)_{\rm main}, \lr^{\Sigma_0}_{M}(\lq, \e)_{{\rm com}, \eta}$ must be cancelled.

\begin{proposition}
\label{prop: cancellation}
For irreducible constituent $\lr^{\Sigma_0}$ in $\lif{X}_{C}, \lif{X}_{\eta}$, if there exists $x \in [B+1, A]$ such that ${\rm Jac}^{\rho}_{\zeta x} \lr^{\Sigma_0} \neq 0$, then $\lr^{\Sigma_0}$ is cancelled in $\lr^{\Sigma_0}_{M}(\lq, \e)$.
\end{proposition}

The idea is to characterize $\lr^{\Sigma_0}$ by ${\rm Jac}^{\rho}_{\zeta x} \lr^{\Sigma_0}$ and show the cancellation of the corresponding Jacquet modules.

\begin{lemma}
For irreducible constituent $\lr^{\Sigma_0}$ in $\lif{X}_{C}, \lif{X}_{\eta}$ and $x \in [B, A]$, we have ${\rm Jac}^{\rho}_{\zeta x, \zeta x} \lr^{\Sigma_0} = 0$.
\end{lemma}

\begin{proof}
It suffices to show ${\rm Jac}^{\rho}_{\zeta x, \zeta x} \lif{X}_{C} = {\rm Jac}^{\rho}_{\zeta x, \zeta x}  \lif{X}_{\eta} = 0$ for all $x \in [B, A]$, which follows from the same vanishing results for $G^{\Sigma_0}(F)$.
\end{proof}

\begin{corollary}
For irreducible constituent $\lr^{\Sigma_0}_1, \lr^{\Sigma_0}_2$ in $\lif{X}_{C}, \lif{X}_{\eta}$ such that ${\rm Jac}^{\rho}_{\zeta x} \lr^{\Sigma_0}_i \neq 0$ for some $x \in [B, A]$, we have $\lr^{\Sigma_0}_1 = \lr^{\Sigma_0}_2$ if and only if ${\rm Jac}^{\rho}_{\zeta x} \lr^{\Sigma_0}_1 = {\rm Jac}^{\rho}_{\zeta x} \lr^{\Sigma_0}_2$.
\end{corollary}

\begin{proof}
Since ${\rm Jac}^{\rho}_{\zeta x} \lr^{\Sigma_0}_i \neq 0$, there exists an irreducible representation $\tilde{\sigma}^{\Sigma_0}_i$ such that 
\[
\lr^{\Sigma_0}_i \hookrightarrow \rho||^{\zeta_x} \rtimes \tilde{\sigma}^{\Sigma_0}_i
\] 
Since ${\rm Jac}^{\rho}_{\zeta x, \zeta x} \lr^{\Sigma_0}_i = 0$, we must have ${\rm Jac}^{\rho}_{\zeta x} \tilde{\sigma}^{\Sigma_0}_i = 0$. So
\[
{\rm Jac}^{\rho}_{\zeta x} \, (\rho||^{\zeta_x} \rtimes \tilde{\sigma}^{\Sigma_0}_i) = \tilde{\sigma}^{\Sigma_0}_i.
\]
This is means ${\rm Jac}^{\rho}_{\zeta x} \lr^{\Sigma_0}_i = \tilde{\sigma}^{\Sigma_0}_i$ and $\lr^{\Sigma_0}_i$ is the unique irreducible subrepresentation of $\rho||^{\zeta_x} \rtimes \tilde{\sigma}^{\Sigma_0}_i$. The rest is clear.
\end{proof}

For the proof of Proposition~\ref{prop: cancellation}, it remains to show

\begin{lemma}
\label{lemma: cancellation of Jacquet module}
For $x \in [B+1, A]$,
\(
{\rm Jac}^{\rho}_{\zeta x} \lr^{\Sigma_0}_M(\lq, \e) = 0.
\)

\end{lemma}

\begin{proof}
First let us consider $x \in [B + 2, A]$. We will compute ${\rm Jac}^{\rho}_{\zeta x} \lif{X}_{C}$ as follows. If $C = x$, 
\[
{\rm Jac}^{\rho}_x \lif{X}_{C} = \langle \zeta B, \cdots, -\zeta 
(C-1) \rangle_{\rho} \rtimes  {\rm Jac}^{\rho}_{B+2, \cdots, C} \, \lr^{\Sigma_0}_{M}\Big(\lq', \e', (\rho, A, B+2, \zeta, \eta_0)\Big) \otimes \eta^{C}_{\rho} \omega^{\zeta \gamma_{B}(C)}_{\rho} \chi_{\tilde{\rho}} \cdot \eta_{\rho} \omega_{\rho}^{-\zeta C} 
\]
If $C = x - 1$,
\[
{\rm Jac}^{\rho}_x \lif{X}_{C} = \langle \zeta B, \cdots, -\zeta C \rangle_{\rho} \rtimes  {\rm Jac}^{\rho}_{B+2, \cdots, C, C + 1} \, \lr^{\Sigma_0}_{M}\Big(\lq', \e', (\rho, A, B+2, \zeta, \eta_0)\Big) \otimes \eta^{C}_{\rho} \omega^{\zeta \gamma_{B}(C)}_{\rho} \chi_{\tilde{\rho}} 
\]
If $C \neq x, x -1$, ${\rm Jac}^{\rho}_x \lif{X}_{C} = 0$. One checks that
\[
{\rm Jac}^{\rho}_{\zeta C} \lif{X}_{C} = {\rm Jac}^{\rho}_{\zeta C} \lif{X}_{C - 1}.
\]
Moreover, ${\rm Jac}^{\rho}_{\zeta x} \lif{X}_{\eta} = 0$. So ${\rm Jac}^{\rho}_{\zeta x} \lr^{\Sigma_0}_M(\lq, \e) = 0$ for $x \in [B+2, A]$. 

Next let us consider the case that $x = B + 1$. For $C > B + 1$, 
\[
{\rm Jac}^{\rho}_{\zeta (B+1)} \lif{X}_{C} = \langle \zeta B, \cdots, -\zeta 
C \rangle_{\rho} \rtimes  {\rm Jac}^{\rho}_{\zeta (B+1)}{\rm Jac}^{\rho}_{\zeta (B+2), \cdots, \zeta C} \, \lr^{\Sigma_0}_{M}\Big(\lq', \e', (\rho, A, B+2, \zeta, \eta_0)\Big) \otimes \eta^{C}_{\rho} \omega^{\zeta \gamma_{B}(C)}_{\rho} \chi_{\tilde{\rho}} 
\]
To simplify further, we need to expand $\lr^{\Sigma_0}_{M}\Big(\lq', \e', (\rho, A, B+2, \zeta, \eta_0)\Big)$. Let
\[
\lif{X}'_{C'} := \langle \zeta (B+2), \cdots, -\zeta C' \rangle_{\rho} \rtimes  {\rm Jac}_{\zeta(B+4), \cdots, \zeta C'} \, \lr^{\Sigma_0}_{M}\Big(\lq', \e', (\rho, A, B+4, \zeta, \eta_0)\Big) \otimes \eta^{C'}_{\rho} \omega^{\zeta \gamma_{B+2}(C')}_{\rho} \chi_{\tilde{\rho}}
\]
\[
\lif{X}'_{\eta'} := \lr^{\Sigma_0}_M \Big(\lq', \e', (\rho, A, B+3, \zeta, \eta'), (\rho, B+2, B+2, \zeta, \eta'\eta_0) \Big)
\]
If $\lif{X}'_{+} \neq \lif{X}'_{-}$, then
\begin{align*}
\lr^{\Sigma_0}_{M}\Big(\lq', \e', (\rho, A, B+2, \zeta, \eta_0)\Big) = \oplus_{C' \in ]B + 2, A]} (-1)^{A-C'} \lif{X}'_{C'} \oplus_{\eta' = \pm} (-1)^{[(A - B +1)/2] + 1} \eta'^{A-B+1} \eta^{A-B}_0 \lif{X}'_{\eta'} 
\end{align*}
If $\lif{X}'_{+} = \lif{X}'_{-}$, then 
\[
\lr^{\Sigma_0}_{M}\Big(\lq', \e', (\rho, A, B+2, \zeta, \eta_0)\Big) = \oplus_{C' \in ]B + 2, A]} (-1)^{A-C'} \lif{X}'_{C'} \oplus (-1)^{(A - B +1)/2 + 1} \eta_0 \lif{X}'_{+} 
\]
Consider $\lif{X}'_{C'}$ and the corresponding contribution to $\lr^{\Sigma_0}_M(\lq, \e)$. If $C < C'$, we get
\begin{align*}
& (-1)^{C + C'} \langle \zeta B, \cdots, -\zeta C \rangle_{\rho} \times \langle \zeta B, \cdots, -\zeta C' \rangle_{\rho} \rtimes {\rm Jac}^{\rho}_{\zeta (B+3), \cdots, \zeta C} {\rm Jac}^{\rho}_{\zeta (B+4), \cdots, \zeta C'} \\
& \lr^{\Sigma_0}_M \Big(\lq', \e', (\rho, A, B+4, \zeta, \eta_0) \Big) \otimes \eta_{\rho}^{C + C'} \omega^{\zeta (\gamma_{B}(C) + \gamma_{B+2}(C'))}_{\rho} \chi^{2}_{\tilde{\rho}} 
\end{align*}
If $C \geqslant C'$, we get
\begin{align*}
& (-1)^{C + C'} \langle \zeta B, \cdots, -\zeta C \rangle_{\rho} \times \langle \zeta B, \cdots, -\zeta (C'-1) \rangle_{\rho} \rtimes {\rm Jac}^{\rho}_{\zeta (B+3), \cdots, \zeta (C' -1)} {\rm Jac}^{\rho}_{\zeta (B+4), \cdots, \zeta C} \\
& \lr^{\Sigma_0}_M \Big(\lq', \e', (\rho, A, B+4, \zeta, \eta_0) \Big) \otimes \eta_{\rho}^{C + C' - 1} \omega^{\zeta (\gamma_{B}(C) + \gamma_{B+2}(C') - C')}_{\rho} \chi^{2}_{\tilde{\rho}} 
\end{align*}
Note $C > B + 1$, $C' > B+2$. We can pair $(C, C')$ with $(C', C+1)$ for $C < C'$, and their contributions cancel each other. The contribution of $\lif{X}'_{\eta'}$ to $\lr^{\Sigma_0}_M(\lq, \e)$ is
\begin{align}
\label{eq: vanishing}
& (-1)^{[(A - B +1)/2] + 1} \eta'^{A-B+1} \eta^{A-B}_0 \cdot 
 (-1)^{A -C}  \langle \zeta B, \cdots, -\zeta 
C \rangle_{\rho} \rtimes  {\rm Jac}^{\rho}_{\zeta(B+3), \cdots, \zeta C} \\
&  \lr^{\Sigma_0}_{M}\Big(\lq', \e', (\rho, A, B+3, \zeta, \eta'), (\rho, B, B, \zeta, \eta'\eta_0)\Big) \otimes \eta^{C}_{\rho} \omega^{\zeta (\gamma_{B}(C) - (B+2)/2 - (B+1)/2)}_{\rho} \chi_{\tilde{\rho}} 
\end{align}
For $C = B +1$, we also get contribution from
\begin{align}
\label{eq: vanishing 1}
{\rm Jac}^{\rho}_{\zeta (B+1)} \lif{X}_{B+1} = \langle \zeta B, \cdots, -\zeta 
B \rangle_{\rho} \rtimes \lr^{\Sigma_0}_{M}\Big(\lq', \e', (\rho, A, B+2, \zeta, \eta_0)\Big) \otimes \eta^{B+1}_{\rho} \omega^{\zeta \gamma_{B}(B+1)}_{\rho} \chi_{\tilde{\rho}} \cdot \eta_{\rho} \omega_{\rho}^{-\zeta (B+1)}
\end{align}
with sign $(-1)^{A - B - 1}$. To cancel these terms, we need to further expand
\[
\lif{X}_{\eta} := \lr^{\Sigma_0}_M \Big(\lq', \e', (\rho, A, B+1, \zeta, \eta), (\rho, B, B, \zeta, \eta\eta_0) \Big).
\]
Let
\[
\lif{X}_{\eta, C} := \langle \zeta (B+1), \cdots, -\zeta C \rangle_{\rho} \rtimes  {\rm Jac}^{\rho}_{\zeta(B+3), \cdots, \zeta C} \, \lr^{\Sigma_0}_{M}\Big(\lq', \e', (\rho, A, B+3, \zeta, \eta), (\rho, B, B, \zeta, \eta\eta_0)\Big) \otimes \eta^{C}_{\rho} \omega^{\zeta \gamma_{B+1}(C)}_{\rho} \chi_{\tilde{\rho}}
\]
\[
\lif{X}_{\eta, \eta''} := \lr^{\Sigma_0}_M \Big(\lq', \e', (\rho, A, B+2, \zeta, \eta''), (\rho, B+1, B+1, \zeta, \eta''\eta), (\rho, B, B, \zeta, \eta\eta_0) \Big)
\]
It is necessary that $\lif{X}_{\eta, +} \neq \lif{X}_{\eta, -}$. Then
\begin{align*}
\lif{X}_{\eta} = \oplus_{C \in ]B + 1, A]} (-1)^{A-C} \lif{X}_{\eta, C} \oplus_{\eta'' = \pm} (-1)^{[(A - B)/2]} \eta''^{A-B} \eta^{A-B -1} \lif{X}_{\eta, \eta''}.
\end{align*}
Since $\lif{X}_{+} = \lif{X}_{-}$ if and only if $\lif{X}'_{+} = \lif{X}'_{-}$, then one checks that the contributions from $\lif{X}_{\eta, C}$ for $C > B+1$ cancel \eqref{eq: vanishing} after matching $\eta, \eta'$. Also note 
\[
{\rm Jac}^{\rho}_{\zeta (B+1)} \lif{X}_{\eta, \eta''} \neq 0
\]
if and only if $\eta'' = \eta_0$, in which case
\[
{\rm Jac}^{\rho}_{\zeta (B+1)} \lif{X}_{+, \eta_0} \text{ if  $\lif{X}_+ = \lif{X}_{-}$ \, or \, }  \sum_{\eta} {\rm Jac}^{\rho}_{\zeta (B+1)} \lif{X}_{\eta, \eta_0}  \text{ if  $\lif{X}_+ \neq \lif{X}_{-}$ }
\]
is equal to
\[
\langle \zeta B, \cdots, -\zeta 
B \rangle_{\rho} \rtimes \lr^{\Sigma_0}_{M}\Big(\lq', \e', (\rho, A, B+2, \zeta, \eta_0)\Big) \otimes \eta^{B}_{\rho} \omega^{-\zeta (B+1)/2}_{\rho} \chi_{\tilde{\rho}} 
\]
which is exactly \eqref{eq: vanishing 1}. To see the cancellation, we still need to check the sign in front of ${\rm Jac}^{\rho}_{\zeta (B+1)} \lif{X}_{\eta, \eta_0}$
\[
(-1)^{[(A - B +1)/2]} \eta^{A-B+1} \eta^{A-B}_0 \cdot (-1)^{[(A - B)/2]} \eta_0^{A-B} \eta^{A-B -1} = (-1)^{A - B}
\]
which is opposite to that of \eqref{eq: vanishing 1}.
\end{proof}

\begin{lemma}
\label{lemma: special case}
Assume all $(\rho, A', B', \zeta') \in Jord_{\rho}(\q)$ such that $B' > B$ satisfy $\zeta' = \zeta$. If $\lr^{\Sigma_0}$ is in $\lif{X}_{C}$ or $\lif{X}_{\eta}$ such that $\lr^{\Sigma_0} \otimes \eta_{\rho} \ncong \lr^{\Sigma_0}$ and ${\rm Jac}^{\rho}_{\zeta x} \lr^{\Sigma_0} = 0$ for $x \in [B+1, A]$, then
\(
{\rm Jac}^{\rho}_{\zeta B, \cdots, -\zeta A} \lr^{\Sigma_0} \neq 0.
\)
\end{lemma}

\begin{proof}
Since $\lr^{\Sigma_0} \otimes \eta_{\rho} \ncong \lr^{\Sigma_0}$, we claim that   
\begin{align}
\label{eq: standard}
\lr^{\Sigma_0} \hookrightarrow \times_{x' \geqslant y'} \, \langle \zeta x', \cdots, \zeta y' \rangle_{\rho} \rtimes \tilde{\sigma}^{\Sigma_0},
\end{align}
where $x' + y' < 0$ and is increasing, and 
\[
2\sum_{x' \geqslant y'} (x' - y' + 1) = \sum_{(\rho, a, b) \in Jord_{\rho}(\q)} ab,
\]
and $\tilde{\sigma}^{\Sigma_0}$ is an irreducible representation of $\lG^{\Sigma_0}_{-}(F)$. In case $\zeta = +$, this can be achieved by first considering the standard representation containing $\lr$ as the unique irreducible subrepresentation and the fact that the inducing representation can not be invariant under twist by $\eta_{\rho}$. It follows that 
\begin{align}
\label{eq: standard 1}
\lr^{\theta} \hookrightarrow \times_{x' \geqslant y'} \, \langle  x', \cdots,  y' \rangle_{\rho} \rtimes \tilde{\sigma}
\end{align}
for some $\theta \in \Sigma_0$, where $x' + y' < 0$ and is increasing, and 
\[
2\sum_{x' \geqslant y'} (x' - y' + 1) = \sum_{(\rho, a, b) \in Jord_{\rho}(\q)} ab,
\]
and $\tilde{\sigma}$ is an irreducible representation of $\lG_{-}(F)$. Then \eqref{eq: standard} is clear. In case $\zeta = -$, we can apply the same argument to $|{\rm inv} \, \lr|$, the Aubert dual of $\lr$ forgetting the sign. Note ${\rm inv} \, (\lr^{\theta} \otimes \omega) = ({\rm inv} \, \lr)^{\theta} \otimes \omega$ for $\theta \in \Sigma_0, \omega \in X$, and hence $|{\rm inv} \, \lr|^{\Sigma_0} \otimes \eta_{\rho} \ncong |{\rm inv} \, \lr|^{\Sigma_0}$. Suppose
\begin{align*}
|{\rm inv} \, \lr^{\theta}| \hookrightarrow \times_{x' \geqslant y'} \, \langle  x', \cdots,  y' \rangle_{\rho} \rtimes \tilde{\sigma}
\end{align*}
for some $\theta \in \Sigma_0$ as in \eqref{eq: standard 1}. After applying the composition of the cohomological dual $\mathbb{D}$ (cf. \cite{Bernstein:1992}) and the smooth dual, which are both exact contravariant functors, we get
\begin{align*}
\mathbb{D}(|{\rm inv} \, \lr^{\theta}|^{\vee})^{\theta'} \hookrightarrow \times_{x' \geqslant y'} \, \mathbb{D}(\langle  x', \cdots,  y' \rangle^{\vee}_{\rho})^{\vee} \rtimes (\mathbb{D}(\tilde{\sigma}^{\vee}) \otimes \omega')
\end{align*}
for some character $\omega'$ and $\theta' \in \Sigma_0$. Since $\mathbb{D}( (\cdot)^{\vee} )$ induces the same involution as $|{\rm inv} ( \cdot )|$ on irreducible representations (cf. \cite{SS:1997}), then we get
\begin{align*}
(\lr^{\theta})^{\theta'} \hookrightarrow \times_{x' \geqslant y'} \, |{\rm inv} \, \langle  x', \cdots,  y' \rangle_{\rho}|^{\vee} \rtimes (|{\rm inv} \,\tilde{\sigma}| \otimes \omega')  \cong \times_{x' \geqslant y'} \, \langle  -x', \cdots,  -y' \rangle_{\rho} \rtimes (|{\rm inv} \,\tilde{\sigma}| \otimes \omega').
\end{align*}
Then \eqref{eq: standard} is also clear.

Next we would like to consider those $\langle \zeta x', \cdots, \zeta y' \rangle_{\rho}$ in \eqref{eq: standard} containing $\rho||^{\pm(A+1)}$ in its cuspidal support. If it contains $\rho||^{\zeta(A+1)}$, then 
\[
x' \geqslant A + 1, \quad y' < -(A + 1).
\]
So it also contains $\rho||^{-\zeta(A+1)}$. If it contains $\rho||^{-\zeta(A+1)}$ but not $\rho||^{\zeta(A+1)}$, then 
\[
x' < A+1, \quad y' \leqslant -(A+1).
\]
We will first show the second case never occurs. The representation $\langle \zeta x'', \cdots, \zeta y'' \rangle_{\rho}$ in front of $\langle \zeta x', \cdots, \zeta y' \rangle_{\rho}$ in \eqref{eq: standard} satisfies $x'' \leqslant x'$ or $y'' \leqslant y'$. In particular, if the conditions $x'' \leqslant x'$ and $y'' \leqslant y'$ are not both satisfied, then we can interchange $\langle \zeta x', \cdots, \zeta y' \rangle_{\rho}$ with $\langle \zeta x'', \cdots, \zeta y'' \rangle_{\rho}$. As a result, if the second case occurs, then 
\begin{align}
\label{eq: special case}
{\rm Jac}^{\rho}_{\zeta x'', \cdots, \zeta y''} \lr^{\Sigma_0} \neq 0
\end{align}
for some $x'' < A+1$ and $y'' \leqslant -(A+1)$. Since ${\rm Jac}^{\rho}_{\zeta x} \lr^{\Sigma_0} = 0$ for $x \in [B+1, A]$, then $x'' < B + 1$. Since all $(\rho, A', B', \zeta') \in Jord(\q)$ such that $B' > B$ satisfy $\zeta' = \zeta$, we can further show
\begin{align*}
& {\rm Jac}^{\rho}_{\zeta x'', \cdots, \zeta y''} \, \lif{X}_{C} = 0 \\
& {\rm Jac}^{\rho}_{\zeta x'', \cdots, \zeta y''} \, \lif{X}_{\eta} = 0
\end{align*}
for any $x'' \leqslant B$ and $y'' \leqslant -(A+1)$, which contradicts to \eqref{eq: special case}.

Since only the first case occurs, then the number of $\rho||^{\pm (A+1)}$ is even in the cuspidal support of $\lr^{\Sigma_0}$, hence 
\[
\sum_{\substack{(\rho, A', B', \zeta') \in Jord_{\rho}(\q) \\ B' > A}} A' - B' +1
\] 
is even. It follows that the number of $\rho||^{\pm A}$ is odd in the cuspidal support. By the same argument, we see if $\langle \zeta x', \cdots, \zeta y' \rangle_{\rho}$ in \eqref{eq: standard} contains $\rho||^{\zeta A}$, then it must also contain $\rho||^{-\zeta A}$. So there exists $\langle \zeta x', \cdots, \zeta y' \rangle_{\rho}$ such that it only contains $\rho||^{-\zeta A}$, but not $\rho||^{A}$, i.e., $x' < A, y' \leqslant -A$. In the same way, one can show
\[
{\rm Jac}^{\rho}_{\zeta x'', \cdots, \zeta y''} \lr^{\Sigma_0} \neq 0
\]
for some $x'' < B+1$ and $y'' \leqslant -A$. This is only possible when $x'' = B, y'' = -A$. 
\end{proof}

\begin{proposition}
\label{prop: special case} 
Theorem~\ref{thm: discrete diagonal restriction} holds if for any $\rho' \in J(\q)$, either $A' = B'$ for all $(\rho', A', B', \zeta') \in Jord_{\rho'}(\q)$ or there exists $(\rho', A', B', \zeta') \in Jord_{\rho'}(\q)$ with $A' > B'$ such that all $(\rho', A'', B'', \zeta'') \in Jord_{\rho'}(\q)$ with $B'' > B'$ satisfy $\zeta'' = \zeta'$.
\end{proposition}

\begin{proof}
Let $\lr^{\Sigma_0}$ be an irreducible constituent in $\lif{X}_{C}, \lif{X}_{\eta}$ excluded from $\lr^{\Sigma_0}_{M}(\lq, \e)_{\rm main}, \lr^{\Sigma_0}_{M}(\lq, \e)_{{\rm com}, \eta}$. If $X(\lr^{\Sigma_0}) = X^{\Sigma_0}(\tilde{\lambda})$, then $\lr^{\Sigma_0}$ is determined by its restriction to $G^{\Sigma_0}(F)$ together with its $\Sigma_0$-infinitesimal character. In this case, the cancellation of $\lr^{\Sigma_0}$ follows from that of $\lr^{\Sigma_0}|_{G^{\Sigma_0}(F)}$ (cf. \cite[Theorem 4.1]{Moeglin:2009}). Suppose $X(\lr^{\Sigma_0}) \neq X^{\Sigma_0}(\tilde{\lambda})$. Since $X^{\Sigma_0}(\tilde{\lambda}) = X^{\Sigma_0}(\lq_d)$, then $\lr^{\Sigma_0} \otimes \eta_{\rho'} \ncong \lr^{\Sigma_0}$ for some $\rho' \in J(\q)$. It implies that $Jord_{\rho'}(\q)$ can not be elementary. By Lemma~\ref{lemma: independence}, it is enough to consider the case that $\lr^{\Sigma_0}$ is also contained in $\lif{X}'_{C'}$ or $\lif{X}'_{\eta'}$ in the recursive formula with respect some $(\rho', A', B', \zeta') \in Jord(\q)$ with $A' > B'$ satisfying the assumption in the proposition. Then the result follows from Lemma~\ref{lemma: special case} and Lemma~\ref{lemma: independence resolved}. 
\end{proof}

To complete the proof of Theorem~\ref{thm: discrete diagonal restriction}, we still need to remove the assumption in Proposition~\ref{prop: special case}. Let us choose an admissible order $>_{\q}$ such that for any $\rho' \in J(\q)$, either $A' = B'$ for all $(\rho', A', B', \zeta') \in Jord_{\rho'}(\q)$  or there exists $(\rho', A', B', \zeta') \in Jord_{\rho'}(\q)$ with $A' > B'$ and all $(\rho', A'', B'', \zeta'') \in Jord_{\rho'}(\q)$ with $(\rho', A'', B'', \zeta'') >_{\q} (\rho', A', B', \zeta')$ satisfy $\zeta'' = \zeta'$. Let $\q_{\gg}$ be obtained from $\q$ by shifting certain Jordan blocks such that it has discrete diagonal restriction and $>_{\q}$ induces the natural order, then Theorem~\ref{thm: discrete diagonal restriction} holds for $\lq_{\gg}$. In particular, $\lr^{\Sigma_0}_{M}(\lq_{\gg}, \e_{\gg})$ is a representation, where $\e_{\gg}$ is related to $\e$ by the change of order formula, cf. \cite[Theorem 6.3]{Xu:Comb}. By Lemma~\ref{lemma: shift},
\begin{align}
\label{eq: shift 1}
\circ_{(\rho', A', B', \zeta)} {\rm Jac}_{(\rho', A'_{\gg}, B'_{\gg}, \zeta') \mapsto (\rho, A', B', \zeta')} \, \lr^{\Sigma_0}_{M}(\lq_{\gg}, \e_{\gg}) = \lr^{\Sigma_0}_{M}(\lq, \e) \otimes \prod_{(\rho', A', B', \zeta')}\omega^{-|X^{A'_{\gg} - A'}_{(\rho', A', B', \zeta')}|/2}_{\rho'},
\end{align}
where the composition is taken in the decreasing order. As a consequence, $\lr^{\Sigma_0}_{M}(\lq, \e)$ is a representation. Then Theorem~\ref{thm: discrete diagonal restriction} follows from Proposition~\ref{prop: existence} and \eqref{eq: restriction}.

\subsubsection{Stability and character relations}
\label{subsubsec: character relation}

For $[\q] \in \cQ{G}$ with discrete diagonal restriction, we define $\Pkt{\lq}^{\Sigma_0}$ to be the set of irreducible constituents of
\[
\bigoplus_{\tilde{\e} \in \D{\S{\lq}^{\Sigma_0}}} \lr^{\Sigma_0}_{M}(\lq, \e).
\]
Let $\lr^{\Sigma_0}_{W}(\lq, \e) := \lr^{\Sigma_0}_{W}(\lq, \tilde{\e})$ be the direct sum of the preimages of $\tilde{\e}$ in $\Pkt{\lq}^{\Sigma_0}$ under \eqref{eq: Whittaker pairing} and $\lr^{\Sigma_0}_{MW}(\lq, \e) := \lr^{\Sigma_0}_{W}(\lq, \e\e^{MW/W}_{\q})$. It follows from $\r^{\Sigma_0}_{M}(\q, \e) = \r^{\Sigma_0}_{MW}(\q, \e\e^{M/MW}_{\q})$ (cf. \cite[Theorem 7.5]{Xu:Apacket}) that $\lr^{\Sigma_0}_{M}(\lq, \e) = \lr^{\Sigma_0}_{MW}(\lq, \e\e^{M/MW}_{\q})$. For $s \in \S{\q}^{\theta}$ with $\omega = \alpha(s)$, let
\[
(\lr_{M}(\lq, \bar{\e}), A(\theta, \omega)_{M}) := \sum_{\substack{[\lr] \in \cPkt{\lq} \\ \langle \cdot, \lr \rangle_{M} = \tilde{\bar{\e}}}} (\lr, \bar{\e}(s_{\q})A_{\lr}(\theta, \omega)_{M}) \in K\overline{{\rm Rep}}(\lG(F), \theta, \omega)
\]
and
\[
\cPkt{M, s}(\lq) = \sum_{\tilde{\bar{\e}} \in \D{\S{\lq}}} \bar{\e}(s_{\q})(\lr_{M}(\lq, \bar{\e}), A(\theta, \omega)_{M}) \in K\overline{{\rm Rep}}(\lG(F), \theta, \omega),
\]
where $A_{\lr}(\theta, \x)_{M}$ is an intertwining operator between $\lr \otimes \x$ and $\lr^{\theta}$, which is normalized in a way so that if $f$ is the restriction of $\lf$ to $G(F)$, then 
\begin{align*}
(\lf|_{\lif{Z}_{F}G})_{\lG^{\theta}, M}(\lr, \x) = \sum_{\r \subseteq \lr|_{G(F)}} \langle s_{\q}s, \r^+ \rangle_{M} \,  f_{G^{\theta}}(\r), \quad \lf \in \sH(\lG(F), \tilde{\chi}).
\end{align*}
where $\r^{+}$ is an extension of $\r$ to $G^{+}(F) = G(F) \rtimes \langle \theta \rangle$. We also define
\[
(\lr^{\Sigma_0}_{M}(\lq, \e), A(\omega)_{M}) := \sum_{\substack{\lr^{\Sigma_0} \in \Pkt{\lq}^{\Sigma_0}\\ \langle \cdot, \lr^{\Sigma_0} \rangle_{M} = \tilde{\e}}} (\lr^{\Sigma_0}, \e(s_{\q})A_{\lr^{\Sigma_0}}(\omega)_{M}) \in K{\rm Rep}(\lG^{\Sigma_0}(F), \omega).
\]
and
\[
\Pkt{M, s}^{\Sigma_0}(\lq) = \sum_{\tilde{\e} \in \D{\S{\lq}^{\Sigma_0}}} \e(s_{\q})(\lr^{\Sigma_0}_{M}(\lq, \e), A(\omega)_{M}) \in K{\rm Rep}(\lG^{\Sigma_0}(F), \omega).
\]
where $A_{\lr^{\Sigma_0}}(\omega)_{M}$ is an intertwining operator between $\lr^{\Sigma_0} \otimes \x$ and $\lr^{\Sigma_0}$ given by 
\[
\begin{cases}
\lr^{\Sigma_0}(\theta) \circ A_{\lr}(\theta, \x)_{M} \quad & \text{ if $\lr^{\Sigma_0}|_{\lG} = \lr$ irreducible, } \\
\lr^{\Sigma_0}(\theta) \circ (A_{\lr}(\theta, \x)_{M} \oplus A_{\lr^{\theta_0}}(\theta, \x)_{M}) \quad & \text{ if $\lr^{\Sigma_0}|_{\lG} = \lr \oplus \lr^{\theta_0}$.}
\end{cases}
\] 
If $G$ is special even orthogonal, then
\[
(\lf_{\theta})_{\lG^{\Sigma_0}}(\Pkt{M, s}^{\Sigma_0}(\lq)) = 2 \lf_{\lG^{\theta}}(\cPkt{M, s}(\lq)),  \quad \lf \in \sH(\lG(F), \tilde{\chi}),
\]
where $\lf_{\theta} (g \rtimes \theta) : = \lf(g)$ is supported on $\lG(F) \rtimes \theta$. Similarly we can define $\cPkt{MW, s}(\lq)$, $\Pkt{MW, s}^{\Sigma_0}(\lq)$. One can check that
\[
\Pkt{M, s}^{\Sigma_0}(\lq) = \e^{M/MW}_{\q}(s) \, \Pkt{MW, s}^{\Sigma_0}(\lq), \quad \cPkt{M, s}(\lq) = \e^{M/MW}_{\q}(s) \, \cPkt{MW, s}(\lq).
\]
We denote $\cPkt{MW}(\lq) = \cPkt{MW, 1}(\lq)$ and $\cPkt{M}(\lq) = \cPkt{M, 1}(\lq)$. Note $\cPkt{MW}(\lq) = \cPkt{M}(\lq)$. The next proposition shows that $\Pkt{M, s}^{\Sigma_0}(\lq)$ (resp. $\cPkt{M,s}(\lq)$) also satisfies a recursive formula.

\begin{proposition}
\label{prop: recursive formula for packet}
Suppose $[\q] \in \cQ{G}$ has discrete diagonal restriction and we fix $(\rho, A, B, \zeta) \in Jord(\q)$ such that $A > B$, then for any $s \in \S{\q}^{\Sigma_0}$ 
\begin{align*}
\Pkt{M, s}^{\Sigma_0}(\lq) = & \+_{C \in ]B, A]} (-1)^{A-C} \langle \zeta B, \cdots, -\zeta C \rangle_{\rho} \rtimes \Jac^{\rho}_{\zeta (B+2), \cdots, \zeta C} \Pkt{M, s}^{\Sigma_0}(\lq^{1}) \otimes \eta^{C}_{\rho} \omega^{\zeta \gamma_{B}(C)}_{\rho} \chi_{\tilde{\rho}}  \\
& \+ (-1)^{[(A-B+1)/2]} \Pkt{M, s}^{\Sigma_0}(\lq^{2})
\end{align*}
and 
\begin{align*}
\cPkt{M, s}(\lq) = & \+_{C \in ]B, A]} (-1)^{A-C} \langle \zeta B, \cdots, -\zeta C \rangle_{\rho} \rtimes \overline{\Jac}^{\rho}_{\zeta (B+2), \cdots, \zeta C} \cPkt{M, s}(\lq^{1}) \otimes \eta^{C}_{\rho} \omega^{\zeta \gamma_{B}(C)}_{\rho} \chi_{\tilde{\rho}}\\
& \+ (-1)^{[(A-B+1)/2]} \cPkt{M, s}(\lq^{2}).
\end{align*}
\end{proposition}

\begin{proof}
The second equality follows from the first one by taking the map
\begin{align*}
K{\rm Rep}(\lG^{\Sigma_0}(F), \omega) & \longrightarrow K\overline{\rm Rep}(\lG(F), \theta, \omega) \\
(\lr^{\Sigma_0}, A) & \mapsto (\lr^{\Sigma_0}|_{\lG}, \, \lr^{\Sigma_0}(\theta) \circ A).
\end{align*}
Hence it is enough to prove the first one. From the proof of \cite[Lemma 7.6]{Xu:Apacket}, we have $\bar{\e}_{1}(s_{\q^{1}}) = \bar{\e}(s_{\q})$ and 
\[
\bar{\e}_{2}(s_{\q^{2}}) / \bar{\e}(s_{\q}) =  \e_{2}(\rho, A, B+1, \zeta)^{A - B + 1} \e(\rho, A, B, \zeta)^{A - B}.
\]
Then it suffices to show that
\begin{align*}
(\lr^{\Sigma_0}_{M}(\lq, \e), A(\omega)) = & \+_{C \in ]B, A]} (-1)^{A-C} \langle \zeta B, \cdots, -\zeta C \rangle_{\rho} \rtimes \Jac^{\rho}_{\zeta (B+2), \cdots, \zeta C} (\lr^{\Sigma_0}_{M}(\lq^1, \e_1), A(\omega)) \otimes \eta^{C}_{\rho} \omega^{\zeta \gamma_{B}(C)}_{\rho} \chi_{\tilde{\rho}}  \\
& \+_{\tilde{\e} \leftarrow \tilde{\e}_{2} \in \D{\S{\lq^{2}}^{\Sigma_0}}} (-1)^{[(A-B+1)/2]} \e_{2}(\rho, A, B+1, \zeta)^{A - B + 1} \e(\rho, A, B, \zeta)^{A - B} (\lr^{\Sigma_0}_{M}(\lq^2, \e_2), A(\omega)).
\end{align*}
If we forget the intertwining operator $A(\omega)_{M}$, this is a direct consequence of our original definition of $\lr^{\Sigma_0}_{M}(\lq, \e)$ and Theorem~\ref{thm: discrete diagonal restriction}. When $s \in \S{\lq}^{\Sigma_0}$, $A(\omega)_{M}$ acts by multiplying scalar $\e(s) = \e_1(s) = \e_2(s)$, then the equality still holds. In general, we need to show the cancellations in the proof of Theorem~\ref{thm: discrete diagonal restriction} still hold after taking count of the intertwining operators. First, we can extend Proposition~\ref{prop: cancellation} to ${\rm Rep}(\lG^{\Sigma_0}(F), \omega)$. To extend Lemma~\ref{lemma: cancellation of Jacquet module}, one just needs to keep check of the intertwining operators. For the application of Lemma~\ref{lemma: special case}, we do not need to consider the intertwining operators. At last, for those $\lr^{\Sigma_0}$ such that $X(\lr^{\Sigma_0}) = X^{\Sigma_0}(\tilde{\lambda})$ considered in the proof of Proposition~\ref{prop: special case}, it suffices to show their cancellations under $p_1, p_2$ in the set up of Corollary~\ref{cor: determination}. The cancellation under $p_1$ follows from the proof of \cite[Theorem 4.1]{Moeglin:2009}. The cancellation under $p_2$ follows from \cite[Lemma 7.6 and the proof of Theorem 7.14]{Xu:Apacket}.
\end{proof}

Let $(H, \q_{H}) \rightarrow (\q, s)$ and suppose $\q_{H} = \q_{I} \times \q_{II}$. Let $\cPkt{\lq_{H}} = \cPkt{\lq_{I}} \tilde{\otimes} \, \cPkt{\lq_{II}}$ and $c_{I}, c_{II}$ be the $1$-cochains for defining $\cPkt{\lq_{I, d}}, \cPkt{\lq_{II,d}}$ respectively. We choose $1$-cochain $c_s$ to get the twisted endoscopic embedding for $\lif{H}$.  Let ${}'\!c_{\q} = c_{I}c_{II}c_{s}$, then we can define a homomorphism
\(
\chi^{c_{\q}, {}'\!c_{\q}}: W_{F} \rightarrow \mathbb{C}^{\times}
\)
(cf. \cite[Lemma 8.6]{Xu:Lpacket}). We also denote by $\chi^{c_{\q}, {}'\!c_{\q}}$ the corresponding character of $F^{\times}$ by local class field theory.

\begin{theorem}
\label{thm: discrete endoscopy}\,
\begin{enumerate}
\item
\[
\lf_{MW}(\lq) := \sum_{[\lr] \in \cPkt{\lq}} \langle s_{\lq}, \lr \rangle_{MW} \, \lf_{\lG}(\lr), \quad \lf \in \sH(\lG(F), \tilde{\chi})
\]
is stable.

\item For $s \in \S{\q}^{\theta}$ with $\x = \a(s)$ and $(H, \q_{H}) \rightarrow (\q, s)$, the associated twisted character of $\cPkt{MW, s}(\lq)\otimes \chi^{c_{\q}, {}'\!c_{\q}}$ is the twisted endoscopic transfer of that of $\cPkt{MW}(\lq_{H})$.

\end{enumerate}
\end{theorem}

\begin{proof}
Part (1) follows from the recursive formula of $\cPkt{MW}(\lq)$ in Proposition~\ref{prop: recursive formula for packet} and the fact that parabolic induction and Jacquet module preserve stability. The proof of part (2) follows exactly the same line of that of \cite[Theorem 7.5]{Xu:Apacket}. Let us suppose $\q_{s} := \q_{H} = \q_{I} \times \q_{II}$. We can assume $(\rho, A, B, \zeta) \in Jord(\q_{II})$ for the other case is similar. Let $\q^{1}_{s} = \q^{1}_{I} \times \q^{1}_{II} \in \Q{H^{1}}$ and $\q^{2}_{s} = \q^{2}_{I} \times \q^{2}_{II} \in \Q{H^{2}}$. In particular, $\q_{I} = \q^{1}_{I} = \q^{2}_{I}$. The $1$-cochain $c_{s}$ gives rise to twisted endoscopic embeddings for both $\lif{H}^{1}$ and $\lif{H}^{2}$. By induction and \cite[Proposition 8.7]{Xu:Lpacket}, we can assume that the associated twisted character $\cPkt{MW, s}(\lq^{1})\otimes \chi^{c_{\q}, {}'c_{\q}}$ (resp. $\cPkt{MW, s}(\lq^{2})\otimes \chi^{c_{\q}, {}'c_{\q}}$) is the twisted endoscopic transfer of that of $\cPkt{MW}(\lq^{1}_{I}) \,\tilde{\otimes} \, \cPkt{MW}(\lq^{1}_{II})$ (resp. $\cPkt{MW}(\lq^{2}_{I}) \, \tilde{\otimes} \, \cPkt{MW}(\lq^{2}_{II})$). By the compatibility of twisted endoscopic transfer with Jacquet module and parabolic induction, we can conclude that the associated twisted character $\cPkt{M, s}(\lq) \otimes \chi^{c_{\q}, {}'c_{\q}}$ is the twisted endoscopic transfer of that of 
\begin{align*}
& \+_{C \in ]B, A]} (-1)^{A-C} \e^{M/MW}_{\q^{1}}(s) \langle \zeta B, \cdots, -\zeta C \rangle_{\rho} \rtimes \overline{\Jac}^{\rho}_{\zeta (B+2), \cdots, \zeta C} \,  (\cPkt{MW}(\lq^{1}_{I}) \, \tilde{\otimes} \,\cPkt{MW}(\lq^{1}_{II})) \otimes \eta^{C}_{\rho} \omega^{\zeta \gamma_{B}(C)}_{\rho} \chi_{\tilde{\rho}}   \\
&\+ (-1)^{[(A-B+1)/2]}  \e^{M/MW}_{\q^{2}}(s) \, \cPkt{MW}(\lq^{2}_{I}) \, \tilde{\otimes} \, \cPkt{MW}(\lq^{2}_{II}).
\end{align*}
Note $\Jac^{\rho}_{\zeta D} \cPkt{MW}(\lq^{1}_{I}) = 0$ for any $B+2 \leqslant D \leqslant A$. Then we can rewrite it as 
\begin{align*}
& \+_{C \in ]B, A]} (-1)^{A-C} \e^{M/MW}_{\q^{1}}(s) \, \cPkt{MW}(\lq^{1}_{I}) \, \tilde{\otimes} \, \begin{pmatrix} \langle \zeta B, \cdots, -\zeta C \rangle_{\rho} \rtimes \overline{\Jac}^{\rho}_{\zeta (B+2), \cdots, \zeta C} \cPkt{MW}(\lq^{1}_{II}) \otimes \eta^{C}_{\rho} \omega^{\zeta \gamma_{B}(C)}_{\rho} \chi_{\tilde{\rho}} \end{pmatrix} \\
&\+ (-1)^{[(A-B+1)/2]}  \e^{M/MW}_{\q^{2}}(s) \, \cPkt{MW}(\lq^{2}_{I}) \, \tilde{\otimes} \, \cPkt{MW}(\lq^{2}_{II}).
\end{align*}
Since
\begin{align}
\label{eq: M/MW discrete 1}
\e^{M/MW}_{\q}(s) = \e^{M/MW}_{\q^{1}}(s) = \e^{M/MW}_{\q^{2}}(s),
\end{align}
then the associated twisted character $\cPkt{M, s}(\lq) \otimes \chi^{c_{\q}, {}'c_{\q}}$ is the twisted endoscopic transfer of that of $\e^{M/MW}_{\q}(s) \, \cPkt{MW}(\lq_{I}) \, \tilde{\otimes} \, \cPkt{MW}(\lq_{II})$. 
\end{proof}

\subsection{General case}
\label{subsec: general case}

We first assume $\q = \q_{p}$ and fix an admissible order $>_{\q}$ on $Jord(\q)$. We say $\q_{\gg}$ with order $>_{\q_{\gg}}$ dominates $\q$ with respect to $>_{\q}$, if there is an order-preserving bijection between $Jord(\q_{\gg})$ and $Jord(\q)$, which sends $(\rho, A_{\gg}, B_{\gg}, \zeta_{\gg})$ to $(\rho, A, B, \zeta)$ such that $A_{\gg} - A = B_{\gg} - B \geqslant 0$ and $\zeta_{\gg} = \zeta$. Let us choose a dominating parameter $\q_{\gg}$ of $\q$ with discrete diagonal restriction and the natural order. Identify $\S{\q^{>}}^{\Sigma_{0}} \cong \S{\q_{\gg}}^{\Sigma_{0}}$. For $\e \in \D{\S{\q^{>}}^{\Sigma_0}}$, we define
\begin{align*}
\lr^{\Sigma_0}_{M, >_{\q}}(\lq, \e) := & \circ _{(\rho, A, B, \zeta) \in Jord(\q)} \Jac_{(\rho, A_{\gg}, B_{\gg}, \zeta) \mapsto (\rho, A, B, \zeta)} \lr^{\Sigma_0}_{M}(\lq_{\gg}, \e) \\ 
& \otimes_{(\rho, A, B, \zeta) \in Jord(\q)} \omega^{|X^{A_{\gg} - A}_{(\rho, A, B, \zeta)}|/2}_{\rho},
\end{align*}
and
\begin{align*}
\lr_{M, >_{\q}}(\lq, \bar{\e}) := & \circ _{(\rho, A, B, \zeta) \in Jord(\q)} \overline{\Jac}_{(\rho, A_{\gg}, B_{\gg}, \zeta) \mapsto (\rho, A, B, \zeta)} \lr_{M}(\lq_{\gg}, \bar{\e}) \\
& \otimes_{(\rho, A, B, \zeta) \in Jord(\q)} \omega^{|X^{A_{\gg} - A}_{(\rho, A, B, \zeta)}|/2}_{\rho},
\end{align*}
where the compositions are taken in the decreasing order respectively. 

\begin{lemma}
$\lr^{\Sigma_0}_{M, >_{\q}}(\lq, \e)$ (resp. $\lr_{M, >_{\q}}(\lq, \bar{\e})$) is independent of the choice of $\q_{\gg}$.
\end{lemma}

\begin{proof}
Let $\lr^{\Sigma_0}_{M, >_{\q}}(\lq, \e)^{i}$ be defined with respect to $\q^{i}_{\gg}$ for $i = 1, 2$. We can choose $\q_{\gg}$ that dominates both $\q^{i}_{\gg}$ with discrete diagonal restriction and the natural order. By Lemma~\ref{lemma: shift},
\begin{align*}
\lr^{\Sigma_0}_{M}(\lq^{i}_{\gg}, \e) = &  \circ _{(\rho, A, B, \zeta) \in Jord(\q)} \Jac_{(\rho, A_{\gg}, B_{\gg}, \zeta) \mapsto (\rho, A^{i}_{\gg}, B^{i}_{\gg}, \zeta)} \lr^{\Sigma_0}_{M}(\lq_{\gg}, \e) \\
& \otimes_{(\rho, A, B, \zeta) \in Jord(\q)} \omega^{|X^{A_{\gg} - A^{i}_{\gg}}_{(\rho, A^{i}_{\gg}, B^{i}_{\gg}, \zeta)}|/2}_{\rho}.
\end{align*}
Since
\begin{align*}
& \circ _{(\rho, A, B, \zeta) \in Jord(\q)} \Jac_{(\rho, A^{i}_{\gg}, B^{i}_{\gg}, \zeta) \mapsto (\rho, A, B, \zeta)} \, \circ _{(\rho, A, B, \zeta) \in Jord(\q)} \Jac_{(\rho, A_{\gg}, B_{\gg}, \zeta) \mapsto (\rho, A^{i}_{\gg}, B^{i}_{\gg}, \zeta)} \\ 
= &   \circ _{(\rho, A, B, \zeta) \in Jord(\q)} \Jac_{(\rho, A_{\gg}, B_{\gg}, \zeta) \mapsto (\rho, A, B, \zeta)},
\end{align*}
then
\begin{align*}
\lr^{\Sigma_0}_{M, >_{\q}}(\lq, \e)^{i} = &  \circ _{(\rho, A, B, \zeta) \in Jord(\q)} \Jac_{(\rho, A^{i}_{\gg}, B^{i}_{\gg}, \zeta) \mapsto (\rho, A, B, \zeta)} \lr^{\Sigma_0}_{M}(\lq^{i}_{\gg}, \e) \\
& \otimes_{(\rho, A, B, \zeta) \in Jord(\q)} \omega^{|X^{A^{i}_{\gg} - A}_{(\rho, A, B, \zeta)}|/2}_{\rho} \\
= & \circ _{(\rho, A, B, \zeta) \in Jord(\q)} \Jac_{(\rho, A_{\gg}, B_{\gg}, \zeta) \mapsto (\rho, A, B, \zeta)} \lr^{\Sigma_0}_{M}(\lq_{\gg}, \e) \\
& \otimes_{(\rho, A, B, \zeta) \in Jord(\q)} \omega^{|X^{A_{\gg} - A}_{(\rho, A, B, \zeta)}|/2}_{\rho}.
\end{align*}
So $\lr^{\Sigma_0}_{M, >_{\q}}(\lq, \e)^{1} = \lr^{\Sigma_0}_{M, >_{\q}}(\lq, \e)^{2}$. The case of $\lr_{M, >_{\q}}(\lq, \bar{\e})$ also follows from this.
\end{proof}

For functions $\ul(\rho, A, B, \zeta) \in [0, [(A-B+1)/2]]$ and $\ueta(\rho, A, B, \zeta) \in \Two$ on $Jord(\q)$ such that
\[
\e_{\ul, \ueta}(\rho, A, B, \zeta) := \ueta(\rho, A, B, \zeta)^{A - B + 1} (-1)^{[(A-B+1)/2] + \ul(\rho, A, B, \zeta)}
\]
defines a character $\e_{\ul, \ueta}$ of $\S{\q^{>}}^{\Sigma_{0}}$, we define
\begin{align*}
\lr_{M, >_{\q}}^{\Sigma_{0}}(\q, \ul, \ueta) := & \circ _{(\rho, A, B, \zeta) \in Jord(\q)} \Jac_{(\rho, A_{\gg}, B_{\gg}, \zeta) \mapsto (\rho, A, B, \zeta)} \lr_{M}^{\Sigma_{0}}(\q_{\gg}, \ul, \ueta) \\
& \otimes_{(\rho, A, B, \zeta) \in Jord(\q)} \omega^{|X^{A_{\gg} - A}_{(\rho, A, B, \zeta)}|/2}_{\rho} \\
\lr_{M, >_{\q}}(\q, \ul, \ueta) := & \circ _{(\rho, A, B, \zeta) \in Jord(\q)} \Jac_{(\rho, A_{\gg}, B_{\gg}, \zeta) \mapsto (\rho, A, B, \zeta)} \lr_{M}(\q_{\gg}, \ul, \ueta) \\
& \otimes_{(\rho, A, B, \zeta) \in Jord(\q)} \omega^{|X^{A_{\gg} - A}_{(\rho, A, B, \zeta)}|/2}_{\rho},
\end{align*}
where the composition is taken in the decreasing order,  
\[
\ul(\rho, A, B, \zeta) = \ul(\rho, A_{\gg}, B_{\gg}, \zeta) \text{ and } \ueta(\rho, A, B, \zeta) = \ueta(\rho, A_{\gg}, B_{\gg}, \zeta).
\] 
Note
\begin{align}
\label{eq: general restriction}
\lr_{M, >_{\q}}^{\Sigma_{0}}(\lq, \ul, \ueta)|_{G^{\Sigma_0}(F)} = \bigoplus_{(\ul', \ueta') \sim_{\lG^{\Sigma_0}} (\ul, \ueta) / \sim_{\Sigma_0}} \r_{M, >_{\q}}^{\Sigma_{0}}(\q, \ul', \ueta').
\end{align}
So $\lr_{M, >_{\q}}^{\Sigma_{0}}(\lq, \ul, \ueta)$ is irreducible or zero. The same is true of $\lr_{M, >_{\q}}(\q, \ul, \ueta)$.

In general, we define
\[
\lr^{\Sigma_0}_{M, >_{\q}}(\lq, \e) := \Big(\times_{(\rho, a, b) \in Jord(\q_{np})}Sp(St(\rho, a), b) \Big) \rtimes \lr^{\Sigma_0}_{M}(\lq_{p}, \e)
\]
and
\[
\lr_{M, >_{\q}}(\lq, \bar{\e}) := \Big(\times_{(\rho, a, b) \in Jord(\q_{np})}Sp(St(\rho, a), b) \Big) \rtimes \lr_{M}(\lq_{p}, \bar{\e}).
\]
From the definitions it is clear that Lemma~\ref{lemma: restriction} extends to the general case. We define $\Pkt{\lq}^{\Sigma_0}$ (resp. $\cPkt{\lq}$) to be the set of irreducible constituents of
\[
\bigoplus_{\tilde{\e} \in \D{\S{\lq^{>}}^{\Sigma_0}}} \lr^{\Sigma_0}_{M, >_{\q}}(\lq, \e) \quad \Big(\text{resp. $\bigoplus_{\tilde{\bar{\e}} \in \D{\S{\lq^{>}}}} \lr_{M, >_{\q}}(\lq, \bar{\e})$} \Big)
\]
which will be shown to be independent of the order $>_{\q}$. Let $\lr^{\Sigma_0}_{W}(\lq, \e) := \lr^{\Sigma_0}_{W}(\lq, \tilde{\e})$ be the the direct sum of the preimages of $\tilde{\e}$ in $\Pkt{\lq}^{\Sigma_0}$ under \eqref{eq: Whittaker pairing} and $\lr^{\Sigma_0}_{MW, >_{\q}}(\lq, \e) := \lr^{\Sigma_0}_{W}(\lq, \e\e^{MW/W}_{\q})$. Since $\r^{\Sigma_0}_{M, >_{\q}}(\q, \e) = \r^{\Sigma_0}_{MW, >_{\q}}(\q, \e\e^{M/MW}_{\q})$ (cf. \cite[Proposition 8.2]{Xu:Apacket}), then $\lr^{\Sigma_0}_{M, >_{\q}}(\lq, \e) = \lr^{\Sigma_0}_{MW, >_{\q}}(\lq, \e\e^{M/MW}_{\q})$. Hence,
\[
\lr^{\Sigma_{0}}_{M, >_{\q}}(\lq, \e) = \begin{cases}
                           \lr^{\Sigma_{0}}_{W}(\lq, \e \e^{M/MW}_{\q} \e^{MW/W}_{\q}), & \text{ if $\e \e^{M/MW}_{\q} \e^{MW/W}_{\q} \in \D{\S{\q}^{\Sigma_{0}}}$,} \\
                            0, & \text{ otherwise.}
                          \end{cases}
\]
For $\ul(\rho, A, B, \zeta) \in [0, [(A-B+1)/2]]$ and $\ueta(\rho, A, B, \zeta) \in \Two$ on $Jord(\q_{p})$ such that $\e_{\ul, \ueta} \in \D{\S{\q^{>}}^{\Sigma_{0}}}$, we also define
\[
\lr^{\Sigma_{0}}_{M}(\lq, \ul, \ueta) = \Big(\times_{(\rho, a, b) \in Jord(\q_{np})}Sp(St(\rho, a), b) \Big) \rtimes \lr^{\Sigma_{0}}_{M}(\lq_{p}, \ul, \ueta)
\]
and
\[
\lr_{M}(\lq, \ul, \ueta) = \Big(\times_{(\rho, a, b) \in Jord(\q_{np})}Sp(St(\rho, a), b) \Big) \rtimes \lr_{M}(\lq_{p}, \ul, \ueta).
\]
We still have \eqref{eq: general restriction} and $\lr_{M, >_{\q}}^{\Sigma_{0}}(\lq, \ul, \ueta)$ is irreducible or zero by the similar result for $G^{\Sigma_0}(F)$ (cf. \cite[Theorem 6]{Moeglin1:2006}).

For $s \in \S{\q^{>}}^{\theta}$ with $\omega = \alpha(s)$, we can define $\cPkt{M, s}(\lq), \Pkt{M, s}^{\Sigma_{0}}(\lq), \cPkt{MW, s}(\lq), \Pkt{MW, s}^{\Sigma_{0}}(\lq)$ as in the case of discrete diagonal restriction. In particular,
\[
\Pkt{M, s}^{\Sigma_{0}}(\lq) = \Big(\times_{(\rho, a, b) \in Jord(\q_{np})}Sp(St(\rho, a), b) \Big) \rtimes \Pkt{M, s}^{\Sigma_0}(\lq_{p}),
\]
where
\begin{align*}
\Pkt{M, s}^{\Sigma_0}(\lq_{p}) = & \circ _{(\rho, A, B, \zeta) \in Jord(\q_p)} \Jac_{(\rho, A_{\gg}, B_{\gg}, \zeta) \mapsto (\rho, A, B, \zeta)} \, \Pkt{M, s}^{\Sigma_0}(\lq_{p, \gg}) \\
& \otimes_{(\rho, A, B, \zeta) \in Jord(\q)} \omega^{|X^{A_{\gg} - A}_{(\rho, A, B, \zeta)}|/2}_{\rho}.
\end{align*}

\begin{theorem}
\label{thm: general endoscopy}
\begin{enumerate}
\,
\item
\[
\lf_{MW}(\lq) := \sum_{[\lr] \in \cPkt{\lq}} \langle s_{\lq}, \lr \rangle_{MW} \, \lf_{\lG}(\lr), \quad \lf \in \sH(\lG(F), \tilde{\chi})
\]
is stable.

\item For $s \in \S{\q^{>}}^{\theta}$ with $\x = \a(s)$ and $(H, \q_{H}) \rightarrow (\q, s)$, $\cPkt{MW, s}(\lq)\otimes \chi^{c_{\q_{p, \gg}}, {}'c_{\q_{p, \gg}}}$ is the twisted endoscopic transfer of $\cPkt{MW}(\lq_{H})$.

\end{enumerate}
\end{theorem}

\begin{proof}
Part (1) follows from the stability of $\cPkt{MW}(\lq_{p, \gg})$ and the fact that parabolic induction and Jacquet restriction preserve stability. Part (2) follows from the result for $\lq_{p, \gg}$ and the fact that the twisted endoscopic transfer is compatible with parabolic induction and Jacquet restriction.
\end{proof}

\begin{corollary}
\begin{enumerate}
\item
$\Pkt{\lq}^{\Sigma_0}, \cPkt{\lq}$ are independent of the choice of admissible order $>_{\q}$ on $Jord(\q_p)$.

\item If $[\q] \in \cPbd{G}$, then $\cPkt{\lq}$ is the tempered $L$-packets constructed in \cite[Theorem 4.6]{Xu:2018}.
\end{enumerate}
\end{corollary}

\begin{proof}
By induction on the number of Jordan blocks we can assume that $\cPkt{MW}(\lq_{H})$ is independent of the admissible order. By definition, $c_{\q_{p, \gg}}$ is also independent of the order. We can further choose the $1$-cochains $c_{s}$ to be the same, then both the corresponding twisted endoscopic embedding and ${}'c_{\q_{p, \gg}}$ are the same. By Theorem~\ref{thm: general endoscopy}, $\cPkt{MW, s}(\lq)$ is the same, and determines $\cPkt{\lq}$ uniquely. At last, $\Pkt{\lq}^{\Sigma_0}$ is determined by $\cPkt{\lq}$. This proves (1). For (2), note
\[
\cPkt{\lq} = \Big(\times_{(\rho, a, 1) \in Jord(\q_{np})} St(\rho, a) \Big) \rtimes \cPkt{\lq_p},
\]
which also applies to the tempered $L$-packets by our construction in \cite{Xu:2018}. So it suffices to consider $\lq_{p}$. To apply Theorem~\ref{thm: general endoscopy} to $\lq_p$, we first assume $\cPkt{\q_{H}}$ is the tempered $L$-packet by induction on the number of Jordan blocks. Then the result follows from the twisted endoscopic character relation for the tempered $L$-packets (cf. \cite[Theorem 4.6]{Xu:2018}) and $\cPkt{\lq_p}$ (cf. Theorem~\ref{thm: general endoscopy}).
 
\end{proof}


\begin{theorem}
\label{thm: involution good parity}
For $[\q] \in \cQ{G}$ such that $\q = \q_{p}$, and $Jord_{\rho}(\q) \neq \emptyset$,
\begin{align*}
|{\rm inv}^{\rho}_{\infty}| \lr^{\Sigma_0}_{M, >_{\q}}(\lq, \ul, \ueta) & = \lr^{\Sigma_0}_{M, >_{\q}}(\lq^{\sharp}, \ul, \ueta) \\
|\overline{\rm inv}^{\rho}_{\infty}| \lr_{M, >_{\q}}(\lq, \ul, \ueta) & = \lr^{\Sigma_0}_{M, >_{\q}}(\lq^{\sharp}, \ul, \ueta), 
\end{align*}
where $\q^{\sharp}$ is obtained from $\q$ by changing $(\rho, a, b) \in Jord(\q)$ to $(\rho, b, a)$.
\end{theorem}

\begin{proof}
It suffices to show the first equality. When $\q$ is elementary, this is a consequence of Lemma~\ref{lemma: involution}. Suppose $\q$ has discrete diagonal restriction. For $(\ul, \ueta)$, let us assume $\ul(\rho', A, B, \zeta) > 0$ for some $(\rho', A, B, \zeta) \in Jord(\q)$, then
\[
\Jac^{\rho'}_{\zeta B, \cdots, -\zeta A} \, \r^{\Sigma_0}_{M, >_{\q}}(\q, \ul, \ueta) = \r^{\Sigma_0}_{M, >_{\q}}(\q', \ul', \ueta') \otimes \eta^{A}_{\rho} \omega^{\zeta((B+1) + \cdots + A)/2}_{\rho} \chi_{\tilde{\rho}},
\]
where $\q'$ is obtained from $\q$ by changing $(\rho', A, B, \zeta)$ to $(\rho', A-1, B+1, \zeta)$, and $\ul'(\rho', A-1,B+1,\zeta) = \ul(\rho', A, B, \zeta) - 1$. By induction on $\sum_{(\rho', A', B', \zeta') \in Jord(\q)} (A' - B')$, we can assume 
\[
|{\rm inv}^{\rho}_{\infty}| \lr^{\Sigma_0}_{M, >_{\q}}(\q', \ul', \ueta') = \lr^{\Sigma_0}_{M, >_{\q}}(\q'^{\sharp}, \ul', \ueta').
\]
If $\rho' = \rho$, then by \eqref{eq: involution and Jac}
\begin{align*}
& \Jac^{\rho}_{-\zeta B, \cdots, \zeta A} |{\rm inv}^{\rho}_{\infty}| \lr^{\Sigma_0}_{M, >_{\q}}(\q, \ul, \ueta)  = |{\rm inv}^{\rho}_{\infty}| \Jac^{\rho}_{\zeta B, \cdots, -\zeta A} \lr^{\Sigma_0}_{M, >_{\q}}(\q, \ul, \ueta) \otimes \eta^{A - B +1}_{\rho} \omega^{-\zeta((B+1) + \cdots A)}_{\rho} \\
& = |{\rm inv}^{\rho}_{\infty}| \lr^{\Sigma_0}_{M, >_{\q}}(\q', \ul', \ueta') \otimes \eta^{A + (A- B+1)}_{\rho} \omega^{-\zeta((B+1) + \cdots A)/2}_{\rho} \chi_{\tilde{\rho}} = \lr^{\Sigma_0}_{M, >_{\q}}(\q'^{\sharp}, \ul', \ueta') \otimes \eta^{A}_{\rho} \omega^{-\zeta((B+1) + \cdots A)/2}_{\rho} \chi_{\tilde{\rho}},
\end{align*}
where we have used the fact that  $\lr^{\Sigma_0}_{M, >_{\q}}(\q'^{\sharp}, \ul', \ueta')$ is invariant under twist by $\eta_{\rho}$ when $A - B + 1$ is odd. Then
\[
|{\rm inv}^{\rho}_{\infty}| \lr^{\Sigma_0}_{M, >_{\q}}(\q, \ul, \ueta) \hookrightarrow \rho||^{-\zeta B} \times \cdots \times \rho||^{\zeta A} \rtimes \r^{\Sigma_0}_{M, >_{\q}}(\q'^{\sharp}, \ul', \ueta') \otimes \eta^{A}_{\rho} \omega^{-\zeta((B+1) + \cdots A)/2}_{\rho} \chi_{\tilde{\rho}}.
\]
Since $\lr^{\Sigma_0}_{M, >_{\q}}(\q^{\sharp}, \ul, \ueta)$ is the unique irreducible subrepresentation of the right hand side, then 
\[
|{\rm inv}^{\rho}_{\infty}| \lr^{\Sigma_0}_{M, >_{\q}}(\q, \ul, \ueta) = \lr^{\Sigma_0}_{M, >_{\q}}(\q^{\sharp}, \ul, \ueta).
\] 
If $\rho' \neq \rho$, then by \eqref{eq: involution and Jac} again
\begin{align*}
& \Jac^{\rho'}_{\zeta B, \cdots, -\zeta A} |{\rm inv}^{\rho}_{\infty}| \lr^{\Sigma_0}_{M, >_{\q}}(\q, \ul, \ueta) = |{\rm inv}^{\rho}_{\infty}| \Jac^{\rho'}_{\zeta B, \cdots, -\zeta A} \lr^{\Sigma_0}_{M, >_{\q}}(\q, \ul, \ueta) \\
& = |{\rm inv}^{\rho}_{\infty}| \lr^{\Sigma_0}_{M, >_{\q}}(\q', \ul', \ueta') \otimes \eta^{A}_{\rho} \omega^{\zeta((B+1) + \cdots + A)/2}_{\rho} \chi_{\tilde{\rho}} = \lr^{\Sigma_0}_{M, >_{\q}}(\q'^{\sharp}, \ul', \ueta') \otimes \eta^{A}_{\rho} \omega^{\zeta((B+1) + \cdots + A)/2}_{\rho} \chi_{\tilde{\rho}}.
\end{align*}
Then
\[
|{\rm inv}^{\rho}_{\infty}| \lr^{\Sigma_0}_{M, >_{\q}}(\q, \ul, \ueta) \hookrightarrow \rho'||^{\zeta B} \times \cdots \times \rho'||^{-\zeta A} \rtimes \r^{\Sigma_0}_{M, >_{\q}}(\q'^{\sharp}, \ul', \ueta') \otimes \eta^{A}_{\rho} \omega^{\zeta((B+1) + \cdots + A)/2}_{\rho} \chi_{\tilde{\rho}}.
\]
Since $\lr^{\Sigma_0}_{M, >_{\q}}(\q^{\sharp}, \ul, \ueta)$ is the unique irreducible subrepresentation of the right hand side, then 
\[
|{\rm inv}^{\rho}_{\infty}| \lr^{\Sigma_0}_{M, >_{\q}}(\q, \ul, \ueta) = \lr^{\Sigma_0}_{M, >_{\q}}(\q^{\sharp}, \ul, \ueta).
\] 
At last, the case of good parity follows from the definition and \eqref{eq: involution and Jac}.
\end{proof}

So far we have constructed $\cPkt{\lq}$ and shown that it satisfies the properties (1), (2) in Theorem~\ref{thm: Apacket}. We have also proved (3) when $H$ is elliptic (cf. the end of Subsection~\ref{subsec: Arthur parameters}). To treat the remaining cases, we need to show

\begin{proposition}
\label{prop: induction of A-packet}
For $\q = \q_p \in \Q{G}$, if $\q$ factors through 
\[
\q_{M} := (\prod_{i = 1}^{l} \rho_{i} \otimes \nu_{a_i} \otimes \nu_{b_i}) \times \q_{-} \in \Q{M}
\]
for $M = \prod_{i = 1}^{l} GL(n_i) \times G_{-}$ and $\rho_{i} \otimes \nu_{a_i} \otimes \nu_{b_i} \in \Q{GL(n_i)}$, $\q_{-} \in \Q{G_{-}}$, then
\begin{align}
\label{eq: induction of A-packet}
\cPkt{W}(\lq) = \prod_{i} Sp(St(\rho_i, a_i), b_i) \rtimes \cPkt{W}(\lq_{-}) \, \otimes_{i} \chi^{A_i - B_i + 1}_{\tilde{\rho_i}}
\end{align}
\end{proposition} 

This will be proved in the next section.

\begin{corollary}
\label{cor: character relation general}
Part (3) of Theorem~\ref{thm: Apacket} holds for $\cPkt{\lq}$.
\end{corollary}

\begin{proof}
For $\theta \in \Sigma_0$, $s \in \cS{\q}^{\theta}$ with $\x = \a(s)$ and $(H, \q_{H}) \rightarrow (\q, s)$, the result follows from Theorem~\ref{thm: general endoscopy} when $H$ is elliptic. If $H$ is not elliptic, then it is a Levi subgroup of some $(\theta, \x)$-twisted elliptic endoscopic group $G'$ of $G$ and $\q_{H}$ induces $\q' \in \Q{G'}$. By Proposition~\ref{prop: induction of A-packet}, 
\(
\lf^{\lif{H}}_{W}(\lq_H) = \lf^{\lG'}_{W}(\lq')
\)
for some choices of $\cPkt{\lq_{H}}$ and $\cPkt{\lq'}$. So it reduces to the elliptic case.
\end{proof}


\section{Uniqueness}
\label{sec: uniqueness}

To complete the proof of Theorem~\ref{thm: Apacket}, we still need to show Proposition~\ref{prop: induction of A-packet} and the uniqueness part of the theorem. First, we would like to show in many cases $\cPkt{\lq}$ can be uniquely characterized by properties (1), (2) of Theorem~\ref{thm: Apacket} up to twists by $X$.

\begin{theorem}
\label{thm: uniqueness}
Suppose $\q = \q_p \in \Q{G}$ and $Jord_{\rho}(\q)$ is indexed for each $\rho$ satisfying
\begin{align}
\label{eq: ladder}
\zeta_{\rho} := \zeta_{i+1} = \zeta_{i}, \quad A_{i+1} \geqslant A_i, \quad B_{i+1} \geqslant B_i.
\end{align}
We fix the order $>_{\q}$ on $Jord_{\rho}(\q)$ so that
\[
(\rho, A_{i}, B_{i}, \zeta_{i}) >_{\q} (\rho, A_{i-1}, B_{i-1}, \zeta_{i-1}).
\]

\begin{enumerate}

\item The packet $\cPkt{\lq}$ is the unique subset of $\clPkt{\q, \tilde{\zeta}}$ up to twists by $X$ such that the properties (1), (2) of Theorem~\ref{thm: Apacket} are satisfied. 

\item If there exists a stable distribution supported on $\clPkt{\q, \tilde{\zeta}}$, i.e.,
\[
S(\lf) = \sum_{\lr \in \clPkt{\q, \tilde{\zeta}}} c_{\lr} \, \lf_{\lG}(\lr), \quad \lf \in \sH(\lG(F), \tilde{\chi}),
\]
then it must be of the form
\[
S(\lf) = \sum_{\omega \in X/\a(\S{\q}^{\Sigma_0})} c_{\omega} \, \lf_{W}(\lq \otimes \omega), \quad \lf \in \sH(\lG(F), \tilde{\chi}).
\]

\end{enumerate}
\end{theorem}

We will postpone its proof until the end of this section.

\begin{corollary}
Suppose $[\q] \in \cQ{G}$ satisfies \eqref{eq: ladder}. Then any stable distribution supported on $\cPkt{\lq}$ (resp. $\cPkt{\q}$) is a scalar multiple of $\lf_{W}(\lq)$ (resp. $f_{W}(\q)$).
\end{corollary}

\begin{proof}
For $\cPkt{\lq}$, it follows directly from part (2) of Theorem~\ref{thm: uniqueness}. Any stable distribution supported on $G(F)$ is invariant under the conjugation of $\lG(F)$, so the case of $\cPkt{\q}$ follows from that for $\cPkt{\lq}$ by restricting $\lf$ to $\lif{Z}_{F}G(F)$.
\end{proof}

\begin{corollary}
The uniqueness part of Theorem~\ref{thm: Apacket} holds.
\end{corollary}

\begin{proof}
By property (3) of Theorem~\ref{thm: Apacket}, it suffices to show the uniqueness in the case that $\q = \rho \otimes \nu_a \otimes \nu_b$, to which Theorem~\ref{thm: uniqueness} applies.
\end{proof}

We will deduce Proposition~\ref{prop: induction of A-packet} from Theorem~\ref{thm: uniqueness}. The critical step is to show the following special case.

\begin{lemma}
For $\q = \q_p = 2 \rho \otimes \nu_a \otimes \nu_b \in \Q{G}$, 
\[
\cPkt{W}(\lq) =  Sp(St(\rho, a), b) \rtimes \chi^{A - B + 1}_{\tilde{\rho}}.
\]
\end{lemma}

\begin{proof}
It is known that $\cPkt{W}(\q) =  Sp(St(\rho, a), b) \rtimes  1$, so the right hand side is supported on $\clPkt{\q, \tilde{\zeta}}$, where $\tilde{\zeta}$ is the central character of $\cPkt{W}(\lq)$. Since the right hand side is also stable, then by Theorem~\ref{thm: uniqueness} it is equal to
\[
\sum_{\x \in X/\a(\S{\q}^{\Sigma_0})} c_{\x} \, \cPkt{W}(\lq) \otimes \x.
\]
By comparing the $\Sigma_0$-infinitesimal characters, it suffices to sum over $X^{\Sigma_0}(\tilde{\lambda})/\a(\S{\q}^{\Sigma_0})$. Note $X^{\Sigma_0}(\tilde{\lambda}) = \langle \eta_\rho \rangle$. Then it is equal to
\(
\cPkt{W}(\lq) \text{ or } \cPkt{W}(\lq) \otimes \eta_{\rho}.
\)

Suppose $\cPkt{W}(\lq) \neq \cPkt{W}(\lq) \otimes \eta_{\rho}$, then it is necessary that $a, b$ are even. Let
\[
X = \begin{bmatrix}
    \zeta B & \cdots & -\zeta A  \\
\vdots & & \vdots \\
\zeta (A + B - 1)/2 & \cdots & -\zeta (A + B + 1)/2
  \end{bmatrix}.
\]
For $\ul(\rho, A, B, \zeta) = (A - B + 1)/2$, we can show 
\begin{align*}
\lr_{M, >_{\q}}(\lq, \ul, 1) & \hookrightarrow \begin{pmatrix} \zeta B & \cdots & -\zeta A  \\
\vdots & & \vdots \\
\zeta (A + B - 1)/2 & \cdots & -\zeta (A + B + 1)/2
\end{pmatrix} \times \begin{pmatrix} \zeta B & \cdots & -\zeta A  \\
\vdots & & \vdots \\
\zeta (A + B - 1)/2 & \cdots & -\zeta (A + B + 1)/2
\end{pmatrix}
\\
& \rtimes \omega_{\rho}^{-|X|} \chi_{\tilde{\rho}}^{(A - B + 1)}
\end{align*}
as the unique irreducible subrepresentation following similar argument in \cite[Theorem 4.1]{Xu:Non}. 
It suffices to show that $\lr_{M, >_{\q}}(\lq, \ul, 1) \otimes \eta_{\rho}$ is not contained in $Sp(St(\rho, a), b) \rtimes \chi^{A - B + 1}_{\tilde{\rho}}$. This can be checked by computing
\(
\overline{\Jac}^{\rho}_{X} \circ \overline{\Jac}^{\rho}_{X}.
\)
In particular, 
\[
\overline{\Jac}^{\rho}_{X} \circ \overline{\Jac}^{\rho}_{X} \, (Sp(St(\rho, a), b) \rtimes \chi^{A - B + 1}_{\tilde{\rho}}) = 2^{(A - B + 1)/2} \cdot \omega_{\rho}^{-|X|} \chi_{\tilde{\rho}}^{(A - B + 1)},
\]
whereas $\overline{\Jac}^{\rho}_{X} \circ \overline{\Jac}^{\rho}_{X} \, \lr_{M, >_{\q}}(\lq, \ul, 1) \otimes \eta_{\rho}$ contains $\eta_{\rho} \omega_{\rho}^{-|X|} \chi_{\tilde{\rho}}^{(A - B + 1)}$.
\end{proof}

Now we can give the proof of Proposition~\ref{prop: induction of A-packet}.

\begin{proof}
By induction on the number of general linear factors in $M$, it reduces to show the case that $M = GL(n_1) \times G_{-}$. From the above lemma,
\[
\cPkt{W}(\lq_{1}) = Sp(St(\rho_1, a_1), b_1) \rtimes \chi^{A_1 - B_1 + 1}_{\tilde{\rho}_1}
\]
for $\q_{1} := 2 \rho_1 \otimes \nu_{a_1} \otimes \nu_{b_1} \in \Q{G_1}$. Since $M$ is also a Levi subgroup of the elliptic endoscopic group $G_{1} \times G_{-}$ of $G$, then 
\[
Sp(St(\rho_1, a_1), b_1) \rtimes \cPkt{W}(\lq_{-}) \otimes \chi^{A_1 - B_1 + 1}_{\tilde{\rho}_1}
\]
is equal to the endoscopic transfer of
\(
\cPkt{W}(\lq_{1}) \, \tilde{\otimes} \, \cPkt{W}(\lq_{-}),
\)
where the endoscopic embedding is defined with respect to the trivial $1$-cochain $c_{s}$. Moreover, $c_{\q, \gg} = {}'c_{\q, \gg}$. So this is equal to $\cPkt{W}(\lq)$ by Theorem~\ref{thm: general endoscopy}.
\end{proof}

Before starting the proof of Theorem~\ref{thm: uniqueness}, let us look at some of the implications of the condition \eqref{eq: ladder}. 

\begin{lemma}
\label{lemma: L-packet in A-packet}
Suppose $\q = \q_p \in \Q{G}$ satisfies \eqref{eq: ladder} and $\zeta_{\rho} = +$ for all $\rho$. Then
\[
\clPkt{\p_{\q}, \tilde{\zeta}} \cap \cPkt{\lq}= \Big\{\lr_{M, >_{\q}}(\lq, \ul, \ueta) \in \cPkt{\lq} \, | \, \ul(\rho, A, B, \zeta) = [(A - B + 1)/2] \Big\}.
\]
\end{lemma}

\begin{proof}
It suffices to show that
\[
\cPkt{\p_{\q}} = \Big\{\r_{M, >_{\q}}(\q, \ul, \ueta) \in \cPkt{\q} \, | \, \ul(\rho, A, B, \zeta) = [(A - B + 1)/2]  \Big\}.
\] 
Note the Langlands parameters of elements on the right hand side have been given in \cite[Theorem 4.1]{Xu:Non}. So the result follows directly from there.
\end{proof}

\begin{remark}
We will see that $\clPkt{\p_{\q}, \tilde{\zeta}} \cap \cPkt{\lq} = \cPkt{\p_{\lq}}$ in the next section.
\end{remark}

\begin{corollary}
\label{cor: tempered}

Suppose $\q = \q_p \in \Q{G}$ satisfies \eqref{eq: ladder} and $\zeta_{\rho} = +$ for all $\rho$.

\begin{enumerate}
\item
\(
\Big\{\lr_{M, >_{\q}}(\lq, \ul, \ueta) \in \cPkt{\lq} \, | \, \ul = 0 \Big\} = \cPkt{\lq}
\)
if and only if $\q$ is trivial on the second $SL(2, \mathbb{C})$.

\item If $\q$ is not trivial on the second $SL(2, \mathbb{C})$, then there are no nontrivial stable distributions supported on 
\[
\bigcup_{\omega \in X / \a(\S{\q}^{\Sigma_0})} \Big\{\lr_{M, >_{\q}}(\lq, \ul, \ueta) \in \cPkt{\lq} \, | \, \ul = 0 \Big\} \otimes \omega. 
\]
\end{enumerate}
\end{corollary}

\begin{proof}
If $\q$ is not trivial on the second $SL(2, \mathbb{C})$, then $\clPkt{\p_{\q}, \tilde{\zeta}} \cap \cPkt{\lq}$ is nonempty and disjoint from $\Big\{\lr_{M, >_{\q}}(\lq, \ul, \ueta) \in \cPkt{\lq} \, | \, \ul = 0 \Big\}$. So (1) is clear. Part (2) follows from Theorem~\ref{thm: uniqueness} and Lemma~\ref{lemma: twist}.
\end{proof}

Now we can start the proof of Theorem~\ref{thm: uniqueness}. 

\begin{proof}
It suffices to show part (2). Let $\q^{\sharp}$ be obtained from $\q$ by changing $\zeta_{\rho}$ to $-\zeta_{\rho}$ for all $\rho$ such that $\zeta_{\rho} = -$. By Theorem~\ref{thm: involution good parity}, 
\[
\circ_{\rho: \zeta_{\rho} = -} \, |\overline{\rm inv}^{\rho}_{\infty}| \,\, \lr_{M, >_{\q}}(\lq, \ul, \ueta) = \lr_{M, >_{\q}}(\lq^{\sharp}, \ul, \ueta).
\]
So it suffices to consider the case that $\zeta_{\rho} = +$ for all $\rho$. We will prove it by induction on $\sum_{(\rho, A, B, \zeta) \in Jord(\q)} (A - B)$. If $\sum_{(\rho, A, B, \zeta) \in Jord(\q)} (A - B) = 0$, then it follows from the tempered case (cf. \cite[Corollary 4.8]{Xu:2018}). So let us assume $\sum_{(\rho, A, B, \zeta) \in Jord(\q)} (A - B) > 0$. Suppose 
\[
R := \sum_{\omega \in X / \a(\S{\q}^{\Sigma_0})} \, \sum_{(\ul, \ueta)/\sim_{\lG^{\Sigma_0}}} c^{\omega}_{\ul, \ueta} \, \lr_{M, >_{\q}}(\lq, \ul, \ueta) \otimes \omega
\]
is stable. By subtracting linear combinations of $\lf_W(\lq \otimes \omega)$ for $\omega \in X/\a(\S{\q}^{\Sigma_0})$, we can assume $c^{\omega}_{\ul, \ueta} = 0$ if $\ul(\rho, A, B, \zeta) = [(A - B + 1)/2]$ and $\ueta = 1$. Then it is enough to show $R = 0$. 

Suppose $R \neq 0$, then there exists $(\ul, \ueta)$ such that 
\[
\sum_{\x \in X / \a(\S{\q}^{\Sigma_0})} c^{\omega}_{\ul, \ueta} \, \lr_{M, >_{\q}}(\lq, \ul, \ueta) \otimes \omega \neq 0.
\]
If $\ul(\rho, A, B, +) \neq 0$ for some $(\rho, A, B, +)$, then over $Jord_{\rho}(\q) = \{(\rho, A_i, B_i, +)\}^{n}_{i= 1}$ we can choose $t \leqslant s$ such that 
\(
A = A_{s} = \cdots = A_{t}, B = B_{s} = \cdots = B_{t}
\)
with 
\[
B_{s + 1} > B_{s} \text{ or } s = n
\]
and
\[
A_{t} > A_{t-1} \text{ or } t = 1.
\] 
Consider the map
\[
\frac{1}{(s - t + 1)!} \circ_{i = t}^{s} \, (\overline{{\rm Jac}}^{\rho}_{B_i, \cdots, -A_i}(\cdot) \otimes \eta^{-A_i}_{\rho} \x^{-((B_i + 1) + \cdots + A_i)/2}_{\rho} \chi^{-1}_{\tilde{\rho}}): \clPkt{\q, \tilde{\zeta}} \rightarrow \clPkt{\q', \tilde{\zeta}'} \cup \{0\},
\]
where $\q'$ is obtained from $\q$ by replacing $(\rho, A_i, B_i, +)$ by $(\rho, A_{i} - 1, B_{i} + 1, +)$ for $t \leqslant i \leqslant s$. It is a bijection on the preimage of $\clPkt{\q', \tilde{\zeta}'}$ (cf. \cite[Lemma 4.3]{Xu:Non}). Since the image of $R$ is stable and can not be a nontrivial linear combination of $\lf_W(\lq' \otimes \omega)$ for $\x \in X$, then it must be zero by the induction assumption. But the image of $\lr_{M, >_{\q}}(\lq, \ul, \ueta)$ is nonzero, so we get a contradiction. 

It remains to consider the case that
\[
R = \sum_{\omega \in X / \a(\S{\q}^{\Sigma_0})} \, \sum_{(0, \ueta)/\sim_{\lG^{\Sigma_0}}} c^{\omega}_{0, \ueta} \, \lr_{M, >_{\q}}(\lq, 0, \ueta) \otimes \omega.
\]
Let us choose $\rho$ and minimal $t$ over the index set of $Jord_{\rho}(\q)$ such that $A_t - B_t \neq 0$. Then we divide it into the following two cases.

\begin{itemize}

\item $B_{t} \neq B_{t+1}$: We can break $(\rho, A_t, B_t, +)$ into $(\rho, A_t, B_t + 1, +)$ and $(\rho, B_t, B_t, +)$, and denote the new parameter by $\q'$. It also satisfies \eqref{eq: ladder}. Note 
\[
\Big\{\lr_{M, >_{\q}}(\lq, \ul, \ueta) \in \cPkt{\lq} \, | \, \ul = 0 \Big\} \subseteq \Big\{\lr_{M, >_{\q'}}(\lq', \ul', \ueta') \in \cPkt{\lq'} \, | \, \ul' = 0 \Big\}.
\]
If $\q'$ is not trivial on the second $SL(2, \mathbb{C})$, then $R = 0$ by the induction assumption, cf. Corollary~\ref{cor: tempered}. If $\q'$ is trivial on the second $SL(2, \mathbb{C})$, then it is necessary that $A_t = B_t + 1$. In this case, 
\[
\Big\{\lr_{M, >_{\q}}(\lq, \ul, \ueta) \in \cPkt{\lq} \, | \, \ul = 0 \Big\} \subsetneq \cPkt{\lq'}.
\]
Hence $R = 0$.

\item $B_t = B_{t + 1} = \cdots = B_s < B_{s +1}$: We break $(\rho, A_s, B_s, +)$ into $(\rho, A_s, B_s + 1, +)$ and $(\rho, B_s, B_s, +)$, and move $(\rho, B_s, B_s, +)$ after $(\rho, A_t, B_t, +)$ with respect to the order $>_{\q}$. Denote the new parameter by $\q'$. It also satisfies \eqref{eq: ladder}. Since we have assumed $R \neq 0$ stable, then we must have
\begin{align*}
\clPkt{\p_{\q'}, \tilde{\zeta}'} \cap \cPkt{\lq'} \subseteq \Big\{\lr_{M, >_{\q}}(\lq, \ul, \ueta) \in \cPkt{\lq} \, | \, \ul = 0 \Big\} \subseteq \cPkt{\lq'}
\end{align*}
by the induction assumption. So we can assume $A_i - B_i = 1$ for $t \leqslant i \leqslant s$ and it follows from the change of order formula (cf. \cite[Theorem 6.1]{Xu:Comb}) and Lemma~\ref{lemma: L-packet in A-packet} that 
\[
\clPkt{\p_{\q'}, \tilde{\zeta}'} \cap \cPkt{\lq'} \supseteq \Big\{\lr_{M, >_{\q}}(\lq, \ul, \ueta) \in \cPkt{\lq} \, | \, \ul = 0 \Big\}.
\]
Take
\begin{align*}
\frac{1}{(s - t)!}\circ_{i = t}^{s - 1} \, (\overline{{\rm Jac}}^{\rho}_{B_i, \cdots, -A_i}(\cdot) \otimes \eta^{-A_i}_{\rho} \x^{-((B_i + 1) + \cdots + A_i)/2}_{\rho} \chi^{-1}_{\tilde{\rho}}): \cPkt{\lq'} \rightarrow \cPkt{\lq''} \cup \{0\},
\end{align*}
where $\q''$ is obtained from $\q'$ by removing $(\rho, A_i, B_i, +)$ for $t \leqslant i \leqslant s -1$. It induces a bijection from $\clPkt{\p_{\q'}, \tilde{\zeta}'} \cap \cPkt{\lq'}$ to $\cPkt{\lq''}$. So the stability of $R \neq 0$ implies that the image of 
\[
\Big\{\lr_{M, >_{\q}}(\lq, \ul, \ueta) \in \cPkt{\lq} \, | \, \ul = 0 \Big\} 
\]
should be all of $\cPkt{\lq''}$. On the other hand, the image consists of $\lr_{M, >_{\q}}(\lq'', 0, \ueta'')$ with 
\[
\ueta''(\rho, B_s, B_s, +) = (-1)^{s - t + 1}\ueta''(\rho, A_s, A_s, +)
\] 
by direct computation.
\begin{enumerate}
\item If $s - t +1$ is odd, then the image can not be all of $\cPkt{\lq''}$. So we get a contradiction. 
\item If $s - t + 1$ is even, then we can further assume that $(\rho, B_s, B_s, +)$ and $(\rho, A_s, A_s, +)$ both have even multiplicities in
\(
Jord(\q) \backslash \{(\rho, A_i, B_i, +)\}_{i = t}^{s},
\)
otherwise the image can not be all of $\cPkt{\lq''}$. As a consequence,
\[
\Big\{\lr_{M, >_{\q'}}(\lq', \ul', \ueta') \in \cPkt{\lq'} \, | \, \ul' = 0 \Big\} \neq \emptyset.
\]
Note $\alpha(\S{\q'}^{\Sigma_0}) = \alpha(\S{\q'_d}^{\Sigma_0})$. Then $R$ can not be stable by the induction assumption. So we get a contradiction again.
\end{enumerate}
\end{itemize}
\end{proof}


\section{$L$-packets inside Arthur packets}
\label{sec: Arthur LLC}

For $[\q] \in \cQ{G}$, it follows from the discussion in \cite[Section 2.2]{Arthur:2013} that $\cPkt{\p_{\q}} \subseteq \cPkt{\q}$. This result is also obtained in \cite[Proposition 6.0.3]{Moeglin1:2011} by different methods. In this section, we would like to extend this result to $\lG(F)$ (see Lemma~\ref{lemma: stable expansion}). As an application, we will extend the local Langlands correspondence constructed in Section~\ref{sec: LLC} to the case of Arthur parameters and Arthur packets.

First we will deduce $\cPkt{\p_{\q}} \subseteq \cPkt{\q}$ following Arthur. By \cite[(2.2.12)]{Arthur:2013},
\begin{align}
\label{eq: stable expansion}
f_{W}(\q) = \sum_{[\p] \in \cP{G}_{\lambda}} m([\q], [\p]) f(\p), \quad f \in \sH(G(F), \chi),
\end{align}
where $\p$ factors through $\p_{M} \otimes \xi$ for $[\p_{M}] \in \cPbd{M}$ and ${\rm Re} \, \xi \in (\mathfrak{a}^{*}_{P})^{+}$,
\begin{align}
\label{eq: standard stable character}
f(\p) = \sum_{[\sigma] \in \cPkt{\p_M}} f_{G}( {\rm Ind}^{G(F)}_{P(F)}\, \sigma \otimes \xi ).
\end{align}
We fix the twisted endoscopic embedding
\(
\iota_{G}: \L{G} \rightarrow GL(N, \mathbb{C})
\)
and $\D{\theta}_N$-stable $\Gamma_{F}$-splitting $(\mathcal{B}_{N}, \mathcal{T}_{N}, \{\mathcal{X}_{\alpha_N}\})$ of $GL(N, \mathbb{C})$ and $\D{\theta}_0$-stable $\Gamma_{F}$-splitting $(\mathcal{B}_{G}, \mathcal{T}_{G}, \{\mathcal{X}_{\alpha_G}\})$ of $\D{G}$, such that $\iota_{G}(\mathcal{T}_{G}) = (\mathcal{T}_{N}^{\D{\theta}_{N}})^0$ and $\iota_{G}(\mathcal{B}_{G}) \subseteq \mathcal{B}_{N}$. Let $A_{\D{P}} := (Z(\D{M})^{\Gamma_F})^0$ and $\mathfrak{a}_{\D{P}} := X_{*}(A_{\D{P}}) \otimes \mathbb{R}$. There is a commutative diagram 
\[
\xymatrix{\mathfrak{a}^{*}_{P} \ar@{^{(}->}[r] \ar[d]_{\cong} & \mathfrak{a}^{*}_{B} \ar[d]^{\cong}  \\
\mathfrak{a}_{\D{P}} \ar@{^{(}->}[r] & \mathfrak{a}_{\mathcal{B}_{G}}
}
\]
where the vertical arrows send roots to coroots. So we can also identify the positive Weyl chambers $(\mathfrak{a}^{*}_{B})^{+} \cong (\mathfrak{a}_{\mathcal{B}_{G}})^{+}$, and the cones ${}^{+}(\mathfrak{a}^{*}_{B}) \cong {}^{+}(\mathfrak{a}_{\mathcal{B}_{G}})$ spanned by positive roots (coroots). Via $\iota_{G}$ we also have an inclusion $\mathfrak{a}_{\mathcal{B}_{G}} \subseteq \mathfrak{a}^{\D{\theta}_{N}}_{\mathcal{B}_{N}} \cong \mathbb{R}^{n}$. We define partial orders on $\mathfrak{a}_{\mathcal{B}_{G}}$ by $\Xi \leqslant_{G} \Xi'$ if $\Xi' - \Xi \in \overline{^{+}(\mathfrak{a}_{\mathcal{B}_{G}})}$ (resp. $\Xi \leqslant_{N} \Xi'$ if $\Xi' - \Xi \in \overline{^{+}(\mathfrak{a}_{\mathcal{B}_N})}$). Since $\overline{^{+}(\mathfrak{a}_{\mathcal{B}_{G}})} \subseteq \overline{^{+}(\mathfrak{a}_{\mathcal{B}_N})}$, then 
\(
\Xi \leqslant_{G} \Xi' \Rightarrow \Xi \leqslant_{N} \Xi'.
\) 
For $\Xi \in \overline{(\mathfrak{a}_{\mathcal{B}_{G}})^{+}}$, there exists a unique element in $\{\Xi, \Xi^{\D{\theta}_0}\} \cap \overline{(\mathfrak{a}_{\mathcal{B}_{N}})^{+}}$, denoted by $\Xi^{+}$. Let $\Xi_{[\p]} := ({\rm Re} \, \xi )^{+}$ and $\Xi_{[\q]} := \Xi_{[\p_{\q}]}$. It follows from \cite[Section 2.2]{Arthur:2013} that
\begin{enumerate}
\item $m([\q], [\p_{\q}]) = 1$,
\item if $m([\q], [\p]) \neq 0$, then $\Xi_{[\p]} \leqslant_{N} \Xi_{[\q]}$ and the equality holds only when $[\p] = [\p_{\q}]$.
\end{enumerate}

\begin{lemma}
\label{lemma: order}
Suppose $\Xi, \Xi' \in \overline{(\mathfrak{a}_{\mathcal{B}_{G}})^{+}}$ such that $\Xi \leqslant_{G} \Xi'$, then $\Xi^{+} \leqslant_{G} \Xi'^{+}$. Moreover, $\Xi^{+} = \Xi'^{+} \Rightarrow \Xi = \Xi'$.
\end{lemma}

\begin{proof}
If $G$ is $Sp(2n)$ or $SO(2n, \eta)$ for $\eta \neq 1$, then $\Xi^{+} = \Xi$. So there is nothing to prove. When $G = SO(2n)$, we can write $\Xi = (\mu_{i})_{i=1}^{n} \in \mathbb{R}^{n}$, where $\mu_{1} \geqslant \cdots \geqslant \mu_{n-1} \geqslant |\mu_{n}|$. Then $\Xi^{+} = (\mu_1, \cdots, \mu_{n-1}, |\mu_n|)$. The rest is a direct verification, so we omit it here.
\end{proof}

For any irreducible representation $\r$ of $G(F)$, we realize it as the Langlands quotient $J(P, \sigma, \xi)$ of ${\rm Ind}^{G(F)}_{P(F)}\, \sigma \otimes \xi$ and denote $\Xi_{[\r]} := ({\rm Re} \, \xi)^{+}$. It follows from \cite[Lemma 2.14]{BorelWallach:2000} and Lemma~\ref{lemma: order} that the irreducible constituents $\r'$ of ${\rm Ind}^{G(F)}_{P(F)}\, \sigma \otimes \xi$ satisfies $\Xi_{[\r']} \leqslant_{G} \Xi_{[\r]}$ and the equality holds if and only if $\r' = J(P, \sigma, \xi)$ . 

\begin{lemma}
\label{lemma: estimate of exponent} \,
\begin{enumerate}
\item $\cPkt{\p_{\q}} \subseteq \cPkt{\q}$.
\item For any $[\r] \in \cPkt{\q}$, $\Xi_{[\r]} \leqslant_{N} \Xi_{[\q]}$ and the equality holds only when $[\r] \in \cPkt{\p_{\q}}$.
\item If $m([\q], [\p]) \neq 0$ and $[\p] \neq [\p_{\q}]$, the character $f^{G}(\p)$ does not have contributions from $\cPkt{\p_{\q}}$.
\end{enumerate}
\end{lemma} 

\begin{proof}
Suppose $m([\q], [\p]) \neq 0$. Let $\r$ be any irreducible representation contributing to $f^{G}(\p)$. From the above discussion, $\Xi_{[\r]} \leqslant_{G} \Xi_{[\p]} \leqslant_{N} \Xi_{[\q]}$. Then $\Xi_{[\r]} \leqslant_{N} \Xi_{[\p]} \leqslant_{N} \Xi_{[\q]}$. If the equality holds, then $[\r] \in \cPkt{\p}$ and $[\p] = [\p_{\q}]$. This proves (2). If $[\p] \neq [\p_{\q}]$, then $\Xi_{[\p]} <_{N} \Xi_{[\q]}$ and hence $\Xi_{[\r]} <_{N} \Xi_{[\q]}$. So $[\r] \notin \cPkt{\p_{\q}}$. This proves (3). At last, (1) follows from (3) and \eqref{eq: stable expansion}.
\end{proof}

Now let us fix a local Langlands correspondence 
\begin{align}
\label{eq: LLC}
\cPkt{}(\lG(F)) \rightarrow \cP{\lG}
\end{align}
as in Section~\ref{sec: LLC}. We would like to extend \eqref{eq: stable expansion} to $\lG(F)$. For $\lq \in \Q{\lG}$, let $\tilde{\lambda} = \lambda_{\lq}$. For $\lp \in \P{\lG}_{\tilde{\lambda}}$ factoring through $\lp_{M} \otimes \tilde{\xi}$ for $[\lp_{M}] \in \cPbd{\lif{M}}$ and ${\rm Re} \, \tilde{\xi} \in (\mathfrak{a}^{*}_{\lif{P}})^{+} \subseteq \overline{(\mathfrak{a}^{*}_{\lif{B}})^{+}}$, we define
\[
\lf(\lp) = \sum_{[\sigma] \in \cPkt{\lp_M}} \lf_{\lG}( {\rm Ind}^{\lG(F)}_{\lif{P}(F)}\, \tilde{\sigma} \otimes \tilde{\xi} ).
\]
Fix $\Gamma_{F}$-splitting $(\mathcal{B}_{\lG}, \mathcal{T}_{\lG}, \{\mathcal{X}_{\alpha_{\lG}}\})$ of $\D{\lG}$ projecting onto that of $\D{G}$. We have a commutative diagram
\[
\xymatrix{0 \ar[r] & \mathfrak{a}^{*}_{\lG} \ar[r] \ar[d]_{\cong} & \mathfrak{a}^{*}_{\lif{B}} \ar[d]_{\cong} \ar[r] & \mathfrak{a}^{*}_{B} \ar[r] \ar[d]_{\cong} & 0 \\
0 \ar[r] & \mathfrak{a}_{\D{\lG}} \ar[r] & \mathfrak{a}_{\mathcal{B}_{\lG}} \ar[r] & \mathfrak{a}_{\mathcal{B}_{G}} \ar[r] & 0
}
\]
where there are natural splittings of the short exact sequences such that
\[
\mathfrak{a}^{*}_{\lif{B}} \cong \mathfrak{a}^{*}_{B} \oplus \mathfrak{a}^{*}_{\lG}, \quad \mathfrak{a}_{\mathcal{B}_{\lG}} \cong \mathfrak{a}_{\mathcal{B}_{G}} \oplus \mathfrak{a}_{\D{\lG}}.
\]
By our assumption on $\tilde{\lambda}$, we have ${\rm Re} \, \tilde{\xi} \in \mathfrak{a}^{*}_{B} \cong \mathfrak{a}_{\mathcal{B}_{G}}$. Similarly for $[\lr] \in \cPkt{}(\lG(F))_{\tilde{\lambda}}$ realized as the Langlands quotient $J(\lif{P}, \tilde{\sigma}, \tilde{\xi})$, we have ${\rm Re} \, \tilde{\xi} \in \mathfrak{a}^{*}_{B} \cong \mathfrak{a}_{\mathcal{B}_{G}}$. So we can define $\Xi_{[\lp]}$, $\Xi_{[\lq]}$ and $\Xi_{[\lr]}$ as before.


\begin{lemma}
\label{lemma: stable expansion}
For $[\lq] \in \cQ{\lG}$, there exists a unique choice of $\cPkt{\lq}$ such that $\cPkt{\p_{\lq}} \subseteq \cPkt{\lq}$. Moreover, there exists a decomposition
\[
\lf_{W}(\lq) = \sum_{[\lp] \in \cP{\lG}_{\tilde{\lambda}}} m([\lq], [\lp]) \lf(\lp), \quad \lf \in \sH(\lG(F), \tilde{\chi}),
\]
satisfying 
\begin{enumerate}
\item $m([\lq], [\p_{\lq}]) = 1$,
\item if $m([\lq], [\lp]) \neq 0$, then $\Xi_{[\lp]} \leqslant_{N} \Xi_{[\lq]}$ and the equality holds only when $[\lp] = [\p_{\lq}]$.
\end{enumerate}
\end{lemma}

\begin{proof}
For any choice of $\cPkt{\lq}$, we have
\[
\lf_{W}(\lq) = \sum_{[\lr] \in \cPkt{\lq}} \langle s_{\lq}, \lr \rangle_{W} \, \lf_{\lG}(\lr), \quad \lf \in \sH(\lG(F), \tilde{\chi}).
\]
We express $\lf_{\lG}(\lr)$ in terms of standard representations
\[
\lf_{\lG}(\lr) = \sum_{[\lp] \in \cP{\lG}_{\tilde{\lambda}}} \, \sum_{\tilde{\bar{\e}} \in \D{\S{\lp}}} m^{[\lr]}_{([\lp], \tilde{\bar{\e}})} \lf_{\lG}(M(\lp, \tilde{\bar{\e}})).
\]
where $M(\lp, \tilde{\bar{\e}})$ is the standard representation containing $\lr(\lp, \tilde{\bar{\e}})$ as the Langlands quotient. If $m^{[\lr]}_{([\lp], \tilde{\bar{\e}})} \neq 0$, then $\Xi_{[\lp]} \leqslant_{G} \Xi_{[\lr]}$ and the equality holds only when $[\lr] = \lr(\lp, \tilde{\bar{\e}})$ (cf. \cite[Lemma 2.14]{BorelWallach:2000} and Lemma~\ref{lemma: order}). So
\begin{align*}
\lf_{W}(\lq) & = \sum_{[\lr] \in \cPkt{\lq}} \langle s_{\lq}, \lr \rangle_{W} \, \sum_{[\lp] \in \cP{\lG}_{\tilde{\lambda}}} \sum_{\tilde{\bar{\e}} \in \D{\S{\lp}}} m^{[\lr]}_{([\lp], \tilde{\bar{\e}})}  \lf_{\lG}(M(\lp, \tilde{\bar{\e}})) \\
& = \sum_{[\lp] \in \cP{\lG}_{\tilde{\lambda}}} \sum_{\tilde{\bar{\e}} \in \D{\S{\lp}}} \sum_{[\lr] \in \cPkt{\lq}} m^{[\lr]}_{([\lp], \tilde{\bar{\e}})} \langle s_{\lq}, \lr \rangle_{W} \lf_{\lG}(M(\lp, \tilde{\bar{\e}})).
\end{align*}
By the inversion formula of endoscopic character relation,
\[
\lf_{\lG}(M(\lp, \tilde{\bar{\e}})) = \sum_{(\lG', \lp')} m^{(\lp, \tilde{\bar{\e}})}_{\lp'} \lf'(\lp'),
\]
where $\lp$ factors through $\lp' \in \P{\lG'}$ for an elliptic endoscopic group $\lG'$ of $\lG$ and $\lp'$ does not factor through any common Levi subgroup of $\L{\lG'}$ and $\L{\lG}$ unless $\lp' = \lp$, in which case $m^{(\lp, \tilde{\bar{\e}})}_{\lp} = |\S{\lp}|^{-1}$. Then
\[
\lf_{W}(\lq) = \sum_{\lG'} \sum_{\lp'} \Big( \sum_{\tilde{\bar{\e}} \in \D{\S{\lp}}} \sum_{[\lr] \in \cPkt{\lq}} m^{(\lp, \tilde{\bar{\e}})}_{\lp'} m^{[\lr]}_{([\lp], \tilde{\bar{\e}})} \langle s_{\lq}, \lr \rangle_{W} \Big) \lf'(\lp').
\]
We rewrite it as
\[
\lf_{W}(\lq) - \sum_{[\lp] \in \cP{\lG}_{\tilde{\lambda}}} |\S{\lp}|^{-1} \Big(\sum_{\tilde{\bar{\e}} \in \D{\S{\lp}}} \sum_{[\lr] \in \cPkt{\lq}} m^{[\lr]}_{([\lp], \tilde{\bar{\e}})} \langle s_{\lq}, \lr \rangle_{W} \Big) \lf(\lp) = \sum_{\lG' \neq \lG} \sum_{\lp'} \Big( \sum_{\tilde{\bar{\e}} \in \D{\S{\lp}}} \sum_{[\lr] \in \cPkt{\lq}} m^{(\lp, \tilde{\bar{\e}})}_{\lp'}  m^{[\lr]}_{([\lp], \tilde{\bar{\e}})} \langle s_{\lq}, \lr \rangle_{W} \Big) \lf'(\lp').
\]
Since $\lp'$ does not factor through the Levi subgroups of $\L{\lG}$, it follows from Arthur's spectral characterization of the image of $\sH(\lG(F), \tilde{\chi})$ under the endoscopic transfers (cf. \cite{Arthur:1996}) that both sides must be zero. Hence 
\[
\lf_{W}(\lq) = \sum_{[\lp] \in \cP{\lG}_{\tilde{\lambda}}} |\S{\lp}|^{-1} \Big(\sum_{\tilde{\bar{\e}} \in \D{\S{\lp}}} \sum_{[\lr] \in \cPkt{\lq}} m^{[\lr]}_{([\lp], \tilde{\bar{\e}})} \langle s_{\lq}, \lr \rangle_{W} \Big) \lf(\lp).
\]
Let
\[
m([\lq], [\lp]) := |\S{\lp}|^{-1} \sum_{\tilde{\bar{\e}} \in \D{\S{\lp}}} \sum_{[\lr] \in \cPkt{\lq}} m^{[\lr]}_{([\lp], \tilde{\bar{\e}})} \langle s_{\lq}, \lr \rangle_{W}.
\]
If $m([\lq], [\lp]) \neq 0$, then there exists $\tilde{\bar{\e}} \in \D{\S{\lp}}$ and $[\lr] \in \cPkt{\lq}$ such that $m^{[\lr]}_{([\lp], \tilde{\bar{\e}})} \neq 0$. So $\Xi_{[\lp]} \leqslant_{G} \Xi_{[\lr]}$. Since $\Xi_{[\lr]} \leqslant_{N} \Xi_{[\lq]}$ by Lemma~\ref{lemma: estimate of exponent}, then $\Xi_{[\lp]} \leqslant_{N} \Xi_{[\lq]}$. If we further assume $\Xi_{[\lp]} = \Xi_{[\lq]}$, then $\Xi_{[\lp]} = \Xi_{[\lr]} = \Xi_{[\lq]}$, hence $[\lr] = \lr(\lp, \tilde{\bar{\e}}) \in \cPkt{\p_{\lq}} \otimes \omega$ for some $\omega \in X^{\Sigma_0}(\tilde{\lambda})$. It follows that $[\lp] = [\p_{\lq} \otimes \omega]$. In sum, we have
\[
\lf_{W}(\lq) = \sum_{\omega \in X^{\Sigma_0}(\tilde{\lambda})/X^{\Sigma_0}(\lq)} m([\lq], [\p_{\lq} \otimes \omega]) \lf(\p_{\lq} \otimes \omega) + \sum_{\substack{[\lp] \in \cP{\lG}_{\tilde{\lambda}} \\ \Xi_{[\lp]} <_N \Xi_{[\lq]}}} m([\lq], [\lp]) \lf(\lp), \quad \lf \in \sH(\lG(F), \tilde{\chi}).
\]
By \eqref{eq: restriction 1}, 
there can only be one term in the first sum, say $\p_{\lq} \otimes \omega'$. Comparing with \eqref{eq: stable expansion} by restriction, we see $m(\lq, \p_{\lq} \otimes \omega') = 1$. At last, it suffices to change $\cPkt{\lq}$ to $\cPkt{\lq} \otimes \omega'^{-1}$. Then it is also clear that $\cPkt{\p_{\lq}} \subseteq \cPkt{\lq}$. The uniqueness follows from $X^{\Sigma_0}(\lq) = X^{\Sigma_0}(\p_{\lq})$.
\end{proof}

From now on, we associate $[\lq] \in \cQ{\lG}$ with the unique choice of $\cPkt{\lq}$ such that $\cPkt{\p_{\lq}} \subseteq \cPkt{\lq}$. We would like to show this association is compatible with parabolic induction and twisted endoscopic transfer. Let $M \cong GL(n_1) \times G_{-}$ be a standard Levi subgroup of $G$. Via $\iota_{G}$ we embed $\D{M}$ into a $\D{\theta}_{N}$-stable standard Levi subgroup $\D{L}$ of $GL(N, \mathbb{C})$.
\[
\xymatrix{
\L{M} \ar[r] \ar[d] &  \D{L} := GL(n_{1}, \mathbb{C}) \times GL(N_{-}, \mathbb{C}) \times GL(n_1, \mathbb{C})  \ar[d]\\
\L{G} \ar[r]^{\iota_G} & GL(N, \mathbb{C}).
}
\] 
Then $\mathfrak{a}_{\mathcal{B}_{M}} \cong \mathfrak{a}_{\mathcal{B}_{n_1}} \oplus \mathfrak{a}_{\mathcal{B}_{G_{-}}} \subseteq \mathfrak{a}^{\D{\theta}_{N}}_{\mathcal{B}_L} \cong \mathbb{R}^{n_1 + n_2}$. 
We define orders on $\mathfrak{a}_{\mathcal{B}_{M}}$ by $\Xi_{M} \leqslant_{M} \Xi'_{M}$ (resp. $\Xi_{M} \leqslant_{L} \Xi'_{M}$) if $\Xi'_{M} - \Xi_{M} \in \overline{{}^{+}\mathfrak{a}^{\D{M}}_{\mathcal{B}_{M}}}$ (resp. $\Xi'_{M} - \Xi_{M} \in \overline{{}^{+}\mathfrak{a}^{\D{L}}_{\mathcal{B}_{L}}}$). Since $\overline{{}^{+}\mathfrak{a}^{\D{M}}_{\mathcal{B}_{M}}} \subseteq \overline{{}^{+}\mathfrak{a}^{\D{L}}_{\mathcal{B}_{L}}}$, then $\Xi_{M} \leqslant_{M} \Xi'_{M} \Rightarrow \Xi_{M} \leqslant_{L} \Xi'_{M}$. For $\Xi_{M} \in \mathfrak{a}_{\mathcal{B}_{M}} \cap \overline{(\mathfrak{a}_{\mathcal{B}_{L}})^{+}}$, there exists a unique element in $\{w(\Xi_{M}) \in \mathfrak{a}_{\mathcal{B}_{G}} \cap \overline{(\mathfrak{a}_{\mathcal{B}_{N}})^{+}} \, | \, w \in W(\D{G}^{\Sigma_0}, \mathcal{T}_{G})\}$, denoted by $\Xi$.

\begin{lemma}
\label{lemma: combinatorics 1}
Suppose $\Xi_{M}, \Xi'_{M} \in \mathfrak{a}_{\mathcal{B}_{M}} \cap \overline{(\mathfrak{a}_{\mathcal{B}_{L}})^{+}}$ such that $\Xi_{M} \leqslant_{L} \Xi'_{M}$. Then $\Xi \leqslant_N \Xi'$. Moreover, $\Xi = \Xi' \Rightarrow \Xi_{M} = \Xi'_{M}$.
\end{lemma}

\begin{proof}
Let $\D{\theta}_{L}$ be the automorphism of $\D{L}$ permuting the two $GL(n_1, \mathbb{C})$ and fixing $GL(N_{-}, \mathbb{C})$. Then $\D{\theta}_{L}$ acts on $\mathfrak{a}_{\mathcal{B}_{M}} \cap \overline{(\mathfrak{a}_{\mathcal{B}_{L}})^{+}}$ and $\overline{{}^{+}\mathfrak{a}_{\mathcal{B}_{L}}}$. For $\Xi_{M} \in \mathfrak{a}_{\mathcal{B}_{M}} \cap \overline{(\mathfrak{a}_{\mathcal{B}_{L}})^{+}}$, we can write $\Xi_{M} = ({}^{I}\!\mu, {}^{II}\!\mu) \in \mathbb{R}^{n_{1} + n_2}$, where ${}^{I}\!\mu = ({}^{I}\!\mu_i)^{n_1}_{i = 1}, {}^{II}\!\mu = ({}^{II}\!\mu_i)^{n_{2}}_{i=1}$, satisfying
\[
{}^{I}\!\mu_{1} \geqslant \cdots \geqslant {}^{I}\!\mu_{n_1}, \quad {}^{II}\!\mu_1\geqslant \cdots \geqslant {}^{II}\!\mu_{n_{2}} \geqslant 0.
\]
Then $\Xi^{\D{\theta}_{L}}_{M} = ({}^{I}\bar{\mu}, {}^{II}\!\mu)$, where ${}^{I}\bar{\mu} = ({}^{I}\bar{\mu}_i)_{i = 1}^{n_1}$ such that ${}^{I}\!\mu_i + {}^{I}\bar{\mu}_{n_1 - i + 1} = 0$, and $\Xi = (\mu_{i})_{i = 1}^{n}$, where 
\(
\mu_{1} \geqslant \cdots \geqslant \mu_{n} \geqslant 0
\) 
and
\[
\{\mu_1, \cdots, \mu_n\} =  \{{}^{I}\!\mu_1, \cdots, {}^{I}\!\mu_{t}, {}^{I}\bar{\mu}_1, \cdots, {}^{I}\bar{\mu}_{s}, {}^{II}\!\mu_1, \cdots, {}^{II}\!\mu_{n_{2}} \}
\]
as multisets for some nonnegative integers $t, s$ such that $t + s = n_1$ and $\sum_{i = 1}^{t} {}^{I}\!\mu_i + \sum_{i = 1}^{s} {}^{I}\bar{\mu}_i$ is maximal. The rest is a direct verification, so we omit it here.
\end{proof}

For $\lq_{M} \in \Q{\lif{M}}$, let $\tilde{\lambda}_{M} = \lambda_{\lq_{M}}$. For $\lp_{M} = \p_{1} \times \lp_{-} \in \P{\lif{M}}_{\tilde{\lambda}_{M}}$, let $\Xi_{[\lp_{M}]} := \Xi_{\p_1} + \Xi_{[\lp_{-}]} \in \mathfrak{a}_{\mathcal{B}_{M}} \cap \overline{(\mathfrak{a}_{\mathcal{B}_{L}})^{+}}$ and $\Xi_{[\lq_{M}]} := \Xi_{\q_1} + \Xi_{[\lq_{-}]}$.

\begin{proposition}
\label{prop: compatible with parabolic induction}
Suppose $\lq \in \Q{\lG}$ factors through $\lq_M = \q_1 \times \lq_{-} \in \Q{\lif{M}}$ for $\lif{M} \cong GL(n_1) \times \lG_{-}$, then 
\(
\r_{\q_1} \rtimes \cPkt{\lq_{-}} = \cPkt{\lq}.
\)
\end{proposition}

\begin{proof}
By Lemma~\ref{lemma: stable expansion} and analogous result for general linear groups,
\[
\lf^{\lif{M}}_{W}(\lq_{M}) = \sum_{[\lp_{M}] \in \cP{\lif{M}}_{\tilde{\lambda}_{M}}} m([\lq_{M}], [\lp_{M}]) \lf^{\lif{M}}(\lp_{M}), \quad \lf \in \sH(\lG(F), \tilde{\chi})
\]
where $\lf^{\lif{M}}$ is the descent of $\lf$. Suppose $m([\lq_{M}], [\lp_{M}]) \neq 0$, then $\Xi_{[\lp_{M}]} \leqslant_L \Xi_{[\lq_{M}]}$ and the equality holds only when $[\lp_{M}] = [\p_{\lq_{M}}]$. Let $\lp = \iota_{\lif{M}} \circ \lp_{M}$. It follows from Lemma~\ref{lemma: combinatorics 1} that $\Xi_{[\lp]} \leqslant_{N} \Xi_{[\lq]}$ and the equality holds only when $\Xi_{[\lp_{M}]} = \Xi_{[\lq_{M}]}$. By the compatibility of \eqref{eq: LLC} with parabolic induction in the tempered case, we have
\(
\lf^{\lif{M}}(\lp_{M}) = \lf(\lp).
\)
Suppose $[\lp_{M}] \neq [\p_{\lq_{M}}]$, then for any $\lr$ contributing to $\lf(\lp)$, we have
\(
\Xi_{[\lr]} \leqslant_{N} \Xi_{[\lp]} <_{N} \Xi_{[\lq]}.
\)
So $\Xi_{[\lr]} \neq \Xi_{[\lq]}$ and hence $[\lr] \notin \cPkt{\p_{\lq}}$. This shows that $\lf(\lp)$ does not have contributions from $\cPkt{\p_{\lq}}$. On the other hand, 
$\lf^{\lif{M}}(\p_{\lq_{M}}) = \lf(\p_{\lq})$. So $\lf^{\lif{M}}_{W}(\lq_{M})$ has contribution from $\cPkt{\p_{\lq}}$. By Proposition~\ref{prop: induction of A-packet}, we have $\r_{\q_1} \rtimes \cPkt{\lq_{-}} = \cPkt{\lq}$.
\end{proof}

Let $\theta \in \Sigma_0$ and $(H, s_{H}, \iota_{H})$ be a $\theta$-twisted elliptic endoscopic datum of $G$ with twisted endoscopic embedding
\(
\iota_{H}: \L{H} \rightarrow \L{G}.
\)
We fix a $\Gamma_{F}$-splitting $(\mathcal{B}_{H}, \mathcal{T}_{H}, \{\mathcal{X}_{\alpha_H}\})$ of $\D{H}$, such that $\iota_{H}(\mathcal{T}_{H}) = (\mathcal{T}_{G}^{\D{\theta}})^0$ and $\iota_{H}(\mathcal{B}_{H}) \subseteq \mathcal{B}_{G}$. Let $H = G_{I} \times G_{II}$. Via $\iota_{G}$ we embed $\D{H}$ into a $\D{\theta}_{N}$-stable nonstandard Levi subgroup $\D{L} := Z_{GL(N, \mathbb{C})}(\iota_{G}(s_{H}))$ of $GL(N, \mathbb{C})$.  
\[
\xymatrix{ \L{H} \ar[r] \ar[d]_{\iota_H} & \D{L} := GL(N_1, \mathbb{C}) \times GL(N_{2}, \mathbb{C}) \ar[d] \\
\L{G} \ar[r]^{\iota_{G} \quad \quad} & GL(N, \mathbb{C}).
}
\]
If $s_{H} \in \D{G}$, let $\mathcal{B}_{L} := \mathcal{B}_{N} \cap \D{L}$. If $s_{H} \notin \D{G}$, then we are in the case that $G$ is even orthogonal and $H = Sp(2n_1) \times Sp(2n_2)$. Since
\[
\iota_{G}(s_{H}) = \begin{pmatrix}
-I_{n_2}&&&&& \\
&I_{n_1}&&&& \\
&&0&1&&\\
&&1&0&& \\
&&&&I_{n_1}&\\
&&&&&-I_{n_2}
\end{pmatrix} \notin \mathcal{T}_{N},
\]
let $\mathcal{B}_{L} := u^{-1}\mathcal{B}_{N}u \cap \D{L}$ for 
\[
u =\begin{pmatrix}
I_{n_1+n_2}&&& \\
&1&-1&\\
&1&1& \\
&&&I_{n_1+n_2}\\
\end{pmatrix}.
\]
Then $\mathfrak{a}_{\mathcal{B}_{H}} = \mathfrak{a}_{\mathcal{B}_{G_{I}}} \oplus \mathfrak{a}_{\mathcal{B}_{G_{II}}} \subseteq \mathfrak{a}^{\D{\theta}_{N}}_{\mathcal{B}_{L}} \cong \mathbb{R}^{n_1 + n_2}$. We define orders on $\mathfrak{a}_{\mathcal{B}_{H}}$ by $\Xi_{H} \leqslant_{H} \Xi'_{H}$ (resp. $\Xi_{H} \leqslant_{L} \Xi'_{H}$) if $\Xi_{H} - \Xi'_{H} \in \overline{{}^{+}\mathfrak{a}_{\mathcal{B}_{H}}}$ (resp. $\Xi_{H} - \Xi'_{H} \in \overline{{}^{+}\mathfrak{a}_{\mathcal{B}_{L}}}$). Since $\overline{{}^{+}\mathfrak{a}_{\mathcal{B}_{H}}} \subseteq \overline{{}^{+}\mathfrak{a}_{\mathcal{B}_{L}}}$, then $\Xi_{H} \leqslant_{H} \Xi'_{H} \Rightarrow \Xi_{H} \leqslant_{L} \Xi'_{H}$. For $\Xi_{H} \in \mathfrak{a}_{\mathcal{B}_{H}} \cap \overline{(\mathfrak{a}_{\mathcal{B}_{L}})^{+}}$, there exists a unique element in $\{w(\Xi^{H}) \in \overline{(\mathfrak{a}_{\mathcal{B}_{N}})^{+}} \, | \, w \in W(\D{G}^{\Sigma_0}, \mathcal{T}_{G})\}$, denoted by $\Xi$.

\begin{lemma}
\label{lemma: combinatorics}
Suppose $\Xi_{H}, \Xi'_{H} \in \mathfrak{a}_{\mathcal{B}_{H}} \cap \overline{(\mathfrak{a}_{\mathcal{B}_{L}})^{+}}$ such that $\Xi_{H} \leqslant_{L} \Xi'_{H}$. Then $\Xi \leqslant_{N} \Xi'$. Moreover, $\Xi = \Xi' \Rightarrow \Xi_{H} = \Xi'_{H}$.
\end{lemma}

\begin{proof}
Let us write $\Xi_{H} = ({}^{I}\!\mu, {}^{II}\!\mu) \in \mathbb{R}^{n_1 + n_2}$, where ${}^{I}\!\mu = ({}^{I}\!\mu_{i})_{i = 1}^{n_{1}}$ and ${}^{II}\!\mu = ({}^{II}\!\mu_{i})_{i = 1}^{n_{2}}$ satisfying
\[
{}^{I}\!\mu_{1} \geqslant \cdots \geqslant {}^{I}\!\mu_{n_{1}} \geqslant 0, \quad {}^{II}\!\mu_{1} \geqslant \cdots \geqslant {}^{II}\!\mu_{n_{2}} \geqslant 0.
\]
Then $\Xi = (\mu_{i})_{i = 1}^{n_{1} + n_{2}}$, where $\mu_{1} \geqslant \cdots \geqslant \mu_{n_{1} + n_{2}} \geqslant 0$ and 
\[
\{\mu_{1}, \cdots, \mu_{n_1 + n_{2}}\} = \{{}^{I}\!\mu_1, \cdots, {}^{I}\!\mu_{n_{1}}, {}^{II}\!\mu_1, \cdots, {}^{II}\!\mu_{n_{2}} \}
\]
as multisets. The rest is a direct verification, so we omit it here.
\end{proof}

Let $\lif{H}$ be the corresponding $(\theta, \omega)$-twisted elliptic endoscopic group of $\lG$. For $\lq_{H} \in \Q{\lif{H}}$, let $\tilde{\lambda}_{H} = \lambda_{\lq_{H}}$. Write $\lq_{H} = {\bf p} \circ (\lq_1 \times \lq_2)$, where $\lq_1 \in \Q{\lG_{I}}$ and $\lq_{2} \in \Q{\lG_{II}}$. Let $\Xi_{[\lq_{H}]} = \Xi_{[\lq_1]} + \Xi_{[\lq_2]} \in \mathfrak{a}_{\mathcal{B}_{H}} \cap \overline{(\mathfrak{a}_{\mathcal{B}_{L}})^{+}}$. Similarly for $\lp_{H} = {\bf p} \circ (\lp_{1} \times \lp_{2}) \in \P{\lif{H}}_{\tilde{\lambda}_{H}}$, let $\Xi_{[\lp_{H}]} = \Xi_{[\lp_1]} + \Xi_{[\lp_{2}]}$.

\begin{proposition}
\label{prop: compatible with endoscopic transfer}
Suppose $\lq \in \Q{\lG}$ factors through $\lq_H \in \Q{\lif{H}}$ for an elliptic $(\theta, \omega)$-twisted endoscopic group $\lif{H}$ of $\lG$, then ${\rm Tran} \, \cPkt{\lq_{H}} = \cPkt{\lq}$.
\end{proposition}

\begin{proof}
By Lemma~\ref{lemma: stable expansion}, 
\[
\lf^{\lif{H}}_{W}(\lq_{H}) = \sum_{[\lp_{H}] \in \cP{\lif{H}}_{\tilde{\lambda}_{H}}} m([\lq_{H}], [\lp_{H}]) \lf^{\lif{H}}(\lp_{H}), \quad \lf \in \sH(\lG(F), \tilde{\chi}),
\]
where $\lf^{\lif{H}}$ is the transfer of $\lf$. Suppose $m([\lq_{H}], [\lp_{H}]) \neq 0$, then $\Xi_{[\lp_{H}]} \leqslant_{L} \Xi_{[\lq_{H}]}$ and the equality holds only when $[\lp_{H}] = [\p_{\lq_{H}}]$. Let $\lp = \iota_{\lif{H}} \circ \lp_{H}$. It follows from Lemma~\ref{lemma: combinatorics} that $\Xi_{\lp} \leqslant_{N} \Xi_{\lq}$ and the equality holds only when $\Xi_{[\lp_{H}]} = \Xi_{[\lq_{H}]}$. By the compatibility of \eqref{eq: LLC} with the twisted endoscopic transfer in the tempered case, we have
\[
\lf^{\lif{H}}(\lp_{H}) =  \lf(\lp, x) := \sum_{\tilde{\bar{\e}} \in \S{\lp}} \lf_{\lG^{\theta}}(M(\lp, \tilde{\bar{\e}}), \omega),
\]
where $x \in \S{\p}^{\theta}$ is the image of $\tilde{s}$ from the twisted endoscopic datum $(\lif{H}, \tilde{s}, \iota_{\lif{H}})$ and $M(\lp, \tilde{\bar{\e}})$ is the standard representation containing $\lr(\lp, \tilde{\bar{\e}})$ as the Langlands quotient. Suppose $[\lp_{H}] \neq [\p_{\lq_{H}}]$, then for any $\lr$ contributing to $\lf(\lp, x)$, we have
\(
\Xi_{[\lr]} \leqslant_{N} \Xi_{[\lp]} <_{N} \Xi_{[\lq]}.
\)
So $\Xi_{[\lr]} \neq \Xi_{[\lq]}$ and hence $[\lr] \notin \cPkt{\p_{\lq}}$. This shows that $\lf(\lp, x)$ does not have contribution from $\cPkt{\p_{\lq}}$. On the other hand, 
$\lf^{\lif{H}}(\p_{\lq_{H}}) = \lf(\p_{\lq}, x)$. So $\lf^{\lif{H}}_{W}(\lq_{H})$ has contribution from $\cPkt{\p_{\lq}}$. In view of Corollary~\ref{cor: character relation general}, we have ${\rm Tran} \, \cPkt{\lq_{H}} = \cPkt{\lq}$.
\end{proof}

\begin{corollary}
The association $[\lq] \mapsto \cPkt{\lq}$ is compatible with parabolic induction and twisted endoscopic transfer.
\end{corollary}

\begin{proof}
We first show the compatibility with parabolic induction. Let $\lif{M} \cong \prod_{i = 1}^{l} GL(n_i) \times \lG_{-}$ be a standard Levi subgroup of $\lG$ and $\lq_{M} = \prod_{i}\q_{i} \times \lq_{-} \in \Q{\lif{M}}$. Let $\lq = \iota_{\lif{M}} \circ \lq_{M}$. It factors through $\lq_{M'} = \lq' \times \lq_{-} \in \Q{\lif{M}'}$ for $\lif{M}' = GL(\sum_{i = 1}^{l} n_i) \times \lG_{-}$ and $\q' = \oplus_{i = 1}^{l} \q_{i}$. Since $\times_{i = 1}^{l} \, \r_{\q_i} = \r_{\q'}$, it follows from Proposition~\ref{prop: compatible with parabolic induction} that
\[
(\times_{i = 1}^{l} \, \r_{\q_i}) \rtimes \cPkt{\lq_{-}} = \r_{\q'} \rtimes \cPkt{\lq_{-}} = \cPkt{\lq}.
\]

Next we show the compatibility with twisted endoscopic transfer. Let $\lif{H}$ be a $(\theta, \omega)$-twisted endoscopic group of $\lG$ and $\lq_{H} \in \Q{\lif{H}}$. We can view $\lif{H}$ as the Levi subgroup of some twisted elliptic endoscopic group $\lif{H}'$ of $\lG$. 
\[
\xymatrix{\L{\lif{H}'} \ar[r]^{\iota_{\lif{H}'}} & \L{\lG} \\
\L{\lif{H}} \ar[u]^{\iota^{\lif{H}'}_{\lif{H}}} \ar[ur]_{\iota_{\lif{H}}} &  
}
\]
Let $\lq_{H'} = \iota^{\lif{H}'}_{\lif{H}} \circ \lq_{H}$ and $\lq = \iota_{\lif{H}} \circ \lq_{H}$. Then
\[
{\rm Tran}_{\iota_{\lif{H}}} \, \cPkt{\lq_{H}} = {\rm Tran}_{\iota_{\lif{H}'}} \, {\rm Tran}_{\iota_{\lif{H}}^{\lif{H}'}} \, \cPkt{\lq_{H}}.
\]
where ${\rm Tran}_{\iota_{\lif{H}}^{\lif{H}'}}$ is given by the parabolic induction. By Proposition~\ref{prop: compatible with parabolic induction},
\[
{\rm Tran}_{\iota_{\lif{H}}^{\lif{H}'}} \, \cPkt{\lq_{H}} = \cPkt{\lq_{H'}}.
\]
By Proposition~\ref{prop: compatible with endoscopic transfer}, 
\[
{\rm Tran}_{\iota_{\lif{H}'}} \, \cPkt{\lq_{H'}} = \cPkt{\lq}.
\]
Hence ${\rm Tran}_{\iota_{\lif{H}}} \, \cPkt{\lq_{H}} = \cPkt{\lq}$.
\end{proof}

\bibliographystyle{amsalpha}

\bibliography{reps}

\end{document}